\RequirePackage[final]
{showkeys}

\documentclass[EJP,preprint]{ejpecp} % replace ECP by EJP if needed.
\usepackage{graphicx}
\usepackage{placeins}

\usepackage%[final]
{showkeys}

\usepackage{enumerate}

\usepackage{xurl}
\SHORTTITLE{Bounds on isotropic log-concave distributions}

\TITLE{Exact bounds on the distribution function of isotropic log-concave distributions}
\AUTHORS{%
  Iosif~Pinelis\footnote{Michigan Technological University, United States of America.
    \EMAIL{ipinelis@mtu.edu}}%\orcid{0000-0003-0986-3622}
}

 % Separate items with ;

\AMSSUBJ{60E15, 26D10, 26D15, 26A51} % Edit. Separate items with ;
\SUBMITTED{Month Day, YYYY } % Edit.
\ACCEPTED{} % Edit.

\VOLUME{0}
\YEAR{2026}
\PAPERNUM{0}
\DOI{%10.1214/YY-TN
}

\ABSTRACT{For each real $b$, exact upper and lower bounds on the probability $\P(X\ge b)$ over all random variables $X$ with log-concave p.d.f.'s such that $\E X=0$ and $\E X^2=1$ are obtained, as well as the best constant factor $C$ in the inequality $\P(X\ge b)\le C e^{-b}$ for all real $b\ge0$. Exact upper and lower bounds on the corresponding exponential moments (Proposition~\ref{prop:exp}) and explicit exponentially decreasing upper bounds on the mentioned p.d.f.'s (Propositions~\ref{prop:devr} and \ref{prop:iso}) are given as well. 
Some general results concerning log-concave p.d.f.'s are also obtained.}

\newtheorem*{remark*}%[theorem]
{REMARK}

\newcommand{\dd}{\partial}
\renewcommand{\u}{\nearrow}
\renewcommand{\d}{\searrow}

\newcommand{\al}{\alpha}

\newcommand{\si}{\sigma}

\newcommand{\ka}{\kappa}
\newcommand{\la}{\lambda}
\newcommand{\ga}{\gamma}
\newcommand{\de}{\delta}
\newcommand{\be}{\beta}
\newcommand{\De}{\Delta}
\newcommand{\ep}{\varepsilon}

\newcommand{\thh}{\theta}

\newcommand{\tf}{\tilde f}
\newcommand{\tc}{\tilde c}

\newcommand{\q}{\mathrm{q}}
\newcommand{\qc}{\mathrm{qc}}

\newcommand{\ap}{\approx}
\newcommand{\ini}{\mathrm{ini}}

\newcommand{\p}{\mathrm{p}}
\renewcommand{\r}{\mathrm{r}}
\newcommand{\s}{\mathrm{s}}
\newcommand{\FF}{\mathrm{F}}
\newcommand{\EE}{\mathrm{E}}
\newcommand{\rat}{\mathrm{rat}}
\newcommand{\RAT}{\mathrm{RAT}}

\newcommand{\tr}{\tilde r}
\newcommand{\cri}{\mathrm{cr}} 
\newcommand{\eg}{\mathrm{eg}}
\newcommand{\tb}{\tilde b}

\renewcommand{\P}{\operatorname{\mathsf{P}}} 
\newcommand{\PP}{\mathcal{P}}
\newcommand{\E}{\operatorname{\mathsf{E}}}
\newcommand{\Var}{\operatorname{\mathsf{Var}}}

\newcommand{\Z}{\mathbb{Z}}
\newcommand{\R}{\mathbb{R}}

\newcommand{\F}{\mathcal{F}}

\renewcommand{\L}{\mathcal{L}}

\newcommand{\A}{\mathbf{A}}
\newcommand{\B}{\mathbf{B}}
\renewcommand{\c}{\mathbf{c}}
\newcommand{\h}{\mathbf{h}}

\newcommand{\oR}{\overline\R}

\newcommand{\sign}{\operatorname{sign}}

\newcommand{\wt}{\widetilde}

\renewcommand{\le}{\leqslant}
\renewcommand{\ge}{\geqslant}

\newcommand{\conv}{\operatorname{conv}}
\newcommand{\ext}{\operatorname{ext}}

\newcommand{\I}{\mathcal I}

\makeatletter
\newcommand{\pushright}[1]{\ifmeasuring@#1\else\omit\hfill$\displaystyle{#1}$\fi\ignorespaces}
\newcommand{\pushleft}[1]{\ifmeasuring@#1\else\omit$\displaystyle{#1}$\hfill\fi\ignorespaces}
\makeatother

\begin{document}

%%%%%%%%%%%%%%%%%%%%%%%%%%%%%%%%%%%%%%%%%%%%%%%%%%%%%%%%%%%%%%%%%%%
%%                                                               %%
%% No need for \maketitle.                                       %%
%%                                                               %%
%%%%%%%%%%%%%%%%%%%%%%%%%%%%%%%%%%%%%%%%%%%%%%%%%%%%%%%%%%%%%%%%%%%

%%%%%%%%%%%%%%%%%%%%%%%%%%%%%%%%%%%%%%%%%%%%%%%%%%%%%%%%%%%%%%%%%%%
%%                                                               %%
%% Please replace what follows by the body of your article       %%
%% (up to the bibliography):                                     %%
%%                                                               %%
%%%%%%%%%%%%%%%%%%%%%%%%%%%%%%%%%%%%%%%%%%%%%%%%%%%%%%%%%%%%%%%%%%%

\tableofcontents

%%%%% END OF TITLE PAGE %%%%%%%%%%%%%%%%%%%%%%%%%%%%%%%%%%%%%%%%%%%%%%

%%%%% BODY OF THE PAPER %%%%%%%%%%%%%%%%%%%%%%%%%%%%%%%%%%%%%%%%%%%%%%
% You should eventually delete (after reading) the rest of the text below %% 

\section{Preliminaries and generalities}\label{prelims}

Let $\PP$ denote the set of all Borel probability measures over $\R$. Let $\I$ stand for the set of all \emph{isotropic} measures in $\PP$: 
\begin{equation}\label{eq:iso}
	\I:=\Big\{\mu\in\PP\colon\int_\R x\,\mu(dx)=0,\ \int_\R x^2
\,\mu(dx)=1\Big\}.
\end{equation}

Let $\L$ denote the set of all isotropic log-concave measures in $\PP$. A measure $\mu\in\PP$ is called \emph{log concave} (cf.~\cite{prekopa71,borell74}) if 
$\mu_*((1-t)A+tB)\ge\mu(A)^{1-t}\mu(B)^t$ for all Borel subsets $A$ and $B$ of $\R$ and all $t\in(0,1)$, where $\mu_*$ is the inner measure corresponding to the measure $\mu$.

According to Theorem~1.1 by Borell \cite{borell74}, a measure in $\I$ is log concave iff it has a log-concave density (with respect to the Lebesgue measure). 

Recall that a function $f\colon\R\to\R$ is \emph{log concave} if $f\ge0$ and  
\begin{equation}\label{eq:log-concave}
	f(x_t)\ge f(x_0)^{1-t}f(x_1)^t
\end{equation}
for all real $x_0$ and $x_1$ and all $t\in(0,1)$, where  
\begin{equation}\label{eq:x_t}
	x_t:=(1-t)x_0+tx_1
\end{equation}
for real $t$. 

%A broad survey on log-concavity was given by Saumard and Wellner \cite{wellner-logconcav}. 

Let us say that a function $f\colon\R\to\R$ is \emph{two-piece exponential} if 
\begin{equation}\label{eq:2-piece}
	f(x)=e^{l_1(x)\wedge l_2(x)}\,1(-A<x<B)
\end{equation}
for real $x$, where $l_1$ and $l_2$ are affine functions, $u\wedge v:=\min(u,v)$ for real $u$ and $v$, and $-\infty\le-A<B\le\infty$. If \eqref{eq:2-piece} holds with $l_1=l_2$ for all real $x$, let us say that the two-piece exponential function $f$ is \emph{one-piece exponential}. If \eqref{eq:2-piece} holds with $A=B=\infty$ for all real $x$, let us say that the two-piece exponential function $f$ is \emph{non-truncated two-piece exponential}. Let $\F$ stand for the set of the lower semi-continuous (l.s.c.) densities of the measures in $\L$. Let 
$\F_2,\F_1,\F_{2,\infty}$ denote the sets of functions in $\F$ that are two-piece exponential, one-piece exponential, and non-truncated two-piece exponential, respectively. 
Let $\L_2,\L_1,\L_{2,\infty}$ denote the set of all measures in $\L$ with densities in $\F_2,\F_1,\F_{2,\infty}$, respectively. 

The p.d.f.'s $f\in\F_2$ can be parametrized by the formula 
\begin{equation}\label{eq:f_{A,B,u,c,h,k}}
	f(x)=f_{A,B,u,c,h,k}(x):=c e^{h(x-u)}1(-A<x\le u)+c e^{-k(x-u)}1(u<x<B)
\end{equation}
for all real $x$, where 
\begin{equation}\label{eq:A,B,u,c,h,k}
\begin{gathered}
	\text{$-\infty\le-A\le u\le B\le\infty$, $-A<B$, $u\in\R$, $c\in(0,\infty)$,} \\ 
	\text{and $h$ and $k$ are real numbers such that $h\ge-k$.}
\end{gathered}	
\end{equation}
One may note that here necessarily $A>0$ and $B>0$, because otherwise the barycenter of the probability distribution $\mu$ with density $f$ could not be $0$, violating the condition $\mu\in\L_2\subseteq\L$. 
Also, necessarily $h>0$ if $A=\infty$ and $k>0$ if $B=\infty$. In addition, the constant factor $c$ is uniquely determined by $A,B,u,h,k$, in view of the condition $\int_\R f_{A,B,u,c,h,k}=1$. 

\begin{theorem}\label{th:2-piece}
Let $H\colon\R\to\R$ be a %Borel-measurable 
l.s.c.\ locally integrable function such that for some $s\in(0,1)$ 
\begin{equation}\label{eq:H le}
	|H(x)|=O(e^{s |x|})\text{ as }|x|\to\infty. 
\end{equation}
%
%satisfying one of the following two conditions:
%\begin{enumerate}[(i)]
%	\item $|H(x)|=o(x^2)$ as $|x|\to\infty$ and the function 
%$$\R\ni x\mapsto H_{\la_1,\la_2,\la_3}(x):=H(x)-\la_1-\la_2 x-\la_3 x^2$$ 
%has at most $4$ zeros for any real $\la_1,\la_2,\la_3$. 
%\item $H(x)=1(0<x<c)$ for some real $c>0$ and all real $x$.
%\end{enumerate}
Then $H$ is $\mu$-integrable for all $\mu\in \L$ and 
\begin{equation}\label{eq:Th1.1}
	\inf_{\mu\in \L}\int_\R H\,d\mu=\inf_{\mu\in \L_2}\int_\R H\,d\mu; 
\end{equation}
moreover, these infima are attained. 
\end{theorem}

Theorem~\ref{th:2-piece} is a special case of Theorem~\ref{rem:H} below. 
Indeed, condition \eqref{eq:H le} in Theorem~\ref{th:2-piece} can be relaxed. To state the relaxed condition, we need the following definition. 

\begin{definition}\label{def:V}
Let $V$ denote the set of all functions $v\colon(0,\infty)\to\R$ such that, for 
some real $y>0$, % and some real $\al<1$, 
on the interval $[y,\infty)$ the function $v$ is (strictly) positive, nonincreasing, integrable, and differentiable, with $v'\ge-v$.  
\end{definition}
 
%\begin{enumerate}[(i)]
%	\item  on $(y,\infty)$, the function $v$ is nonnegative and differentiable, with $v'\ge-\al v$; 
%	\item $\int_y^\infty v<\infty$. 
%\end{enumerate}

For instance, if $v(x)=1/x^p$ for some real $p>1$ and all real $x>0$, then $v\in V$. 

\begin{theorem}\label{rem:H}
The conclusion in Theorem~\ref{th:2-piece} will hold if condition \eqref{eq:H le} is replaced by the following, more general condition: for some $v\in V$,  
\begin{equation}\label{eq:H le,v}
	|H(x)|=o\big(e^{|x|}v(|x|)\big)\text{ as }|x|\to\infty.  
\end{equation}
%
%the 
%expression $e^{s |x|}$ in \eqref{eq:H le} is replaced by $e^{|x|}v(|x|)$.  
\end{theorem}

That \eqref{eq:H le} (for some $s\in(0,1)$) implies \eqref{eq:H le,v} (for some $v\in V$) is witnessed by letting $v(x):=e^{-t x}$ for any $t\in(0,1-s)$ and all real $x>0$. So, Theorem~\ref{rem:H} implies Theorem~\ref{th:2-piece}. 

Theorem~\ref{rem:H} will be proved in Section~\ref{2-piece proof}. 

\begin{remark*}
Necessary details of symbolic and exact numerical calculations in some of the proofs in this paper will be given in Mathematica notebooks. 
The Mathematica files cited in this paper are provided in Supplement~B, described on page~\pageref{mathematica}. 
%in Supplement~B described on page~\pageref{mathematica}. 
%Any Mathematica notebook referred to in any of those proofs will be in the folder mentioned in the description of Supplement~B on page~\pageref{mathematica}.  
\end{remark*}

%\begin{supplement}
%\stitle{%Title of Supplement A.
%Supplement B: Mathematica Files}\label{mathematica}
%\sdescription{%Short description of Supplement B.
%Zip file MathematicaFiles.zip containing  \break 
%Mathematica notebooks and Mathematica package \texttt{hammer.wl} providing details of symbolic and exact numerical calculations in proofs in the paper. Before using these Mathematica notebooks, extract all the files in this zip file into a folder/directory of your own. 
%}
%\end{supplement}

\begin{theorem}\label{th:4 zeros}
Suppose that the conditions of Theorem~\ref{rem:H} hold. Suppose also that the function $H$ is continuous. Suppose finally that the function $H$ is \emph{$4$-good} in the sense that, for any real $\la_1,\la_2,\la_3$, the function 
\begin{equation*}
\R\ni x\mapsto R(x):=R_{\la_1,\la_2,\la_3}(x):=H(x)-\la_1-\la_2 x-\la_3 x^2
\end{equation*}
has at most four zeros.
Then the infima in \eqref{eq:Th1.1} are attained at a measure $\mu\in \L_1\cup\L_{2,\infty}$. 
% with a density $f$ of the form \eqref{eq:2-piece} where either $l_1=l_2$ or $A=B=\infty$. 
\end{theorem}

Theorem~\ref{th:4 zeros} will be proved in Section~\ref{4 zeros proof}.

%!!!!???? References for this?

\begin{remark}\label{rem:4 zeros}
Let $H\colon\R\to\R$ be a twice continuously differentiable function such that, for some real $x_1$ and $x_2$ such that $x_1<x_2$, $H''$ (strictly) increases from $0$ on the interval $(-\infty,x_1]$, (strictly) decreases on the interval $[x_1,x_2]$, and increases back to $0$ on the interval $[x_2,\infty)$. Then $H$ is \emph{$4$-good}.  

Indeed, then $R''=H''-2\la_3$ has at most two zeros and hence $R$ has at most $4$ zeros. 

In particular, the c.d.f.\ of any normal distribution and the c.d.f.\ of any Cauchy distribution are $4$-good. 

Note also that, if $H$ is $4$-good, then $\ka H$ is also $4$-good for any nonzero real $\ka$. \qed
\end{remark}

\begin{corollary}\label{cor:b}
For any real $b$ and any real $\ka$, if $H=\ka 1_{[b,\infty)}$, 
then the infima in \eqref{eq:Th1.1} are attained at a measure $\mu\in \L_1\cup\L_{2,\infty}$. 
% with a density $f$ of the form \eqref{eq:2-piece} where either $l_1=l_2$ or $A=B=\infty$. 
\end{corollary}

Corollary~\ref{cor:b} will be proved in Section~\ref{4 zeros proof}.

\medskip

Any $f\in\F_1$ has representation \eqref{eq:f_{A,B,u,c,h,k}} with, say, $u=0$ (recall that $A>0$ and $B>0$) and $k=-h$, which leaves four parameters: $A,B,c,h$. However, these four parameters are tied by the three restrictions -- on the zeroth, first, and second moments of $f$, as $f$ is a density of a measure in $\L_1$ and thus in $\L_2$. In the final analysis, this leaves $f$ with just one ``degree of freedom''. And indeed, as stated in Theorem~\ref{th:par1} below, the functions in $\F_1$ can be explicitly parametrized by one (extended) real parameter.  

To state Theorem~\ref{th:par1}, we will need the following notation. 
For real $r\ne0$, let % SomeDefs.nb, parametrization1.nb 
\begin{equation}\label{eq:A,B,c,h(r)}
	\A(r):=\frac{\si(r)}{\sqrt{u(r)}},\quad \B(r):=\frac{s(r)}{\sqrt{u(r)}},
	\quad \c(r):=\frac{\sqrt{u(r)}}{(e^r-1)^2}\exp\frac{s(r)}{e^r-1},
	\quad \h(r):=-\frac{\sqrt{u(r)}}{e^r-1},
\end{equation}
where
\begin{equation}\label{eq:u,s,si}
	u(r):=e^{2 r}-e^r (r^2+2)+1,\quad s(r):=e^r (r-1)+1,\quad \si(r):=e^r-1-r
\end{equation}
for all real $r$. As will be shown in Remark~\ref{rem:e-poly}, $u(r)>0$ for all real $r\ne0$ and $u(0)=0$.  
Extend the functions $\A,\B,\c,\h$ to $\oR:=[-\infty,\infty]$ by continuity, so that % checking\\r=0 or pm infty.nb 
\begin{equation}\label{r=0 or pm infty}
\begin{gathered}
	\A(-\infty):=\infty,\quad \A(0):=\sqrt3,\quad \A(\infty):=1,\\  
	\B(-\infty):=1,\quad \B(0):=\sqrt3,\quad \B(\infty):=\infty,\\  
	\c(-\infty):=\frac1e,\quad \c(0):=\frac1{2\sqrt3},\quad \c(\infty):=\frac1e,\\  
	\h(-\infty):=1,\quad \h(0):=0,\quad \h(\infty):=-1.  
\end{gathered}	
\end{equation}
These limits are verified in Mathematica notebook GeneralDefinitions.nb. 

\begin{theorem}\label{th:par1}
Take any $f\in\F_1$, so that 
\begin{equation}\label{eq:f_{A,B,c,h}}
	f(x)=f_{A,B,c,h}(x):=c e^{hx}1(-A<x<B)
\end{equation}
for all real $x$, where 
\begin{equation}\label{eq:A,B,c,h}
\begin{gathered}
	\text{$-\infty\le-A<B\le\infty$, $c\in(0,\infty)$, $h\in\R$.} 
\end{gathered}	
\end{equation} 
Then for a unique $r\in\oR$ we have 
\begin{equation}\label{eq:par1}
	(A,B,c,h)=(\A(r),\B(r),\c(r),\h(r)).
\end{equation}
Vice versa, if \eqref{eq:par1} holds for some $r\in\oR$, then $f_{A,B,c,h}\in\F_1$. 
\end{theorem}

Theorem~\ref{th:par1} will be proved in Section~\ref{param-proofs}.

The set $\F_{2,\infty}$ also admits an explicit one-parameter parametrization, which in this latter case is much easier to obtain.  

\begin{theorem}\label{th:par2}
Take any $f\in\F_{2,\infty}$, so that 
\begin{equation}\label{eq:f_{c,u,h,k}}
	f(x)=\tf_{c,u,h,k}(x):=c e^{h(x-u)}1(-\infty<x\le u)+c e^{-k(x-u)}1(u<x<\infty)
\end{equation}
for all real $x$, where 
\begin{equation}\label{eq:c,u,h,k}
\begin{gathered}
	\text{$\{c,h,k\}\subset(0,\infty)$, $u\in\R$.} 
\end{gathered}	
\end{equation} 
Then $k>1$, $h>1$,   
\begin{equation}\label{eq:par2}
c=\frac{hk}{h+k},\quad u=\frac1h-\frac1k,\quad \frac1{h^2}+\frac1{k^2}=1,
\end{equation}
so that the functions $f=\tf_{c,u,h,k}\in\F_{2,\infty}$ can be uniquely parametrized by $h\in(1,\infty)$ or by $k\in(1,\infty)$.
Vice versa, if \eqref{eq:par2} holds, then $\tf_{c,u,h,k}\in\F_{2,\infty}$. 
\end{theorem}

Theorem~\ref{th:par2}, too, will be proved in Section~\ref{param-proofs}. 

\section{Statements of the main results}\label{main}

Before stating Theorem~\ref{th:df bounds} below, let us introduce additional definitions and state auxiliary results. 

For a function $R$, let us say that $R$ is $\u$ (respectively, $R$ is $\d$) on an interval if $R$ is increasing (respectively, decreasing) on the interval. Let us say that $R$ is $\u\d$ on an interval $I\subseteq\oR$ if for some point $y$ in the interior of $I$ the function $R$ is increasing on $I\cap[-\infty,y]$ and decreasing on $I\cap[y,\infty]$. Symbols $\d\u$, $\d\u\d$, etc. will be used in a similar way. Increase and decrease are always understood in the strict sense in this paper. 
The strings such as $\d\u\d$ will be referred to as the monotonicity patterns. 

Similarly defined will be sign patterns; for instance, we will say that $R$ is $+-$ on an interval $I\subseteq\oR$ if, for some point $y$ in the interior of $I$, the function $R$ is $>0$ on $I\cap[-\infty,y)$ and $<0$ on $I\cap(y,\infty]$. The strings such as $+-$ will be referred to as the sign patterns. 

\begin{lemma}\label{lem:B}
The function $\B$ is continuously increasing on $\oR$ from $\B(-\infty)=1$ to $\B(\infty)=\infty$. Moreover, $\B'>0$ on $\R$. 
\end{lemma}

Lemma~\ref{lem:B} will be proved in Section~\ref{main thms proofs}. 

%\begin{remark*}
%\pageref{mathematica}
%\end{remark*}

\begin{remark}\label{rem:iB}
In view of Lemma~\ref{lem:B}, the function $\R\ni r\mapsto\B(r)\in(1,\infty)$ has a unique continuously increasing inverse, which we will denote by $\B^{-1}$. For real $b>0$, let then 
\begin{equation}\label{eq:r_b}
	r_b:=\begin{cases}
	\B^{-1}(b) &\text{ if } b>1,\\ 
	-\infty &\text{ if } b\le1. 
	\end{cases}
\end{equation}
Then the function $(0,\infty)\ni b\mapsto r_b\in\oR$ is nondecreasing on $(0,\infty)$ and continuously increasing on $[1,\infty)$ from $r_1=-\infty$ to $r_{\infty-}=\infty$. 
\end{remark}

For real $b>0$ and $r\in\oR$, let 
\begin{equation}\label{eq:p_{1,b}}
	p_1(r):=p_{1,b}(r):=\int_b^\infty f_{\A(r),\B(r),\c(r),\h(r)}(x)\,dx. 
\end{equation}
Note that 
\begin{equation}\label{eq:p_{1,b}=}
	p_1(r)=p_{1,b}(r)=
	\frac1{e^r-1}\,\Big(-1+\exp\frac{\big(s(r)-b\sqrt{u(r)}\big)_+}{e^r-1}\Big)
\end{equation}
for real $r\ne0$, with $u(r)$ and $s(r)$ as defined in \eqref{eq:u,s,si}; as usual, $z_+:=\max(0,z)$ for real $z$.  
The values of $p_{1,b}$ for $r\in\{-\infty,0,\infty\}$ can be obtained from \eqref{eq:p_{1,b}=} by continuity, % checking\\r=0 or pm infty.nb 
so that for all real $b>0$  
\begin{equation}\label{eq:p_{1,b} special}
	p_{1,b}(-\infty)=1-e^{-(1-b)_+}, \quad 
	p_{1,b}(0)=\frac{(\sqrt3-b)_+}{2\sqrt3}, \quad 
	p_{1,b}(\infty)=e^{-1-b};
\end{equation}
to verify the second equality in \eqref{eq:p_{1,b} special}, one may note that $s(r)\sim r^2/2$ and $u(r)\sim r^4/12$ as $r\to0$. 

\begin{proposition}\label{prop:F1}
The function $p_1$ %, defined by \eqref{eq:p_{1,b}}, 
is $\u\d$ on $[r_b,\infty]$. 
Also, $p_1(r)=0$ if $-\infty\le r<r_b$. 
So, $\sup_{r\in\oR}p_1(r)$ is attained at a unique value of $r$, which is actually in $\R$, and  
this unique maximizer of $p_1(r)$ in $r$ can be effectively computed to any precision. 
\end{proposition}

This proposition will be proved in Section~\ref{main thms proofs}.

\begin{lemma}\label{lem:b_{**}}
For real $b\ge1/\sqrt2$, let 
\begin{equation}\label{eq:rat}
	\rat(b):=\frac{\frac{1-\sqrt{2 b^2-1}}{2} e^{\sqrt{2 b^2-1}-1}}{e^{-1-b}}. 
\end{equation}
Then the ratio function $\rat$ is continuously decreasing on the interval $[1/\sqrt2,1)$, from $e^{1/\sqrt2}/2=1.0140\ldots$ to $0$. Therefore, there is a unique root $b_{**}\in[1/\sqrt2,1)$ of the equation $\rat(b)=1$, and 
\begin{equation}\label{eq:b_**}
	b_{**}=0.726749\ldots. 
\end{equation}
\end{lemma}

Lemma~\ref{lem:b_{**}} will be proved in Section~\ref{main thms proofs}.

%\begin{proof}[Proof of Lemma~\ref{lem:b_{**}}]
%Note that, for $b\in[1/\sqrt2,1)$,   
%\begin{equation}
%	\rat(b)=\frac{1-\sqrt{2 b^2-1}}{2}\, e^{b+\sqrt{2 b^2-1}}
%\end{equation}
%and 
%\begin{gather*}
%\rat'(b)=\frac{1-2b-\sqrt{2 b^2-1}}{2} \, e^{b+\sqrt{2 b^2-1}}<0, 
%\end{gather*}
%so that the continuous function $\rat$ is decreasing on $[1/\sqrt2,1)$, from $e^{1/\sqrt2}/2=1.0140\ldots$ to $0$. 
%\end{proof}

Let now 
\begin{equation}\label{eq:J_2}
	J_2(b):=\begin{cases}
	e^{-1-b}&\text{ if }0<b\le b_{**}, \\ 
	\frac{1-\sqrt{2 b^2-1}}{2} e^{\sqrt{2 b^2-1}-1} &\text{ if }b_{**}\le b<1, \\ 
	0&\text{ if }b\ge1,
	\end{cases}
\end{equation}
with $b_{**}$ as defined in Lemma~\ref{lem:b_{**}}.  
%\begin{equation}\label{eq:s_3}
%	s_3(b):=\frac{1+\sqrt{2 b^2-1}}{2 b} 
%\end{equation} 
%and 
%\begin{equation}\label{eq:b_**}
%	b_{**}=0.726749\ldots
%\end{equation} 
%is the unique root $b$ in the interval $[1/\sqrt2,1)$ of the equation 
%\begin{equation}\label{eq:rat}
%	\rat(b):=\frac{\frac{1-\sqrt{2 b^2-1}}{2} e^{\sqrt{2 b^2-1}-1}}{e^{-1-b}}=1;
%\end{equation}
%the ratio function $\rat$ is continuously decreasing on the interval $[1/\sqrt2,1)$, from $e^{1/\sqrt2}/2=1.0140\ldots$ to $0$. 

\begin{theorem}\label{th:df bounds}
For real $b$, let 
\begin{equation}\label{eq:sups,infs}
\begin{aligned}
	S(b):=\sup_{\mu\in\L}\mu([b,\infty))=\sup_{f\in\F}\int_b^\infty f, \\ 
	I(b):=\inf_{\mu\in\L}\mu([b,\infty))=\inf_{f\in\F}\int_b^\infty f.  
\end{aligned}	
\end{equation}
The suprema and infima in \eqref{eq:sups,infs} are all attained. Moreover, 
\begin{equation}\label{eq:S(b)=}
	S(b)=\left\{
\begin{alignedat}{2}
	&\max_{r\in\oR}p_{1,b}(r) 
	=\max_{r\ge r_b}p_{1,b}(r)=\max_{r>r_b}p_{1,b}(r) 
	&&\quad\text{ if }b>0, \\
	&1-\frac 1e &&\quad\text{ if }b=0, \\ 
	&1-J_2(-b) &&\quad\text{ if }b<0  
\end{alignedat}	
	\right.
\end{equation}
and, for all real $b$, 
\begin{equation}\label{eq:I(b)=}
	I(b)=1-S(-b),
\end{equation}
where $p_{1,b}(r)$ is as %defined in \eqref{eq:p_{1,b}} 
expressed in \eqref{eq:p_{1,b}=} and \eqref{eq:p_{1,b} special} 
and $J_2(b)$ is as defined in \eqref{eq:J_2}. 
We also have % these particular values/asymptotic: 
\begin{equation}\label{eq:S(0),S(1)}
%S(0)=p_{1,0}(-\infty)=1-\frac 1e,\quad	
S(1)=p_{1,1}(0)=\frac12-\frac1{2\sqrt3} 
\end{equation}
and % end of \\exp-bound\\2.nb
\begin{equation}\label{eq:S(b) sim}
	S(b)\underset{b\to\infty}\sim e^{-1-b}.
\end{equation}
A maximizer for $S(0)$
%case of $p_{1,b}(r)$ with $b=0$ and $r=-\infty$ in \eqref{eq:S(b)=} corresponds to 
is the ``negative'' exponential p.d.f.\ given by the formula 
\begin{equation}\label{eq:f_-infty}
f(x)=f_{\A(-\infty),\B(-\infty),\c(-\infty),\h(-\infty)}(x)
=e^{x-1}\,1(-\infty<x<1)	
\end{equation}
for real $x$ (recall here \eqref{eq:f_{A,B,c,h}} and \eqref{r=0 or pm infty}). 

The case of $p_{1,b}(r)$ with $b=1$ and $r=0$ in \eqref{eq:S(0),S(1)} corresponds to the uniform distribution over the interval $(-\sqrt3,\sqrt3)$. 
\end{theorem}

Theorem~\ref{th:df bounds} will be proved in Section~\ref{main thms proofs}.

\begin{remark}\label{rem:df bounds}
In view of \eqref{eq:S(b)=}, % and \eqref{eq:I_2}, 
we have explicit expressions for $S(b)$ with $b\le0$ in terms of elementary functions of $b$. 
In view of \eqref{eq:p_{1,b}=} and \eqref{eq:p_{1,b} special},
%\eqref{eq:S=S_1}, the definition of $S_1(b)$ in \eqref{eq:S_2 le S_1},  formulas \eqref{eq:p_1=} and \eqref{eq:F,g,G}, 
and because $p_{1,b}(r)=0$ if $b>1$ and $r\le r_b$, whereas $r_b=-\infty$ if $0<b\le1$, Proposition~\ref{prop:F1} %provides for an effective calculation of $S(b)$ for $b>0$ with an arbitrary precision
permits $S(b)$ to be computed to arbitrary precision for $b>0$ as well; see further details in Section~\ref{calculations}.  
So, in view of \eqref{eq:I(b)=}, we can effectively compute $S(b)$ and $I(b)$ for all real $b$. 
In view of \eqref{eq:S(0),S(1)} and \eqref{eq:S(b) sim}, we also have simple explicit expressions for $S(1)$ and hence for $I(-1)=1-S(1)$, as well as simple explicit asymptotics for $S(b)$ as $b\to\infty$ and hence for $1-I(b)=S(-b)$ as $b\to-\infty$. 
Moreover, Lemma~5.5(a) by Lovasz and Vempala \cite{lovasz-vempala} states that 
\begin{equation}\label{eq:f le1}
	\text{if $f\in\F$ then $f\le1$.} 
\end{equation}
So, the functions $S$ and $I$ are $1$-Lipschitz.
\qed
\end{remark}

Whereas, according to Remark~\ref{rem:df bounds}, $S(b)$ and $I(b)=1-S(-b)$ can be effectively computed for all real $b$, the expressions for $S(b)$ provided by %formula \eqref{eq:S(b)=} in 
Theorem~\ref{th:df bounds} are not quite explicit for $b\in(0,\infty)\setminus\{1\}$. So, the following simple exponential upper bound on $S(b)$ with the exact constant factor should be useful. 

\begin{theorem}\label{th:upper bound}
\begin{equation}\label{eq:C_*:=}
C_{\sup}:=\sup_{b\ge0}\frac{S(b)}{e^{-b}}=0.681100549\ldots<\infty,	
\end{equation}
so that $C_{\sup}$ is the best constant factor $C$ in the inequality 
\begin{equation}\label{eq:upper bound}
	S(b)\le C e^{-b}
\end{equation}
for all real $b\ge0$. The supremum in \eqref{eq:C_*:=} is attained at a certain real number $b_{\sup}=0.336266195\ldots$. Both $b_{\sup}$ and $C_{\sup}$ can be effectively computed to any precision. 
\end{theorem}

One may note here that, in view of \eqref{eq:S(b) sim}, the coefficient $-1$ of $b$ in the exponent $-b$ in \eqref{eq:upper bound} cannot be replaced by a better (that is, smaller) one; that is, $-1$ cannot be replaced by $-(1+\ep)$, for any real $\ep>0$.  

%!!!! compare with \eqref{eq:exp-decay} 

Theorem~\ref{th:upper bound} will be proved in Section~\ref{main thms proofs}. 

%???? graphs of $S(b),I(b),Ce^{-b}$

For an illustration of Theorems~\ref{th:df bounds} and \ref{th:upper bound}, Figure~\ref{fig:pic} shows the graphs $\{(b,I(b))\colon -5<b<5\}$ (red), $\{(b,S(b))\colon -5<b<5\}$ (blue), and $\{(b,C_{\sup}e^{-b})\colon 0<b<5\}$ (gray); recall here the definitions of $S(b)$ and $I(b)$ in \eqref{eq:sups,infs} and the definition of $C_{\sup}$ in \eqref{eq:C_*:=}.  % \\checking\\plot.nb
\begin{figure}[htbp]
	\centering
		\includegraphics[width=.7\textwidth]{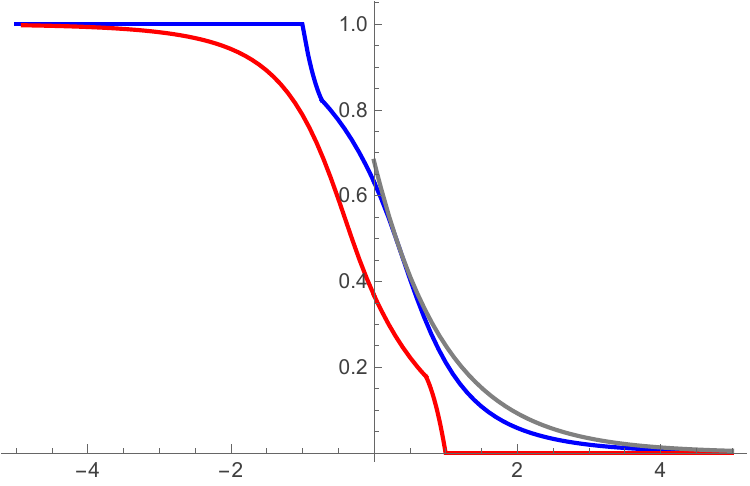}
	\caption{Graphs of the exact pointwise lower and upper bounds in Theorem~\ref{th:df bounds} and of the best exponential upper bound in Theorem~\ref{th:upper bound}: $I(b)$ (red), $S(b)$ (blue), $C_{\sup}e^{-b}$ (gray).} 
%	
%	
%	on the tail probability for isotropic log-concave distributions, and of the exact exponential upper bound on this probability.
%}
	\label{fig:pic}
\end{figure}
 
The two kinks of the red, lower-bound graph occur at $b=b_{**}$ (defined in Lemma~\ref{lem:b_{**}}) and at $b=1$, and the two kinks of the blue, upper-bound graph occur at $b=-1$ and $b=-b_{**}$. 

It appears that the maximum over all real $b\ge0$ of the difference $C_{\sup}e^{-b}-S(b)$ between the upper bound $C_{\sup}e^{-b}$ (on $S(b)$) and $S(b)$ occurs at $b=0$ and is equal to $C_{\sup}-(1-1/e)=0.0489\ldots$; however, this will not be proved or further considered in this paper. 

%!!!! how the blue graph was obtained 

\FloatBarrier

\section{Effective calculation of \texorpdfstring{$S(b)$}{S(b)}}\label{calculations}

By \eqref{eq:S(b)=}, if $b\le0$, then $S(b)$ has explicit expressions in terms of elementary functions. 

In the rest of this section, it is assumed by default that $b>0$. Then, again by \eqref{eq:S(b)=}, 
\begin{equation}\label{eq:=max=max}
	S(b)=\max_{r\ge r_b}p_1(r)=\max_{r>r_b}p_1(r), 
\end{equation}
with $p_1(r)=p_{1,b}(r)$ as defined in \eqref{eq:p_{1,b}} and expressed in \eqref{eq:p_{1,b}=} and \eqref{eq:p_{1,b} special}.  

We will need the following proposition. 

\begin{proposition}\label{prop:p'(r)}
Let $p'_1(r)=p'_{1,b}(r)$ denote both the right and left derivatives of the function $p_1=p_{1,b}$ at a point $r\in\R$, even if these derivatives at $r$ are not equal to each other. 
For all real $r$ we have 
\begin{equation}\label{eq:p'(r)}
	p'_1(r)=p'_{1,b}(r)>-\tfrac12\,e^{-b}; 
\end{equation}
this inequality holds for both the right and left derivatives. 
%where $p'_1(r)=p'_{1,b}(r)$ denotes the right or left derivative of the function $p_1=p_{1,b}$ at the point $r$. 
\end{proposition}

This proposition will be proved in Section~\ref{proof:p'(r)}.

\medskip
\hrule
\medskip

Let us now describe an algorithm to find the maxima in \eqref{eq:=max=max} with any prescribed accuracy. 

Take any real $b>0$. Pick any ``initial'' real $r_\ini>r_b$ (so that $p_{1,b}(r_\ini)>0$) and any real ``step'' $\s>0$. For all $j\in\Z$, let 
\begin{equation*}
	y_j:=p_1(r_\ini+j\s). 
\end{equation*}
It is possible that $y_j=y_{j+1}>0$ for some $j\in\Z$, even though, in view of Proposition~\ref{prop:F1}, this exception may only happen with probability $0$ assuming (say) that the pair $(r_\ini,\s)$ is drawn from an absolutely continuous probability distribution over $\R^2$. Nevertheless, in general we must deal with this unlikely, exceptional possibility. 

What is precluded by Proposition~\ref{prop:F1} is the situation when $y_j=y_{j+1}=y_{j+2}>0$ for some $j\in\Z$, and this observation is the key point in dealing with the mentioned exceptional case. 

Using (say) Mathematica, we can compute the values $y_j$ of the function $p_1$ with any degree of accuracy. So, for each $j\in\Z$ we can get a however small interval $Y_j$ containing the value $y_j$. If $y_j\ne y_{j+1}$, then, sufficiently increasing the accuracy of our approximation of these two values, we can make the corresponding intervals $Y_j$ and $Y_{j+1}$ disjoint. More specifically, if $y_j<y_{j+1}$ for some $j\in\Z$, then we can ``resolve'' this inequality by getting corresponding containing intervals $Y_j$ and $Y_{j+1}$ such that $Y_j<Y_{j+1}$ (meaning that $y<z$ for all $y\in Y_j$ and $z\in Y_{j+1}$);
%the interval $Y_j$ is strictly to the left of the interval $Y__{j+1}$; 
and if $y_j>y_{j+1}$, then we can ``resolve'' this inequality by getting corresponding containing intervals $Y_j$ and $Y_{j+1}$ such that $Y_j>Y_{j+1}$ (meaning that $y>z$ for all $y\in Y_j$ and $z\in Y_{j+1}$). 

%the interval $Y_j$ is strictly to the right of the interval $Y__{j+1}$. 

If now for some $j\in\Z$ we have $0<y_j<y_{j+1}$ -- and, unbeknownst to us, $y_{j+2}$ is either very close or equal to $y_{j+1}$, 
then we can partially ``resolve'' this inequality and the possibly approximate equality $y_{j+1}\ap y_{j+2}$ so that $Y_j<Y_{j+1}$ but with the intersection $Y_{j+1}\cap Y_{j+2}$ remaining nonempty. In such a case, let us code this situation ($y_j<y_{j+1}\ap y_{j+2}$) as the ``pattern'' string $<\ap$. So, for any $j\in\Z$, we could in general have $3\times3=9$ such strings $\r_1\r_2$, with $(\r_1,\r_2)\in\{<,>,\ap\}^2$. 

If we get the string $<>$ for some $j\in\Z$, then we will know that the unique maximizer, say $r_{\max}=r_{\max}(b)$, of $p_1=p_{1,b}$ is in the interval $[r_\ini+j\s,r_\ini+(j+2)\s]$.

On the other hand, in view of Proposition~\ref{prop:F1}, in our situation the string $><$ will be impossible. 

Moreover, crucially, because $y_j=y_{j+1}=y_{j+2}>0$ is impossible for any $j\in\Z$, we will never end up with the string $\ap\ap$, provided that we can increase the accuracy of our approximation sufficiently. 

So, we have to consider the remaining $6$ strings, each of length $2$: 
\begin{equation*}
	<<\qquad <\ap\qquad >>\qquad >\ap\qquad \ap<\qquad \ap>  
\end{equation*}
Note that, in accordance with the preceding discussion, none of these $6$ strings contains both characters $<$ and $>$. 

If the resulting string at some $j\in\Z$ contains the character $<$, we move by one step $\s$ to the right, looking now at $y_{j+1},y_{j+2},y_{j+3}$ ``at $j+1$'' instead of $y_j,y_{j+1},y_{j+2}$ ``at $j$''. If, ``at $j+1$'', the possible relations $<,>,\ap$ between $y_{j+1},y_{j+2},y_{j+3}$ resolve to a string containing $>$, then we will know that the unique maximizer $r_{\max}$ is in the interval $[r_\ini+j\s,r_\ini+(j+3)\s]$. If the resulting string ``at $j+1$'' does not contain $>$, then we move further by one step to the right. Etc. Again in view of Proposition~\ref{prop:F1} \emph{and} because we will never encounter the string $\ap\ap$, we will eventually, say at some $j_b\in\Z$, for the first time get a string containing $>$, and then we will know that the unique maximizer $r_{\max}$ is in the interval $[r_\ini+(j_b-1)\s,r_\ini+(j_b+2)\s]$. 

If the resulting string at some $j\in\Z$ contains the character $>$, we similarly move to the left, by one step or more.  
Here, too, we will never encounter the string $\ap\ap$. Indeed,  
we start fully and strictly to the right of $r_b$, at the point $r_\ini>r_b$, with $y_0=p_1(r_\ini),y_1=p_1(r_\ini+\s),y_2=p_1(r_\ini+2\s)$ all $>0$. If/when during the ``moving'' process $r_\ini+j\s$ first becomes $\le r_b$ and thus $y_j=p_1(r_\ini+j\s)$ becomes $0$ for some $j\in\Z$, then $y_j=0<y_{j+1}$ and hence 
for $y_j,y_{j+1},y_{j+2}$ we get a pattern string containing $<$, so that we then immediately stop, having never encountered $\ap\ap$. 

Thus, in all cases we will know that the unique maximizer $r_{\max}$ is in the interval $[r_\ini+i\s,r_\ini+(i+2)\s]$ or $[r_\ini+i\s,r_\ini+(i+3)\s]$ for some $i\in\Z$, and this interval is of length $2\s$ or $3\s$.  

Subdividing this interval into intervals of length $\s/2$, we will similarly (and possibly much faster) confine the maximizer $r_{\max}$ to an interval of length $\le3\s/2$. 

Continuing so, for any given natural $n$, we will confine $r_{\max}$ to a known interval $[\al,\be]$ of length $\le3\s/2^n$. Then, in view of Proposition~\ref{prop:p'(r)}, $S(b)$ will be in the  interval \break   
$[p_1(\be),p_1(\be)+3e^{-b}\,\s/2^{n+1})$, of length $3e^{-b}\,\s/2^{n+1}$ (which gets smaller as $b$ gets larger). 
Choosing a smaller initial step $\s$ and using the same $n$, one of course will in the end get a narrower interval for $S(b)$ -- albeit for the price of a possibly larger number of moves right or left before confining $r_{\max}$ to a finite interval.  

The described algorithm is effective, in the sense that for any real $b>0$ it allows one to compute $S(b)$ with any prescribed accuracy in a finite time, provided enough computing resources. However, almost certainly there are more efficient (even if more complicated and possibly requiring better \emph{a priori} estimates than \eqref{eq:p'(r)}) algorithms toward the same end. 

%%???? graphs of $S(b),I(b),Ce^{-b}$
%
%For an illustration of Theorems~\ref{th:df bounds} and \ref{th:upper bound}, here are the graphs $\{(b,I(b))\colon 0<b<3.5\}$ (red), $\{(b,S(b))\colon 0<b<3.5\}$ (red), and $\{(b,C_{\sup}e^{-b})\colon 0<b<3.5\}$ (red); recall here the definitions \eqref{eq:sups,infs} of $S(b)$ and $I(b)$ and \eqref{eq:C_*:=}.  % \\checking\\plot.nb
%\begin{figure}[htbp]
%	\centering
%		\includegraphics[width=.6\textwidth]{graphs.pdf}
%	\caption{Graphs of the exact pointwise lower and upper bounds on the tail probability for isotropic log-concave distribution, and of the exact exponential upper bound on this probability.}
%	\label{fig:pic}
%\end{figure}
% 
%The two kinks of the red, lower-bound graph occur at $b=b_{**}$ (defined in Lemma~\ref{lem:b_{**}}) and at $b=1$. 
%
%!!!! how the blue graph was obtained 

%\newpage
%
%\medskip
%\hrule
%\medskip

\section{Proof of Theorem~\ref{rem:H}}\label{2-piece proof}

\subsection{``Compact'' version of Theorem~\ref{rem:H}}\label{compact}

%\begin{proposition}\label{prop:exp decay}
%For any $\mu\in \L$ and any $s\in(0,1)$, 
%\begin{equation}
%	\int_\R e^{s|x|}\mu(dx)\le\frac{2e}{1-s}. 
%\end{equation}
%\end{proposition}

For any real $b\ge\sqrt3$, let $K^b:=[-b,b]$, 
\begin{equation*}
\PP^b:=\{\mu\in\PP\colon S_\mu\subseteq K^b\}, \quad 
	\L^b:=\L\cap\PP^b, \quad 
	\L_2^b:=\L_2\cap\PP^b, 
\end{equation*}
where $S_\mu$ denotes the support set of $\mu$. 

%$\L^b$ and $\L_2^b$ denote the sets of all measures in $\L$ and $\L_2$, respectively, supported on the interval $K^b:=[-b,b]$. Let $\L^b_\le$ denote the set of all measures in $\PP$ supported on $K^b$ and such that $\int_\R x\,\mu(dx)\le0$ and $\int_\R x^2\,\mu(dx)\le1$. Let also $\PP^b$ denote the set of all measures in $\PP$ supported on $

The l.s.c.\ function $H$ is bounded from below on the compact interval $K^b$. So, in view of \cite[Theorem~1.1]{feinberg_etal} (applied to the function $H-\inf_{K^b}H$, which is nonnegative on $K^b$) the function 
\begin{equation}\label{eq:Phi^b}
	\PP^b\ni\mu\mapsto\Phi^b(\mu):=\int_{K^b} H\,d\mu 
\end{equation}
is weakly l.s.c. 

By Borell's Theorem~2.2 \cite{borell74} and the weak compactness of $\PP^b$, the set $\L^b$ is weakly compact. Also, $\L^b$ is nonempty, because $b\ge\sqrt3$ and hence the uniform distribution over the interval $[-\sqrt3,\sqrt3]$ belongs to $\L^b$. 

The set $C^b:=\overline{\conv \L^b}$ is weakly compact, where $\conv M$ and $\overline M$ denote, respectively, the convex hull and  the closure of a subset $M$ of $\PP$ with respect to the weak topology over $\PP$. So, the l.s.c. function $\Phi^b$ attains its minimum value, say $m$, on $C^b$. Since $C^b$ is convex and weakly compact and $\Phi^b$ is affine and weakly l.s.c., the set $C^{b,\Phi}:=C^b\cap(\Phi^b)^{-1}(\{m\})$ is nonempty, convex, and weakly compact. So, by the Krein--Milman theorem (see e.g.\ \cite{phelps}), the set $\ext C^{b,\Phi}$ is nonempty, where $\ext M$ stands for the set of all extreme points of a set $M$. Hence,   
$\Phi^b$ attains its minimum value on $C^b$ at some $\mu^b\in\ext C^{b,\Phi}$. It is easy to see that $\ext C^{b,\Phi}\subseteq \ext C^b$, so that $\mu^b\in\ext C^b$.

Next, by Milman's theorem (see e.g.\ \cite[Theorem~3.25]{rudin}) and the mentioned weak compactness of $\L^b$, we have $\ext C^b=\ext\overline{\conv \L^b}\subseteq \L^b$. So, $\mu^b\in \L^b\subseteq\conv \L^b$. Also, $\mu^b\in\ext C^b=\ext\overline{\conv \L^b}$. So, $\mu^b\in\ext\conv \L^b$. 

Writing 
\begin{equation*}
	\L^b=
	\Big\{\mu\in\PP^b\colon\mu\text{ is log concave}, \int_\R x\,\mu(dx)=0,\ \int_\R (x^2-1)\,\mu(dx)=0\Big\},
\end{equation*}
it is easy to see that $\ext\conv \L^b\subseteq\ext\conv \L^b_\le$, where 
\begin{equation*}
	\L^b_\le:=
	\Big\{\mu\in\PP^b\colon\mu\text{ is log concave}, \int_\R x\,\mu(dx)\le0,\ \int_\R (x^2-1)\,\mu(dx)\le0\Big\}.
\end{equation*}

So, $\mu^b\in\ext\conv \L^b_\le$. Also, the measure $\mu^b$ is in $\I$ and hence non-degenerate. It now follows by Theorem~1 of Fradelizi and Gu\'edon \cite{frad-guel} (used with $p=2$, $d=1$, and $g_1(x)=-x$ and $g_2(x)=1-x^2$ for all real $x$) that $\mu^b$ has a two-piece exponential density; here we also use the fact that no measure in $\L$ can be degenerate. Recalling now that $\mu^b\in \L^b$, we see that the function $\Phi^b$ attains its minimum at 
\begin{equation*}
	\mu^b\in \L_2\cap\PP^b. 
\end{equation*}

\subsection{General case: approximation by the ``compact'' version}\label{compactification}

Here we will complete the proof of Theorem~\ref{rem:H} by reducing the consideration of the measures in $\L$ to that of measures in $\L^b$ and then use the result of Subsection~\ref{compact}. 

Toward this end, we will need the following proposition. 
%, possibly of independent interest. 

\begin{proposition}\label{prop:exp decay}
For any $v\in V$, there exists some real $y=y_v\ge0$ such that
for all log-concave measures $\mu\in\PP$ with $\int_\R x^2\mu(dx)\le1$ we have 
\begin{equation}\label{eq:exp decay}
	\int_\R e^{|x|}v(|x|)1(|x|\ge y_v)\mu(dx)\le M_v:=
	e^y v(y)+e\int_y^\infty v(x)\,dx<\infty. 
\end{equation}
\end{proposition}

\begin{proof}[Proof of Proposition~\ref{prop:exp decay}]
Take any log-concave measure $\mu\in\PP$ with $\int_\R x^2\mu(dx)\le1$. 
Let $X$ be a r.v.\ with the distribution $\mu$, so that $\E X^2\le1$. Then, by Lemma~5.7 of Lovasz and Vempala \cite{lovasz-vempala}, 
\begin{equation}\label{eq:exp-decay}
	\P(|X|>t)\le e^{1-t}
\end{equation}
for all real $t$. 
Take any $v\in V$. Let then $y$ be as in Definition~\ref{def:V}. 

For $x\in(y,\infty)$, let now $w(x):=v(x)+v'(x)$, so that $0\le w(x)\le v(x)$ and $(e^x v(x))'=e^x w(x)$. It follows that 
\begin{multline*}
	\infty>I:=e\int_y^\infty v(x)\,dx \ge e\int_y^\infty w(x)\,dx  
	=\int_y^\infty e^x e^{1-x}w(x)\,dx \\ 
	\ge\int_y^\infty e^x \P(|X|>x)w(x)\,dx  
	\ge\E 1(|X|\ge y)\int_y^\infty e^x 1(|X|>x)w(x)\,dx \\ 
	=\E 1(|X|\ge y)\int_y^{|X|} e^x w(x)\,dx  
	=\E 1(|X|\ge y)\big(e^{|X|}v(|X|)-e^y v(y)\big) \\ 
		=\int_\R e^{|x|}v(|x|)1(|x|\ge y)\mu(dx)-e^y v(y)\P(|X|\ge y). 
\end{multline*}
So, 
\begin{equation*}
	\int_\R e^{|x|}v(|x|)1(|x|\ge y)\mu(dx)\le e^y v(y)\P(|X|\ge y)+I\le M_v. 
\end{equation*}
\end{proof} 

Take now any function $H\colon\R\to\R$ as in Theorem~\ref{rem:H}, so that $H$ is a  
l.s.c.\ locally integrable function such that \eqref{eq:H le,v} holds for some $v\in V$. Then, by Proposition~\ref{prop:exp decay}, $H$ is $\mu$-integrable for all $\mu\in \L$. So, we have the function 
\begin{equation}\label{eq:Phi}
	\L\ni\mu\to\Phi(\mu):=\int_\R H\,d\mu, 
\end{equation}
and, for any real $b>0$,  
\begin{equation*}
	\Phi(\mu)=\Phi^b(\mu)\text{ for }\mu\in \L^b. 
\end{equation*}

Now take any 
\begin{equation}\label{eq: mu in L}
\mu\in \L. 
\end{equation}
Let $X$ be a r.v.\ with the distribution $\mu$. 
For all large enough natural $n$ we have $p_n:=\P(|X|\le n)>0$. For such $n$, let $X_n$ be a r.v.\ whose distribution is the conditional distribution of $X$ given $|X|\le n$. Then 
\begin{equation*}
	p_n|\E X_n|=|\E X1(|X|\le n)|=|\E X1(|X|>n)|\le\E X^2/n=1/n.
\end{equation*}
Therefore and because $p_n\to1$ (as $n\to\infty$), we have $\E X_n\to0$. 
Also, by Proposition~\ref{prop:exp decay} (with any $v\in V$ such that $v(x)=e^{-x}x^3$ for real $x\ge y:=3$), $\E|X|^3<\infty$; the latter inequality may also be deduced from \eqref{eq:exp-decay} or from \cite[formula~(16)]{karlin_etal61} or from \cite[formula~(0.4)]{keilson72}. So, 
\begin{equation*}
1-p_n\E X_n^2=1-\E X^2 1(|X|\le n)=\E X^2 1(|X|>n)\le\E|X|^3/n\to0, 
\end{equation*}
whence $\E X_n^2\to1$ and $\Var X_n\to1$. Also, of course, $X_n$ converges in distribution to $X$. 
So, letting 
\begin{equation*}
	Z_n:=\frac{X_n-\E X_n}{\sqrt{\Var X_n}}
\end{equation*}
for large enough natural $n$, we see that the distribution (say $\mu_n)$ of $Z_n$ is in $\L^{2n}$ and converges to $\mu$. For the p.d.f., say $g_n$, of $Z_n$ and all real $z$ we have  
\begin{equation*}
	g_n(z)=s_n f_n(s_n z+m_n)=\frac{s_n}{p_n} f(s_n z+m_n)1(|s_n z+m_n|\le n),
\end{equation*}
where $f$ is the p.d.f.\ of $X$, $f_n$ is the p.d.f.\ of $X_n$, 
$s_n:=\sqrt{\Var X_n}\to1$, and $m_n:=\E X_n\to0$. 

The function $f$ is log concave and hence has at most two points of discontinuity. So, $g_n\to f$ almost everywhere (a.e.) on $\R$. 
%The log-concavity of $f$ also implies that $f$ is bounded: otherwise, we would have $f(x)\to\infty$ if $x\to\infty$ or if $x\to-\infty$, which would contradict the condition $\int_\R f=1$. So, 

By %\cite[Lemma~5.5(a)]{lovasz-vempala}, 
\eqref{eq:f le1}, $g_n\le1$ for all $n$. 
%the p.d.f's $g_n$ are bounded by the same real constant for all large enough $n$. 
Therefore, for each real $z>0$, by the local integrability of $H$ and dominated convergence, 
\begin{equation}\label{eq:le z}
	\E H(Z_n)1(|Z_n|\le z)=\int_{-z}^z H(t)g_n(t)\,dt
	\to\int_{-z}^z H(t)f(t)\,dt=\E H(X)1(|X|\le z). 
\end{equation}

Take now any real $\ep>0$. By condition \eqref{eq:H le,v}, for some real $x_\ep>0$ and all real $x$ with $|x|\ge x_\ep$ we have 
\begin{equation*}%\label{eq:H le,v}
	|H(x)|\le\ep e^{|x|}v(|x|).  
\end{equation*}
Therefore and in view of Proposition~\ref{prop:exp decay}, for all real $z\ge z_\ep :=\max(x_\ep,y_v)$ and all $n$, 
\begin{equation}\label{eq:le ep M}
	|\E H(Z_n)1(|Z_n|>z)|\le\ep\E e^{|Z_n|}v(|Z_n|)1(|Z_n|>y_v)\le\ep M_v  
\end{equation}
and, similarly, $|\E H(X)1(|X|>z)|\le\ep M_v$. Hence, looking back at \eqref{eq:le z}, we see that $\E H(X)\ge\limsup_n \E H(Z_n)-2\ep M_v$. Letting now $\ep\downarrow0$, we get 
\begin{equation}\label{eq:ge limsup}
	\E H(X)\ge\limsup_n \E H(Z_n). 
\end{equation}
%
%
%The support of the distribution of $X_n$ is contained in the interval $[-n,n]$. Therefore and because $\E X_n\to0$ and $\Var X_n\to1$, for all large enough $n$ the support of the distribution $\mu_n$ of $Z_n$ is contained in the interval $[-2n,2n]$.
By the conclusion of Subsection~\ref{compact} and because, as was noted, $\mu_n\in\L^{2n}$ for all large enough $n$, 
\begin{equation*}
\E H(Z_n)=\int_\R H\,d\mu_n\ge \int_\R H\,d\mu^{2n}.  
\end{equation*}
Therefore and in view of \eqref{eq:ge limsup}, for any $\mu\in \L$ there is a sequence $(\nu_n)$ of measures in $\L_2$ such that $\lim_n \int_\R H\,d\nu_n$ exists and
\begin{equation}\label{eq:int H dmu ge}
	\int_\R H\,d\mu=\E H(X)\ge\lim_n \int_\R H\,d\nu_n.  
\end{equation}
%and 
%\begin{equation}
%	\mu^{2n}\in \L_2\cap\PP^{2n}.
%\end{equation}
Because of the condition $\int_\R x^2\,\mu(dx)=1$ in \eqref{eq:iso}, any sequence in $\I$ (and hence any sequence in $\L$ and hence any sequence in $\L_2$) is tight. So, by Prohorov's theorem (see e.g.\ \cite[Theorem~5.1]{billingsley99}), without loss of generality (wlog), the sequence $(\nu_n)$ in $\L_2$ weakly converges to a measure $\nu\in\PP$: 
\begin{equation}\label{eq: nu_n to nu}
	\nu_n\to\nu.
\end{equation}
Actually, by Borell's Theorem~2.2 \cite{borell74}, the measure $\nu$ is log concave. 
Moreover, by \eqref{eq:exp-decay}, for any r.v.\ $Y$ with the distribution in $\L$, 
\begin{equation*}
	\E|Y|^3=\int_0^\infty 3t^2 \P(|Y|>t)\,dt\le \int_0^\infty 3t^2 e^{1-t}\,dt=6e<\infty. 
\end{equation*}
So, by uniform integrability (cf.\ e.g.\ \cite[Theorem~3.5]{billingsley99}), 
\begin{equation}\label{eq:nu in L}
	\nu\in \L. 
\end{equation}

For real $z>0$ and all real $x$, let 
\begin{equation}\label{eq:H_z}
	H_z(x):=\begin{cases}
	H(x)&\text{ if }-z\le x\le z, \\ 
	H(z)&\text{ if }x\ge z, \\ 
		H(-z)&\text{ if }x\le -z. 
	\end{cases}
\end{equation}
Then the function $H_z$ is l.s.c.\ and bounded from below. So, by the Fatou lemma for weak convergence \cite[Theorem~1.1]{feinberg_etal}, applied to the nonnegative function $H_z-\inf H_z$, 
\begin{equation}\label{eq:H_z compar}
\liminf_n \int_\R H_z\,d\nu_n\ge \int_\R H_z\,d\nu
\end{equation}
for each real $z>0$. 
Next, for all real $z>0$ and all real $x$, 
\begin{multline*}
	|H(x)-H_z(x)|=|H(x)-H_z(z)|\,1(x>z)+|H(x)-H_z(-z)|\,1(x<-z) \\ 
	\le(|H(x)|+h_z)1(|x|>z),
\end{multline*}
where $h_z:=|H_z(z)|+|H_z(-z)|$. 
%Take now any real $\ep>0$. The for all $z\ge z_\ep$ and all $n$, 
Let now $z\to\infty$. Then (cf.\ \eqref{eq:le ep M}) 
\begin{equation*}%\label{eq:le ep M}
	\sup_n\int_\R|H(x)|1(|x|>z)\nu_n(dx)\to0  
\end{equation*}
and, by \eqref{eq:exp-decay} and \eqref{eq:H le,v}, 
\begin{equation*}
	\int_\R h_z1(|x|>z)\nu_n(dx)\le h_z e^{1-z}
	=o\big(e^z v(z)e^{1-z}\big)=o(1), 
\end{equation*}
since the function $v$ is nonnegative, nonincreasing, and 
integrable in a neighborhood of $\infty$. Hence, as $z\to\infty$, 
\begin{equation*}%\label{eq:le ep M}
	\sup_n\int_\R|H-H_z|d\nu_n\to0  
\end{equation*}
%%%%%Again by the Fatou lemma for the weak convergence, for each real $z>0$, 
%%%%%\begin{equation}%\label{eq:le ep M}
%%%%%	\int_\R|H-H_z|d\nu\le\sup_n\int_\R|H-H_z|d\nu_n\to0  
%%%%%\end{equation}
To obtain the latter conclusion, we only used the condition that $\nu_n\in \L$ for all $n$. So, in view of \eqref{eq:nu in L}, 
we similarly have 
\begin{equation*}%\label{eq:le ep M}
	\int_\R|H-H_z|d\nu\to0 
\end{equation*}
as $z\to\infty$. 
Combining the latter two displays with \eqref{eq:H_z compar}, we see that 
\begin{equation}\label{eq:liminf}
	\liminf_n \int_\R H\,d\nu_n\ge \int_\R H\,d\nu. 
\end{equation}
Using also \eqref{eq:int H dmu ge}, we get 
\begin{equation}\label{eq:mu ge nu}
	\int_\R H\,d\mu\ge\int_\R H\,d\nu;   
\end{equation}
recall that here $\mu$ is any measure in $\L$ (as was assumed in \eqref{eq: mu in L}) and, in view of \eqref{eq: nu_n to nu}, $\nu$ is the weak limit of measures in $\L_2$. One may note here that $\nu$ depends on the choice of $\mu$. 

To complete the proof of Theorem~\ref{rem:H}, we will use the following proposition.  

\begin{proposition}\label{prop:L_2}
The sets $\F_2,\F_1,\F_1\cup\F_{2,\infty}$ of functions are sequentially compact with respect to the almost everywhere (a.e.) convergence on $\R$.
The sets $\L_2,\L_1,\L_1\cup\L_{2,\infty}$ of measures are (sequentially) compact with respect to the weak convergence. 
\end{proposition}

\begin{proof}[Proof of Proposition~\ref{prop:L_2}]
%Let $(\mu_n)$ be a sequence in $\L_2$ converging weakly to a measure $\mu\in\PP$. Then (cf.\ \eqref{eq:nu in L}) $\mu\in \L$. For each $n$, $\mu_n$ has a density $f_n$ such that for all real $x$ 
Let $(f_n)$ be any sequence in $\F_2$. Then 
% converging a.e.\ to a function $f$. Then, for each $n$ and all real $x$,
\begin{equation}\label{eq:F_2}
	f_n(x)=\begin{cases}
	e^{l_n(x)}&\text{ if }-A_n<x<u_n, \\ 
	e^{r_n(x)}&\text{ if }u_n<x<B_n, \\ 
	0&\text{ if }x\notin(-A_n,B_n), 
	\end{cases}
\end{equation}
where $-\infty\le-A_n\le B_n\le\infty$, $-A_n\le u_n\le B_n$, and $l_n$ and $r_n$ are affine functions on $\R$ such that $l_n(u_n)=r_n(u_n)$ if $u_n\in\R$. 

For each $n$, let $\mu_n\in\L_2$ be the measure with density $f_n$. 

Passing to a subsequence, wlog we have $A_n\to A$, $B_n\to B$, $u_n\to u$ for some $A,B,u$ such that $-\infty\le-A\le u\le B\le\infty$, and we also wlog have $\mu_n\to\mu$ weakly for some $\mu\in\L$ (cf.\ \eqref{eq:nu in L}). 

Let $D$ be an at most countable dense subset of the interval $(-A,u)$; note that $D\subset\R$. By diagonal selection, wlog there is a function $\la\colon D\to[-\infty,\infty]$ such that 
\begin{equation*}
	l_n(x)\to\la(x)\text{ for all }x\in D. 
\end{equation*}

One of the following three cases must occur. 

\emph{Case 1: There are distinct points $x_0$ and $x_1$ in $D$ such that $\la(x_0)$ and $\la(x_1)$ are real numbers.} 

\noindent With $x_t$ as defined in \eqref{eq:x_t}, 
%\begin{equation}\label{eq:x_t}
%	x_t:=(1-t)x_0+tx_1
%\end{equation}
%for real $t$, 
we have the bijection $\R\ni t\mapsto x_t\in\R$ 
and 
\begin{equation*}
	l_n(x_t)=(1-t)l_n(x_0)+tl_n(x_1)\to(1-t)\la(x_0)+t\la(x_1)=:l(x_t),
\end{equation*}
so that $l_n$ converges to the affine function $l$ pointwise on $\R$ and hence $f_n(x)\to e^{l(x)}$ for each $x\in(-A,u)$. 

\emph{Case 2: There are points $x_0$ and $x_1$ in $D$ such that $x_0<x_1$ and  $\la(x_0)=\infty=\la(x_1)$.} 

\noindent Then $l_n\to\infty$ uniformly on the interval $[x_0,x_1]$, and hence $\mu_n([x_0,x_1])\to\infty$, which contradicts the condition that $\mu_n\in\PP$. So, Case~2 is actually impossible.

\emph{Case 3: There is at most one point in $D$ with a real value of the function $\la$ and there is at most one point in $D$ with value $\infty$ of $\la$}. 

\noindent Then $\la=-\infty$ on a subset $D_*$ of $D$ that differs from $D$ by at most two points. So, $D_*$ is an at most countable dense subset of the interval $(-A,u)$. For any points $x_0$ and $x_1$ in $D_*$ such that $x_0<x_1$, $l_n\to-\infty$ uniformly on the interval $[x_0,x_1]$, and hence $\mu_n([x_0,x_1])\to0$ and $\mu((x_0,x_1))=0$. So, in Case~3 we have $\mu((-A,u))=0$. 

A similar triple of cases involves the $r_n$'s, $u$, and $B$, rather than the $l_n$'s, $u$, and $-A$. 

In each of the two triples of cases, the second case is impossible. So, pairing the first or the third case of the triple of cases involving the $l_n$'s, $u$, and $-A$ with the first or the third case of the triple of cases involving the $r_n$'s, $u$, and $B$, we see that one of the following four cases must occur.  

\emph{Case 11: $f_n(x)\to e^{l(x)}$ for each $x\in(-A,u)$ and $f_n(x)\to e^{r(x)}$ for each $x\in(u,B)$, where $l$ and $r$ are affine functions on $\R$; moreover, $l'_n\to l'$ if $-A<u$ and $r'_n\to r'$ if $u<B$.} 

\noindent In this case, if $-A<u<B$, then $u_n\in\R$ eventually, so that $l_n(u_n)=r_n(u_n)$, $|l_n(u_n)-l_n(u)|=|l'_n|\,|u_n-u|$, and $|r_n(u_n)-r_n(u)|=|r'_n|\,|u_n-u|$, whence $l(u)-r(u)=\lim_n(l_n(u)-r_n(u))=0$, so that $l(u)=r(u)$. Also, for each $x\in(-\infty,-A)\cup(B,\infty)$ eventually $f_n(x)=0$. Thus, letting 
\begin{equation*}
	f(x):=
	\begin{cases}
	e^{l(x)}&\text{ if }-A<x<u, \\ 
	e^{r(x)}&\text{ if }u<x<B, \\ 
	e^{l(u)}=e^{r(u)}&\text{ if }-A<x=u<B, \\ 
	0&\text{ if }x\notin(-A,B), 
	\end{cases}
\end{equation*}
we have $f_n(x)\to f(x)$ for $x\in\R\setminus\{-A,u,B\}$. Also, $l'_n\ge r'_n$ for all $n$ and hence $l'\ge r'$, so that 
$f$ is log concave. 
%, since the $f_n$'s are log concave. 
We conclude that $f$ is a log-concave two-piece exponential function. 

\emph{Case 13: $f_n(x)\to e^{l(x)}$ for each $x\in(-A,u)$, $f_n(x)\to0$ for each $x\in(u,B)$, and hence $\mu((u,B))=0$, where $l$ is an affine function on $\R$; moreover, $l'_n\to l'$ if $-A<u$.} 

\noindent In this case, letting 
\begin{equation*}
	f(x):=
	\begin{cases}
	e^{l(x)}&\text{ if }-A<x<u, \\ 
	0&\text{ if }u\le x<B, \\ 
	0&\text{ if }x\notin(-A,B),  
	\end{cases}
\end{equation*}
we again have $f_n(x)\to f(x)$ for $x\in\R\setminus\{-A,u,B\}$,  
and $f$ is a one-piece exponential function. 
% (with $l_1=l_2=l$ and $u$ in place of $B$ in \eqref{eq:2-piece}). 

\emph{Case 31: $f_n(x)\to0$ for each $x\in(-A,u)$, $f_n(x)\to e^{r(x)}$ for each $x\in(u,B)$, so that $\mu((-A,u))=0$, where $r$ is an affine function on $\R$; moreover, $r'_n\to r'$ if $u<B$.} 

\noindent This case is quite similar to Case~13. 

\emph{Case 33: $f_n(x)\to 0$ for each $x\in\R\setminus\{-A,u,B\}$.} 

\noindent Here $\mu((-A,u))=0=\mu((u,B))$.  
Moreover, $\mu$ is supported on $[-A,B]$ and is non-atomic. Hence, $\mu(\R)=0$ while $\mu\in\PP$. 
So, Case~33 is actually impossible. 

%\noindent So, letting here $f(x):=0$ for all real $x$, we again have $f_n(x)\to f(x)$ for $x\in\R\setminus\{-A,u,B\}$, and $f$ is a two-piece exponential function (with $-A=B$ in \eqref{eq:2-piece}). 

%%This contradicts the conclusion, noted above, that $f_n$ converges pointwise on $\R\setminus\{-A,u,B\}$ to a density of the probability measure $\mu$. So, Case~33 is actually impossible. 
%a.e.\ on $\R$. 
%
%In this case, $\mu(\R)=\int_R f=0$, which contradicts the condition $\mu\in\PP$. 

Thus, 
%of the four cases -- Case~11, Case~13, Case~31, and Case~33 -- only the first three ones are possible. 
%In 
in each of three possible cases -- Case~11, Case~13, and Case~31 -- we had a two-piece exponential function $f\colon\R\to\R$ such that $f_n\to f$ pointwise on $\R\setminus\{-A,u,B\}$ and thus 
a.e.\ on $\R$. Also, by %\cite[Lemma~5.5(a)]{lovasz-vempala}
\eqref{eq:f le1}, $f_n\le1$ for all $n$. So, taking any compact interval $K$, by dominated convergence we will have $\mu_n(K)=\int_K f_n\to\int_K f$, and we also have $\mu_n(K)\to\mu(K)$, since $\mu$ is absolutely continuous and hence non-atomic. 
%
%For each $n$, $f_n$ is in $\F^2$ and hence is a density of a measure $\mu_n\in\L_2$. Passing to a subsequence, wlog we may assume that $\mu_n\to\mu$ weakly for some $\mu\in \L$, so that $\mu$ is absolutely continuous and hence non-atomic, which implies $\mu_n(K)\to\mu(K)$ for $K$ as above. 
So, $\mu(K)=\int_K f$, for every compact interval $K$, which means that $f$ is a density of $\mu$. 
%%(One might now notice that $f$ being a density of the probability measure $\mu$ makes Case~33 actually impossible.)
So, $\mu$ is in $\L_2$ and $f\in\F_2$. 

It remains to note the following: (i) if $f_n\in\F_1$ for each $n$, then \eqref{eq:F_2} will hold with $u_n=B_n$ for each $n$, so that $u=B$ and hence $f\in\F_1$ and (ii) if $f_n\in\F_{2,\infty}$ for each $n$, then \eqref{eq:F_2} will hold with $A_n=B_n=\infty$ for each $n$, so that one of the three possible cases -- Case~11, Case~13, or Case~31 -- will occur with $A=B=\infty$, and in each of such cases we will have $f\in\F_1\cup\F_{2,\infty}$. 
%
%$A=B=\infty$ and hence 
%
%$f\in\F_1$ if $u\in\{-A,B\}$ and $f\in\F_{2,\infty}$ if $u\notin\{-A,B\}$. 
%%Also, $f$ is log concave, as an a.e.\ limit of the log-concave functions $f_n$. 

%%%It remains to consider the four mentioned cases to show that in each of those case the function $f$ is a.e.\ equal to a two-piece exponential function, so that $\mu\in \L_2$. 

This completes the proof of Proposition~\ref{prop:L_2}.
\end{proof}

Take now any sequence $(\la_n)$ in $\L_2$ such that 
$\int_\R H\,d\la_n\to\inf_{\la\in\L_2}\int_\R H\,d\la$. By the weak compactness of $\L_2$, established in Proposition~\ref{prop:L_2}, wlog $\la_n\to\la_\infty$ weakly for some $\la_\infty\in\L_2$. 
Repeating the argument \eqref{eq:H_z}--\eqref{eq:liminf} with the $\la_n$'s instead of the $\nu_n$'s, we will get 
\begin{equation}\label{eq:inf}
	\inf_{\la\in\L_2}\int_\R H\,d\la=\lim_n \int_\R H\,d\la_n
	\ge \int_\R H\,d\la_\infty. 
\end{equation}
Thus, the infimum in \eqref{eq:inf} is attained. 

On the other hand, by \eqref{eq:mu ge nu}, for any $\mu\in\L$, 
$\int_\R H\,d\mu\ge\inf_{\la\in\L_2}\int_\R H\,d\la$ and hence 
$\int_\R H\,d\mu\ge\int_\R H\,d\la_\infty$. It remains to recall that the measure $\la_\infty$ is in $\L_2$ and its construction did not depend on the choice of $\mu$. 

Thus, the proof of Theorem~\ref{rem:H} is completed. 

\section{Proofs of Theorem~\ref{th:4 zeros} and of Corollary~\ref{cor:b}}\label{4 zeros proof} 

Here we will need the following three lemmas, whose proofs will be given at the end of this section. 

\begin{lemma}\label{lem:2 zeros}
Suppose that $-\infty\le a<b\le\infty$ and 
$F\colon(a,b)\to\R$ is a continuous function such that $\int_a^b F=0=\int_a^b xF(x)\,dx$. Then $F$ has at least two zeros. 
\end{lemma}

\begin{lemma}\label{lem:perturb}
Suppose that a p.d.f.\ $f$ of a r.v.\ $X$ is of bounded variation. Let $\nu$ be the finite signed measure defined by the formula 
\begin{equation*}
	\nu(dx)=f(x)dx+\frac1k\,df(x)
\end{equation*}
for some real $k>0$. Let $e_k(x):=ke^{-kx}\,1(x>0)$ for real $x$, so that 
$e_k$ is a p.d.f.\ of the exponential distribution with mean $1/k$. Then 
\begin{equation*}
	f=\nu*e_k.
\end{equation*}
It follows that, if $\nu$ is a nonnegative measure and $\E X^2<\infty$, then $\Var X\ge1/k^2$; if, moreover, the measure $\nu$ is nondegenerate, then $\Var X>1/k^2$.
%If $\nu$ is a (nonnegative) measure, then $\nu$ is the probability distribution of a r.v.\ $Y$.  
\end{lemma}

\begin{lemma}\label{lem:k>1}
Take any $f\in\F_2\setminus(\F_1\cup\F_{2,\infty})$ and consider the parametrization \eqref{eq:f_{A,B,u,c,h,k}}. If $B=\infty$, then $k>1$. If $A=\infty$, then $h>1$. 
\end{lemma} 

%This lemma will be proved at the very end of this section.

\begin{proof}[Proof of Theorem~\ref{th:4 zeros}]
By Theorem~\ref{rem:H}, the infima in \eqref{eq:Th1.1} are attained at some $\mu\in\L_2$. Equivalently, $\inf_{f\in\F_2}\int_\R Hf$ is attained. It is enough to show that any minimizer of $\int_\R Hf$ over $f\in\F_2$ is actually in $\F_1\cup\F_{2,\infty}$. 

%The p.d.f.'s $f\in\F_2$ can be parametrized by the formula 
%\begin{equation}\label{eq:f_{A,B,u,c,h,k}}
%	f(x)=f_{A,B,u,c,h,k}(x):=c e^{h(x-u)}1(-A<x\le u)+c e^{-k(x-u)}1(u<x<B)
%\end{equation}
%for all real $x$, where 
%\begin{equation}\label{eq:A,B,u,c,h,k}
%\begin{gathered}
%	\text{$-\infty\le-A\le u\le B\le\infty$, $-A<B$, $u\in\R$, $c\in(0,\infty)$,} \\ 
%	\text{and $h$ and $k$ are real numbers such that $h\ge-k$.} 
%\end{gathered}	
%\end{equation}
%One may note here that necessarily $A>0$ and $B>0$, because otherwise the barycenter of the probability distribution $\mu$ with density $f$ could not be $0$, violating the condition $\mu\in\L_2\subseteq\L$. 

To obtain a contradiction, suppose that $f_{A,B,u,c,h,k}\in\F_2$ is a minimizer of $\int_\R Hf$ over $f\in\F_2$ but $f_{A,B,u,c,h,k}\notin\F_1\cup\F_{2,\infty}$. Then $-A<u<B$, $h\ne-k$, and at least one of the extended real numbers $A,B$ is real. 
For real $\la_0,\dots,\la_3,x$, let 
\begin{equation*}
	R_{\la_0,\dots,\la_3}(x):=\la_0 H(x)-\la_1-\la_2 x-\la_3 x^2. 
\end{equation*}
For any $(\la_0,\dots,\la_3)\in\R^4$, define the function $\Psi:=\Psi_{\la_0,\dots,\la_3}$ by the formula 
\begin{equation*}
	\Psi_{\la_0,\dots,\la_3}(\tilde A,\tilde B,\tilde u,\tilde c,\tilde h,\tilde k)
	:=\int_\R R_{\la_0,\dots,\la_3}f_{\tilde A,\tilde B,\tilde u,\tilde c,\tilde h,\tilde k}
\end{equation*}
for $(\tilde A,\tilde B,\tilde u,\tilde c,\tilde h,\tilde k)$ satisfying the corresponding conditions in \eqref{eq:A,B,u,c,h,k}. 

By the Lagrange--Carath\'eodory multiplier rule \cite[p.~441]{pourciau}, there exists a nonzero quadruple $(\la_0,\dots,\la_3)\in\R^4$ such that for $R:=R_{\la_0,\dots,\la_3}$ at the tuple $(A,B,u,c,h,k)$ corresponding to the minimizer $f=f_{A,B,u,c,h,k}\in\F_2$ of $\int_\R Hf$ over $f\in\F_2$
%\break  
we have, in view of the definition in \eqref{eq:f_{A,B,u,c,h,k}}, 
\begin{equation}\label{eq:ints}
\begin{aligned}
	0&=\dd_c \Psi(A,B,u,c,h,k)=\frac1c\int_{-A}^u Rf+\frac1c\int_u^B Rf,  \\ 
	0&=\dd_u \Psi(A,B,u,c,h,k)=-h\int_{-A}^u Rf+k\int_u^B Rf,  \\ 
	0&=\dd_h \Psi(A,B,u,c,h,k)=\int_{-A}^u (x-u)R(x)f(x)\,dx,  \\ 
	0&=\dd_k \Psi(A,B,u,c,h,k)=-\int_u^B (x-u)R(x)f(x)\,dx,  
\end{aligned}	
\end{equation}
and also 
\begin{equation}\label{eq:R(-A)}
	0=\dd_A \Psi(A,B,u,c,h,k)=ce^{h(-A-u)}R(-A) 
\end{equation}
if $A\in\R$ or
\begin{equation}\label{eq:R(B)}
	0=\dd_B \Psi(A,B,u,c,h,k)=ce^{-k(B-u)}R(B) 
\end{equation}
if $B\in\R$. 
To obtain the second equalities in the last two of the four formulas in the 4-line display \eqref{eq:ints}, we used interchanging a derivative with
an integral -- see e.g.\ \cite[Theorem~2.27(b)]{folland}. In turn, to use the latter theorem, we recall condition \eqref{eq:H le,v} and use Lemma~\ref{lem:k>1}. 
%the following lemma. 

%\begin{lemma}\label{lem:k>1}
%Take any $f\in\F_2\setminus(\F_1\cup\F_{2,\infty})$ and consider the parametrization \eqref{eq:f_{A,B,u,c,h,k}}. If $B=\infty$, then $k>1$. If $A=\infty$, then $h>1$. 
%\end{lemma} 
%
%This lemma will be proved at the very end of this section. 
%after the rest of the proof of Theorem~\ref{th:4 zeros}. 

Indeed, if $B=\infty$, then, in view of Lemma~\ref{lem:k>1}, any possible value of $k$ can be written as $k_*:=1+\ep$ for some real $\ep>0$. Now note that, 
for any possible fixed values of $A,u,c,h$, there is some real $x_\ep\ge u$ such that $I_k(x):=|(x-u)R(x)f_{A,\infty,u,c,h,k}(x)|\le e^{-\ep x/3}$ for all $k$ in the neighborhood $K:=[k_*-\ep/3,k_*+\ep/3]$ of $k_*$ and all $x\ge x_\ep$. Together with the fact that $\sup_{k\in K}\sup_{x\in[u,x_\ep]}I_k(x)<\infty$, this provides an integrable majorant required in \cite[Theorem~2.27(b)]{folland}.
Similarly we deal with $A=\infty$.

By \eqref{eq:ints} and the condition $h\ne-k$, we have $$0=\int_{-A}^u Rf=\int_u^B Rf=\int_{-A}^u x R(x)f(x)\,dx=\int_u^B x R(x)f(x)\,dx.$$ So, by Lemma~\ref{lem:2 zeros}, the function $Rf$ and hence the function $R$ have at least two zeros in the interval $(-A,u)$ and at least two zeros in the interval $(u,B)$. Also, in view of \eqref{eq:R(-A)} and \eqref{eq:R(B)}, $R(-A)=0$ or $R(B)=0$. It follows that the function $R=R_{\la_0,\dots,\la_3}$ has at least $5$ zeros; since the quadruple $(\la_0,\dots,\la_3)$ is nonzero, we have now obtained a contradiction with the condition in Theorem~\ref{th:4 zeros} that the function $R_{\la_1,\la_2,\la_3}$ has at most $4$ zeros (here consider separately the cases $\la_0=0$ and $\la_0\ne0$; in the latter case, wlog $\la_0=1$). 
This completes the proof of Theorem~\ref{th:4 zeros} -- except that it remains to prove Lemmas~\ref{lem:2 zeros}, \ref{lem:perturb}, and \ref{lem:k>1}, which will be done after the proof of Corollary~\ref{cor:b}.  
\end{proof}

\begin{proof}[Proof of Corollary~\ref{cor:b}]
For simplicity of writing, assume that $\ka=1$; the proof for other real values of $\ka$ is almost literally the same. 

Take any real $b$. For any real $\si>0$, let $H_{b,\si}$ be the c.d.f.\ of the normal distribution with mean $b$ and standard deviation $\si$. Then $H_{b,\si}\to H_b:=1_{[b,\infty)}$ pointwise everywhere on $\R\setminus\{b\}$ as $\si\downarrow0$. For each real $\si>0$, condition \eqref{eq:H le} obviously holds for the function $H_{b,\si}$, which is also $4$-good, by Remark~\ref{rem:4 zeros}. So, by Theorem~\ref{th:4 zeros}, for the function $H_{b,\si}$ in place of $H$, the infima in \eqref{eq:Th1.1} are attained at a measure $\mu_{b,\si}\in \L_1\cup\L_{2,\infty}$ with a density $f_{b,\si}\in\F_1\cup\F_{2,\infty}$. 
By 
%the weak compactness of $\L$ and the penultimate paragraph in the proof of 
Proposition~\ref{prop:L_2}, for some sequence $(\si_n)$ of positive numbers converging to $0$ and some measure $\mu_b\in\L_1\cup\L_{2,\infty}$ with a density $f_b\in\F_1\cup\F_{2,\infty}$, we have $\mu_{b,\si_n}\to\mu_b$ weakly and $f_{b,\si_n}\to f_b$ a.e. 

%{\huge !!!!!} need  $f$ to be of the form \eqref{eq:2-piece} where either $l_1=l_2$ or $A=B=\infty$. 

By the weak convergence $\mu_{b,\si_n}\to\mu_b$ and because any measure in $\L$ is non-atomic, we have 
\begin{equation}\label{eq:cor:b1}
	\int_\R H_b\, f_{b,\si_n}=\int_\R H_b\, d\mu_{b,\si_n}\to
	\int_\R H_b\, d\mu_b=\int_\R H_b\, f_b. 
\end{equation}
Now take any real $\ep>0$. Then $H_{b,\si_n}\to H_b$ uniformly on $\R\setminus(b-\ep,b+\ep)$. Therefore and because the $f_{b,\si_n}$'s are p.d.f.'s, 
\begin{equation*}
	\int_{\R\setminus(b-\ep,b+\ep)}|H_{b,\si_n}-H_b|\,f_{b,\si_n}\to0. 
\end{equation*}
Next, 
\begin{equation*}
	\int_{(b-\ep,b+\ep)}|H_{b,\si_n}-H_b|\,f_{b,\si_n}
	\le\int_{(b-\ep,b+\ep)}f_{b,\si_n}\le2\ep, 
\end{equation*}
because, again by %\cite[Lemma~5.5(a)]{lovasz-vempala}
\eqref{eq:f le1}, $f_{b,\si_n}\le1$. 
So, 
\begin{equation}\label{eq:cor:b4}
	\int_\R H_{b,\si_n}\,f_{b,\si_n}-\int_\R H_b\,f_{b,\si_n}\to0. 
\end{equation}
Comparing \eqref{eq:cor:b1} and \eqref{eq:cor:b4}, we see that 
\begin{equation}\label{eq:cor:b5}
	\int_\R H_{b,\si_n}\,d\mu_{b,\si_n}=\int_\R H_{b,\si_n}\,f_{b,\si_n}
	\to\int_\R H_b\,f_b=\int_\R H_b\,d\mu_b. 
\end{equation}

Now take any $\mu\in \L$, with a p.d.f.\ $f$. Then, similarly to \eqref{eq:cor:b4}, we have $\int_\R H_{b,\si_n}\,f\to \int_\R H_b\,f$, that is 
\begin{equation}\label{eq:cor:b6}
	\int_\R H_{b,\si_n}\,d\mu\to \int_\R H_b\,d\mu. 
\end{equation}
By the definition of $\mu_{b,\si_n}$, 
\begin{equation*}
	\int_\R H_{b,\si_n}\,d\mu_{b,\si_n}\le\int_\R H_{b,\si_n}\,d\mu. 
\end{equation*}
Therefore, and in view of \eqref{eq:cor:b5} and \eqref{eq:cor:b6},  
\begin{equation*}
	\int_\R H_b\,d\mu_b\le\int_\R H_b\,d\mu, 
\end{equation*}
for any $\mu\in \L$. This completes the proof of Corollary~\ref{cor:b}. 
\end{proof}

\begin{proof}[Proof of Lemma~\ref{lem:2 zeros}]
If the continuous function $F$ does not have a zero, then either $F>0$ or $F<0$, which will contradict the conditions $\int_a^b F=0$ and $a<b$. So, $F(c)=0$ for some $c\in(a,b)$. Suppose now that $c$ is the only zero of $F$. Then $F(x)$ is of a constant nonzero sign for all $x\in(a,c)$ and of a constant nonzero sign for all $x\in(c,b)$, and these two constant signs must be opposite to each other, because otherwise $F(x)$ would be of a constant nonzero sign for all $x\in(a,b)\setminus\{c\}$, which would again contradict the conditions $\int_a^b F=0$ and $a<b$. So, $(x-c)F(x)$ is of a constant nonzero sign for all $x\in(a,b)\setminus\{c\}$, which finally contradicts the conditions $\int_a^b(x-c)F(x)\,dx=\int_a^b xF(x)\,dx-c\int_a^b F(x)\,dx=0$ and $a<b$. 
\end{proof}

\begin{proof}[Proof of Lemma~\ref{lem:perturb}]
For any finite signed Borel measure $\la$ over $\R$, let $\hat\la$ denote the Fourier transform of $\la$, so that $\hat\la(t)=\int_\R e^{itx}\la(dx)$ for real $t$. If $l$ is a density of $\la$, let $\hat l:=\hat\la$. 

Then, integrating by parts and noting that $f(x)\to0$ as $|x|\to\infty$, for all real $t$ we have $\widehat{df}(t)=-it\hat f(t)$ and hence 
$\hat\nu(t)=(1-it/k)\hat f(t)$. 
Also, $\widehat{e_k}(t)=1/(1-it/k)$ for real $t$. So, $\hat f=\hat\nu \widehat{e_k}$ and hence $f=\nu*e_k$. 

If now $\nu$ is a nonnegative measure, then it is a probability measure, because $(df)(\R)=0$. 
So, it follows from $f=\nu*e_k$ that the variance of $X$ is no less than the variance $1/k^2$ of the exponential distribution with mean $1/k$, and the inequality $\Var X\ge1/k^2$ is strict if the measure $\nu$ is nondegenerate. 
\end{proof}

\begin{proof}[Proof of Lemma~\ref{lem:k>1}]
Suppose that $B=\infty$. Then $k>0$ and for the signed measure $\nu$ as in Lemma~\ref{lem:perturb} we have 
\begin{equation}\label{eq:nu_k}
	\nu(dx)=\frac ck\,e^{-h(A+u)}\de_{-A}(dx)+c\frac{h+k}k\,e^{h(x-u)}\,1(-A<x<u)\,dx,
\end{equation}
where $\de_a$ is the Dirac probability supported on the singleton set $\{a\}$, assuming the first summand in \eqref{eq:nu_k} to be $0$ if $A=\infty$. 
Since $f\in\F_2\setminus(\F_1\cup\F_{2,\infty})$, we have $h>-k$ and $-A<u$. So, here $\nu$ is a nonnegative nondegenerate measure. Therefore, by Lemma~\ref{lem:perturb}, for any r.v.\ $X$ with p.d.f.\ $f$ we have $1=\Var X>1/k^2$. Recalling that $k>0$, we get $k>1$. 

Quite similarly or by reflection $x\leftrightarrow-x$, we get that $A=\infty$ implies $h>1$.
\end{proof}

Now Theorem~\ref{th:4 zeros} is completely proved. 

\section{Proofs of ``parametrization'' Theorems~\ref{th:par1} and \ref{th:par2}}\label{param-proofs}

%\vspace{2cm}

\begin{proof}[Proof of Theorem~\ref{th:par1}]
Take indeed any $f\in\F_1$, so that condition \eqref{eq:f_{A,B,c,h}} holds with $A,B,c,h$ as in \eqref{eq:A,B,c,h}. Recall that $A>0$ and $B>0$. 
%, because otherwise the barycenter of the probability distribution $\mu$ with density $f$ could not be $0$, violating the condition $\mu\in\L_1\subseteq\L$. 

Suppose now that $A$ and $B$ are real numbers. 
By strict convexity, we have $x^2<A^2+(B-A)(x+A)$ for all $x\in(-A,B)$. So, 
\begin{equation*}
	1=\int_\R x^2 f(x)\,dx<\int_\R (A^2+(B-A)(x+A)) f(x)\,dx=BA,
\end{equation*}
since $\int_\R x f(x)\,dx=0$. So, 
\begin{equation}\label{eq:A,B}
	AB>1,\quad A>0,\quad B>0. 
\end{equation}

Suppose further that $h\ne0$. Then the moments $m_j:=\int_\R x^j f(x)\,dx$ for $j=0,1,2$ are as follows: 
\begin{align*}
	m_0&=\frac ch \,(e^{B h}-e^{-A h}), \\ 
	m_1&=\frac c{h^2}\,\big((A h+1)e^{-A h}+(B h-1)e^{B h}\big), \\ 
	m_2&=\frac c{h^3}\,\big(-(A^2 h^2+2 A h+2)e^{-A h}+(B^2 h^2-2 B h+2)e^{B h}\big).  \end{align*}
Further details on the calculations in this proof are given in Mathematica notebook Proof\_Theorem1.7.nb. 
	
Consider the equalities $m_1=0=m_2-m_0$ as a homogeneous system
of two linear equations in $e^{-Ah}$ and $e^{Bh}$. Since this
system has a nonzero solution, its determinant
%Consider the equalities $m_1=0=m_2-m_0$ as a system of two equations for the ``unknowns'' $e^{-A h}$ and $e^{B h}$. Since the values of these ``unknowns'' are nonzero, the determinant 
\begin{equation*}
	\frac{c^2}{h^3}\,(A + B) \big(B-A + (A B-1) h\big)
\end{equation*}
%of this system 
must be $0$. By \eqref{eq:A,B,c,h}, $c\ne0$. So, in view of \eqref{eq:A,B}, $B-A + (A B-1) h=0$ and 
\begin{equation}\label{eq:h}
	h=\frac{A-B}{A B-1}. 
\end{equation}
%On the other hand, solving the system $m_0=1,m_1=0$ for the ``unknowns'' $e^{-A h}$ and $e^{B h}$, we have 
%\begin{equation}
%	(e^{-A h},e^{B h})=\frac1{c(A+B)}(1-Bh,1+Ah). 
%\end{equation}
%Substituting these expressions for $e^{-A h},e^{B h}$ into the 
Hence, 
\begin{equation*}
	0=m_1\frac{h^2}{ec}=(A h+1)e^{-A h-1}+(B h-1)e^{B h-1}
	=a e^{-a}-b e^{-b},
\end{equation*}
where 
\begin{equation}\label{eq:a,b:=}
	a:=Ah+1=\frac{A^2-1}{A B-1},\quad b:=1-Bh=\frac{B^2-1}{A B-1}.
\end{equation}
%If we had $a=0$, then t
The equality 
\begin{equation}\label{eq:a e^{-a}=b e^{-b}}
	a e^{-a}=b e^{-b}
\end{equation}
implies that either $a=b$ or $a,b>0$. If $a=b$, then, in view of \eqref{eq:A,B}, we would have $A=B$ and, by \eqref{eq:h}, $h=0$, which would contradict our current assumption $h\ne0$. So, currently we have $a,b>0$ and $a\ne b$. Now \eqref{eq:a,b:=} and \eqref{eq:A,B} imply 
\begin{equation*}
	A>1,\quad B>1. 
\end{equation*}

Let  
\begin{equation}\label{eq:r:=b-a}
	r:=b-a,
\end{equation}
so that $r\ne0$. Then \eqref{eq:a e^{-a}=b e^{-b}} can be rewritten as $a e^r=b$. Solving the system of linear equations $r=b-a$ and $a e^r=b$ for $a$ and $b$, we get 
\begin{equation}\label{eq:a(r),b(r)}
	a=\frac r{e^r-1},\quad b=\frac{r e^r}{e^r-1}.
\end{equation}
In particular, it follows that %(with $r\in\R\setminus\{0\}$) 
\begin{equation*}
	1-ab=\frac{u(r)}{(e^r-1)^2}>0,
\end{equation*}
with $u(r)$ as defined in \eqref{eq:u,s,si}. 

%would imply $b=0$, whence, in view of \eqref{eq:A,B}, we would have $A=B=1$ and hence $AB=1$, which contradict \eqref{eq:A,B}. So, $a\ne0$ and, similarly, $b\ne0$. 

Solving the equation $a=\frac{A^2-1}{A B-1}$ for $B$, we then get $B=(A^2+a-1)/(aA)$; using this expression for $B$, we have 
\begin{equation*}
	0=\Big(\frac{B^2-1}{A B-1}-b\Big)aA^2=(1-ab)A^2-(a-1)^2. 
\end{equation*}
So, 
\begin{equation}\label{eq:A=A(r)}
	A=\sqrt{\frac{(a-1)^2}{1-ab}}=\A(r). 
\end{equation}
Similarly, 
\begin{equation}\label{eq:B=B(r)}
	B=\B(r). 
\end{equation}
Now \eqref{eq:h} implies 
\begin{equation}\label{eq:h=h(r)}
	h=\h(r). 
\end{equation}
Solving 
\begin{equation*}
	1=m_0=\frac ch \,(e^{B h}-e^{-A h})
\end{equation*}
for $c$ and using \eqref{eq:A=A(r)}, \eqref{eq:B=B(r)}, and \eqref{eq:h=h(r)}, we also get 
\begin{equation}\label{eq:c=c(r)}
	c=\c(r). 
\end{equation}

These expressions for $A,B,c,h$ in terms of $r$ were obtained assuming that $A$ and $B$ are real numbers and $h\ne0$. 

If now $h=0$, then $h=\h(0)$ and $f$ is a p.d.f.\ of a uniform distribution over the interval $(-A,B)$. So, then $A$ and $B$ are real numbers. Moreover, the condition $f\in\F_1$ implies that $f\in\F$ and hence $A=B=\sqrt3=\A(0)=\B(0)$ and $c=\frac1{2\sqrt3}=\c(0)$. 

If $B$ is infinite, then, in view of the condition $B>0$, we have $B=\infty$. This implies that $h<0$ and hence $A\in(0,\infty)$. Recalling that $f\in\F_1\subseteq\F$, it is now easy to see that 
\begin{equation*}
	A=1=\A(\infty),\quad h=-1=\h(\infty),\quad c=\frac1e=\c(\infty),
\end{equation*}
and also $B=\infty=\B(\infty)$. 
Similarly, if $A$ is infinite, then
\begin{equation*}
	A=\infty=\A(-\infty),\quad B=1=\B(-\infty),\quad 
	h=1=\h(-\infty),\quad c=\frac1e=\c(-\infty).
\end{equation*}
Thus, representation \eqref{eq:par1} in terms of $r$ holds for any $f=f_{A,B,c,h}\in\F_1$. 

Moreover, in view of \eqref{eq:a,b:=}, \eqref{eq:r:=b-a}, and \eqref{r=0 or pm infty}, for each $f=f_{A,B,c,h}\in\F_1$, this representation holds for a unique value of $r\in\oR$. Indeed, in the case of finite $A,B$ and $h\ne0$, we have $r=b-a$, and $a$ and $b$ are determined by $A$ and $B$ -- whereas $h=0$, $B=\infty$, and $A=\infty$ correspond uniquely to $r=0,\infty,-\infty$, respectively.
 
Finally, by direct calculation, $f_{\A(r),\B(r),\c(r),\h(r)}\in\F_1$ for any $r\in\oR$. 
\end{proof}

\begin{proof}[Proof of Theorem~\ref{th:par2}] % \tail,A=B=infty.nb
Here for the moments $m_j:=\int_\R x^j f(x)\,dx$ we have 
\begin{equation*}
	1=m_0=c\frac{h+k}{hk}, \text{ whence } c=\frac{hk}{h+k};
\end{equation*}
\begin{equation*}
	0=\frac{m_1}{m_0}=u+\frac1k-\frac1h, \text{ whence } u=\frac1h-\frac1k; 
\end{equation*}
\begin{equation*}
	1=\frac{m_2}{m_0}=\frac{2}{h^2}-\frac{2 (k u+1)}{h k}+\frac{2}{k^2}+\frac{2 u}{k}+u^2=\frac1{h^2}+\frac1{k^2},
\end{equation*}
where the last equality follows upon substituting $\frac1h-\frac1k$ for $u$ in the preceding expression. 
These calculations were done in Mathematica notebook Proof\_Theorem1.8.nb. 

The last equality in \eqref{eq:par2} implies that $h>1$ and $k>1$ and uniquely determines either of $h,k$ in terms of the other: $k=\frac{h}{\sqrt{h^2-1}}$ and $h=\frac{k}{\sqrt{k^2-1}}$. Vice versa, if $c$ and $u$ are expressed in terms of real $h>1$ and $k>1$ as above and if $\frac1{h^2}+\frac1{k^2}=1$, then $\tf_{c,u,h,k}\in\F_{2,\infty}$. 
\end{proof}

\section{Proofs of Theorems~\ref{th:df bounds} and \ref{th:upper bound}} \label{main thms proofs}

\begin{remark}\label{rem:e-poly}
Let us say that a function $\FF\colon I\to\R$ on an interval $I\subseteq\R$ is an E-polynomial (over $\Z$) if $\FF(x)=P(x,e^x)$ for some nonzero polynomial $P(X,Y)\in\Z[X,Y]$ and all $x\in I$. So, with $\EE(x):=e^x$ for $x\in I$, 
\begin{equation*}%\label{eq:E-poly}
	\FF=\sum_{j=0}^k \p_j \EE^j
\end{equation*}
for some nonnegative integer $k$ (which will be referred to as the E-degree of $\FF$) and some polynomials $\p_j(X)\in\Z[X]$, with $\p_k(X)\ne0$. 

We will want to identify the sign patterns of various E-polynomials. 
%, in our case with $p(x)=x$ for all $x$. 

If $k\ge1$ and the polynomial $\p_0$ above is of degree $d_0$ (so that $d_0=-1$ if $\p_0=0$), then $\FF_1:=\FF^{(d_0+1)}/\EE$ is an E-polynomial of E-degree $k-1$, equal to $\FF^{(d_0+1)}$ in sign. Repeating this reduction in the E-degree, we finally obtain an E-polynomial of E-degree $0$, which is just a usual polynomial over $\Z$. 

If now $I$ is an open interval and $G\colon I\to\R$ is any continuously differentiable function with finitely many zeros of $G'$, then the sign pattern of $G'$ determines the monotonicity pattern of $G$. 

If, in addition, $G(x)\ne0$ for any zero $x$ of $G'$, then the sign pattern of $G'$, together with the signs of $G$ at the zeros of $G'$ and near the endpoints of the interval $I$, will determine the sign pattern of $G$ on $I$. The potentially difficult case is when $G$ has a multiple root $x\in I$, so that $G(x)=G'(x)=0$. 

If such a potentially difficult case is not encountered during the mentioned E-degree reduction process of an E-polynomial $\FF$, then we get an algorithm for the determination of the sign pattern of $\FF$. 

Actually, such an algorithm exists for all E-polynomials \cite{barbagallo_etal}. 

As an illustration of the E-degree reduction algorithm let us show that the E-polynomial function $u$ (of E-degree $2$%, with $p(x)=x$ for all real $x$
) defined by \eqref{eq:u,s,si} is $>0$ on $\R\setminus\{0\}$. Then $u_1(r):=u'(r)/e^r=-2  - 2 r - r^2+ 2 e^r$ for real $r$, so that $u_1$ is an E-polynomial function of E-degree $1$. Next, $u_2(r):=u_1'''(r)/e^r=2$ for real $r$, so that $u_2$ is an E-polynomial function of E-degree $0$, and $u_1'''>0$. 
Also, $u_1(0)=u'_1(0)=u''_1(0)=0$. So, $u_1$ is $-+$ on $\R$, and hence $u'$ is $-+$ on $\R$, so that $u$ is $\d\u$ on $\R$. It remains to note that $u(0)=0$. 

The method described in this remark may be called repeated hammering. Each round of this process consists of several ``hammer strikes'' by differentiation followed by a ``reset'' consisting in the division by $\EE(x)=e^x$; the number of strikes in each round is 
$$1+(\text{the degree of the polynomial currently playing the role of the polynomial $\p_0$}).$$ 
As noted here, the result of such a repeated hammering is a usual polynomial over $\Z$.    
\qed
\end{remark}

\begin{proof}[Proof of Lemma~\ref{lem:B}]
For real $r\ne0$, 
\begin{equation}\label{eq:DB}
	(D\B)(r):=\B'(r)\frac{2u(r)^{3/2}}{e^r}=r^2+4 r+2-e^r (r^3-r^2+4 r+4)+2 e^{2 r}. 
\end{equation}
The function $D\B$ is an E-polynomial, of E-degree $2$. So, in view of Remark~\ref{rem:e-poly}, the sign pattern of $D\B$ can be determined automatically. 
For this purpose, we can use the Mathematica command \texttt{Reduce}. 
This is from the description of this command at \url{https://reference.wolfram.com/language/ref/Reduce.html}: 
``Reduce[\emph{expr},\emph{vars}]
reduces the statement \emph{expr} by solving equations or inequalities for \emph{vars} and eliminating quantifiers.'' and 
``The result of Reduce[\emph{expr},\emph{vars}] always describes exactly the same mathematical set as \emph{expr}.''

Command \texttt{Reduce[DB[r] <= 0]} produces \texttt{r == 0}. So, $(D\B)(r)>0$ and hence $\B'(r)>0$ for real $r\ne0$. Also, the function $\B$ is continuous on $\oR$ and the function $\B'$ is continuous on $\R$, with $\B'(0)=\frac1{2\sqrt3}$. This completes the proof of Lemma~\ref{lem:B}. 

Another proof of the fact that $\B'(r)>0$ for real $r\ne0$ will be given in Appendix~\ref{app:B}, using repeated hammering described in Remark~\ref{rem:e-poly}. 
\end{proof}

\begin{proof}[Proof of Proposition~\ref{prop:F1}] % \\checking\\Prop 1.11 prop(colon)F1.nb  
This proof uses so-called l'Hospital-type rules for monotonicity, including their special-case, general, and derived special-case varieties \cite{pin06}. 

In view of \eqref{eq:p_{1,b}=}, \eqref{eq:r_b}, and the definition of $\B(r)$ in \eqref{eq:A,B,c,h(r)}, for all real $r\ge r_b$ 
\begin{equation}\label{eq:p_1=}
	p_1(r)=p_{1,b}(r)=\frac{F(r)}{G(r)},
\end{equation}
%
%\begin{equation}
%	p_1(r)=p_{1,b}(r)=
%	\begin{cases}
%	\tp(r) &\text{ if }0<b\le1, \\ 
%	\tp(r)_+=\tp(r)\,1(r\ge r_b) &\text{ if }b>1, 
%	\end{cases}
%\end{equation}
with the latter ratio continuously extended to $r=0$, where  
%\begin{equation}\label{eq:p_1=}
%	\tp(r):=\tp_b(r):=\frac{F(r)}{G(r)},
%\end{equation} 
\begin{equation}\label{eq:F,g,G}
	F(r):=F_b(r):=e^{g(r)}-1,\quad g(r):=g_b(r):=\frac{s(r)-b\sqrt{u(r)}}{e^r-1},\quad G(r):=e^r-1;  
\end{equation}
here and in the rest of this proof, by default $r$ denotes an arbitrary nonzero real number. % that is $\ge r_b$. 
Consider the ``derivative ratio'' 
\begin{equation*}
	\rho(r):=\rho_b(r):=\frac{F'(r)}{G'(r)},
\end{equation*}
and then 
\begin{equation}\label{eq:Drho}
	(D\rho)(r):=(D\rho)_b(r):=\rho'(r)4e^{-g(r)}(e^r-1)^4/u(r)
\end{equation}
and 
\begin{equation*}
	(DD\rho)(r):=(DD\rho)_b(r):=(D\rho)'(r)\frac2b\,u(r)^{7/2},
\end{equation*}
so that $(D\rho)(r)$ equals $\rho'(r)$ in sign and $(DD\rho)(r)$ equals $(D\rho)'(r)$ in sign. 

Note that  
\begin{equation}\label{eq:DDrho=}
	(DD\rho)(r)=C_1(r)+2b e^r \sqrt{u(r)}\,r\,C_2(r),
\end{equation}
where $C_1$ and $C_2$ are E-polynomials, of respective E-degrees $7$ and $4$. 

Using the \texttt{Reduce} command again, we see that 
\begin{equation}\label{eq:C1,C2<0}
	\text{$C_1(r)<0$ and $C_2(r)<0$ (again, for nonzero real $r$). } 
\end{equation}
So, $(DD\rho)(r)<0$ for real $r>0$. 

A repeated-hammering proof of \eqref{eq:C1,C2<0} will be given in Appendix~\ref{app:C1,C2<0}. 

If now $r<0$ and $b\le\B(r)$, then 
\begin{equation*}
	(DD\rho)(r)\le C_1(r)+2\B(r) e^r \sqrt{u(r)}\,r\,C_2(r)=(DD\rho)_{\B(r)}(r). 
\end{equation*}
The ratio $(DD\rho\B)(r):=\dfrac{(DD\rho)_{\B(r)}(r)}{1-e^r}$ is an E-polynomial, of E-degree $6$, and 
\begin{equation}\label{eq:DDrhoB(r)<0}
	(DD\rho\B)(r)<0\text{ if }r<0.
\end{equation}
A repeated-hammering proof of \eqref{eq:DDrhoB(r)<0} will be given in Appendix~\ref{app:DDrhoB}. 

Thus, $(DD\rho)(r)<0$ and hence $(D\rho)'(r)<0$ for all real $r\ne0$ such that $b\le\B(r)$, that is, for all real $r\ne0$ such that $r\ge r_b$. 
%Also, $(D\rho)(r)\to4(b^2-1)$ as $r\to0$.
 
It follows that $D\rho$ is $\d$ on $(-\infty,0)\cap[r_b,\infty)$ and on $(0,\infty)\cap[r_b,\infty)$. Also, 
\begin{equation}\label{eq:Drho special}
	(D\rho)(-\infty+)=\infty.\quad (D\rho)(0\pm)=4(b^2-1),\quad (D\rho)(\infty-)=-\infty
\end{equation} 
and 
\begin{equation*}
\rho'(0\pm)=\frac{b^2-1}{12},\quad	p_1(0)=p_1(0\pm)=\rho(0\pm)_+=\frac{(\sqrt3-b)_+}{2\sqrt3}. 
\end{equation*}
So, $D\rho$ is $\d$ on $[r_b,\infty]$, if $D\rho$ is extended to $\oR$ by continuity -- which will be henceforth assumed.

In particular, if $b\le1$ (so that $0<b\le1$), then $r_b=-\infty$ and hence $D\rho$ is $\d$ on $\oR$. %, if $D\rho$ is extended to $\oR$ by continuity. % $(-\infty,0)$ and on $(0,\infty)$. 
So, $D\rho$ is $+-$ on $\oR$ and hence, in view of \eqref{eq:Drho}, $\rho$ is $\u\d$ on $\R$.  

Consider now the case $b>1$. 
Note that 
\begin{equation}\label{eq:DrhoB:=}
(D\rho\B)(r):=
(D\rho)_{\B(r)}(r)\frac{u(r)^3}{(e^r-1)^2}
\end{equation}
is an E-polynomial, of E-degree $4$,   
and for real $r\ne0$
\begin{equation}\label{eq:DrhoB}
	(D\rho\B)(r)>0\iff r<r_*,
\end{equation}
%for real $r\ne0$ this E-polynomial is $>0$ iff $r<r_*$, 
where $r_*$ is a certain real number; numerically, 
\begin{equation*}
	r_*\ap3.29. 
\end{equation*}
%So, $(DD\rho)(r)<0$ for real $r>0$. 

A repeated-hammering proof of \eqref{eq:DrhoB} will be given in Appendix~\ref{app:DrhoB}, where it will also be shown that $r_*$ can be effectively computed to any precision. 

Let now 
\begin{equation}\label{eq:b_*}
	b_*:=\B(r_*)\ap3.22. 
\end{equation}
It also follows that $b_*$ can be effectively computed to any precision.

For real $r\ne0$ we have $(D\rho)_{\B(r)}(r)>0$ iff $r<r_*$; that is, in view of Lemma~\ref{lem:B},  
for all $b\in(1,\infty)\setminus\{\sqrt3\}$ we have 
$(D\rho)_b(r_b)>0$ iff $b<b_*$. 
If now $b=\sqrt3$, then $r_b=0$ and, in view of \eqref{eq:Drho special} and $D\rho$ having been extended to $\oR$ by continuity,   $(D\rho)_b(r_b)=(D\rho)_b(r_b\pm)=4(b^2-1)=8>0$. So, for all $b\in(1,\infty)$ we have 
$(D\rho)_b(r_b)>0$ iff $b<b_*$. 
Also by \eqref{eq:Drho special}, $(D\rho)(\infty-)=-\infty$. 
 
Therefore, if $b<b_*$ (and still assuming that $b>1$), then $(D\rho)_b$ is $+-$ and hence $\rho_b$ is $\u\d$ on $[r_b,\infty)$; and if $b\ge b_*$, then $(D\rho)_b<(D\rho)_b(r_b)\le0$ on $(r_b,\infty)$ and hence $\rho_b$ is $\d$ on $[r_b,\infty)$. 

%Note also that 
%\begin{equation}
%	p_1(-\infty)=p_1(-\infty+)=e^{-1-b}
%\end{equation}

To complete the proof of Proposition~\ref{prop:F1}, we will consider the following five cases separately. 

\emph{Case~1: $b<1$.} Then, as was noted, $\rho$ is $\u\d$ on $\R$ and $\rho'(0\pm)=\frac{b^2-1}{12}<0$. %Also, $\rho'(r)\sim
So, $\rho$ is $\u\d$ on $(-\infty,0)$ and $\d$ on $(0,\infty)$. So, by \cite[line~3 of Table~4.1]{pin06}, $p_1$ is $\d$ or $\u\d$ on $(-\infty,0)$. But $e^{-r} p'_1(r)\to\infty>0$ as $r\to-\infty$. So, $p_1$ is $\u\d$ on $(-\infty,0)$. By \cite[Proposition 4.1]{pin06}, $p_1$ is $\d$ on $(0,\infty)$. Therefore and by continuity, in Case~1 $p_1$ is $\u\d$ on $\oR=[r_b,\infty]$. 

\emph{Case~2: $b=1$.} \label{case2} Then again, $\rho$ is $\u\d$ on $\R$. Also, then $\rho'(0\pm)=0$ and $\rho''(0\pm)=-\frac{11}{120\sqrt3}<0$. %Also, $\rho'(r)\sim
So, $\rho$ is $\u$ on $(-\infty,0)$ and $\d$ on $(0,\infty)$. So, by \cite[Proposition 4.1]{pin06}, $p_1$ is $\u$ on $(-\infty,0)$ and $\d$ on $(0,\infty)$. Therefore and by continuity, in Case~2 $p_1$ is $\u\d$ on $\oR=[r_b,\infty]$. 

\emph{Case~3: $1<b\le\sqrt3$.} Then $\rho$ is $\u\d$ on $[r_b,\infty)$. Also, here $r_b\le0$ and $\rho'(0\pm)>0$. So, $\rho$ is $\u$ on $[r_b,0]$ and hence, by \cite[Proposition 4.1]{pin06}, $p_1$ is $\u$ on $[r_b,0]$. 
On $[0,\infty)$, $\rho$ is $\u\d$, so that, by \cite[line~1 of Table~4.1]{pin06}, $p_1$ is $\u$ or $\u\d$; also, $e^r p'_1(r)\to-\infty<0$ as $r\to\infty$, which excludes the pattern $\u$. So, in Case~3 as well, $p_1$ is $\u\d$ on $[r_b,\infty]$.

\emph{Case~4: $\sqrt3<b<b_*$.} Then $\rho$ is $\u\d$ on $[r_b,\infty)$. Also, here $r_b>0$. % and $\rho$ is $\u\d$ on $[r_b,\infty)$. 
So, by \cite[line~1 of Table~1.2]{pin06}, $p_1$ is $\u$ or $\d$ or $\u\d$ or $\d\u$ or $\d\u\d$ on $[r_b,\infty)$. But $p_1(r_b)=0<p_1(r)$ for $r>r_b$, which excludes the patterns $\d$, $\d\u$, and $\d\u\d$. Also, $e^r p'_1(r)\to-\infty<0$ as $r\to\infty$. So, in Case~4 as well, $p_1$ is $\u\d$ on $[r_b,\infty]$.

\emph{Case~5: $b\ge b_*$.} Then $r_b>0$ and $\rho$ is $\d$ on $[r_b,\infty)$. So, by \cite[line~2 of Table~1.1]{pin06}, $p_1$ is $\u$ or $\d$ or $\u\d$ on $[r_b,\infty)$. But $p_1(r_b)=0<p_1(r)$ for $r>r_b$, which excludes the pattern $\d$. Also, $e^r p'_1(r)\to-\infty<0$ as $r\to\infty$. So, in Case~5 as well, $p_1$ is $\u\d$ on $[r_b,\infty]$. 

This proves the first sentence in the statement of Proposition~\ref{prop:F1}. 

If $b\le1$, then $r_b=-\infty$ and hence the second sentence in the statement of Proposition~\ref{prop:F1} is vacuous. If now $b>1$, then, by Remark~\ref{rem:iB}, the condition $-\infty\le r<r_b$ is equivalent to $b>\B(r)$, so that $-\infty\le r<r_b$ implies $p_1(r)=0$, in view of \eqref{eq:p_{1,b}} and \eqref{eq:f_{A,B,c,h}}. This proves the second sentence in the statement of Proposition~\ref{prop:F1}. 

The third sentence in the statement of Proposition~\ref{prop:F1} follows from the first two sentences, in view of \eqref{eq:p_{1,b}=}.

This completes the proof of Proposition~\ref{prop:F1}. 
\end{proof}

%\begin{proposition}\label{prop:F1}
%The function $p_1$ %, defined by \eqref{eq:p_{1,b}}, 
%is $\u\d$ on $[r_b,\infty]$. 
%Also, $p_1(r)=0$ if $-\infty\le r<r_b$. 
%So, $\sup_{r\in\oR}p_1(r)$ is attained at a unique value of $r$, which is actually in $\R$, and  
%this unique maximizer of $p_1(r)$ in $r$ can be effectively computed to any precision. 
%\end{proposition}

\begin{proof}[Proof of Lemma~\ref{lem:b_{**}}]
Note that, for $b\in[1/\sqrt2,1)$,   
\begin{equation*}
	\rat(b)=\frac{1-\sqrt{2 b^2-1}}{2}\, e^{b+\sqrt{2 b^2-1}}
\end{equation*}
and 
\begin{gather*}
\rat'(b)=\frac{1-2b-\sqrt{2 b^2-1}}{2} \, e^{b+\sqrt{2 b^2-1}}<0, 
\end{gather*}
so that the continuous function $\rat$ is decreasing on $[1/\sqrt2,1)$, from $e^{1/\sqrt2}/2=1.0140\ldots$ to $0$. 
\end{proof}

According to Theorem~\ref{th:par2}, the densities $f\in\F_{2,\infty}$ can be parametrized by the one parameter 
\begin{equation*}
	s:=\frac1h\in(0,1), 
\end{equation*}
so that for $c,u,h,k$ as in \eqref{eq:par2} we have 
\begin{equation}\label{eq:tc etc}
	c=\tc(s):=\frac1{s+t},\quad u=s-t,\quad h=\frac1s,\quad k=\frac1t,
\end{equation}
where 
\begin{equation}\label{eq:t}
	t=%\thh(s):=
	\sqrt{1-s^2};
\end{equation}
the latter convention on $t$ will be assumed by default. 

To prove Theorem~\ref{th:df bounds}, we will also need the following proposition. 

\begin{proposition}\label{prop:F2}
For real $b>0$ and $s\in(0,1)$, let 
\begin{equation}\label{eq:p_{2,b}:=}
	%p_2(s):=
	p_{2,b}(s):=\int_b^\infty \tf_{\tc(s),s-t,1/s,1/t}(x)\,dx, 
\end{equation}
with $\tf_{c,u,h,k}$ defined in \eqref{eq:f_{c,u,h,k}}. 
Then 
\begin{equation}\label{eq:p_{2,b}=}
	p_{2,b}(s)=\begin{cases}
	p_{21,b}(s):=\frac{t }{s+t}\,e^{-\frac{b+t-s}{t}} &\text{ if } s-t\le b, \\ 
	p_{22,b}(s):=1-\frac{s }{s+t}\,
 e^{\frac{b+t-s}{s}}&\text{ if }s-t\ge b  
	\end{cases}
\end{equation}
and, with $p_{1,b}(r)$ as defined in Proposition~\ref{prop:F1}, 
\begin{equation}\label{eq:S_2 le S_1}
	S_2(b):=\sup_{s\in(0,1)}p_{2,b}(s)\le S_1(b):=\max_{r\in\oR}p_{1,b}(r)
\end{equation}
and 
\begin{equation}\label{eq:I_2 le I_1}
	I_2(b):=\inf_{s\in(0,1)}p_{2,b}(s)\le I_1(b):=\min_{r\in\oR}p_{1,b}(r). 
\end{equation}
Moreover, 
\begin{equation}\label{eq:I_1 ge I_2}
\begin{aligned}
I_1(b)=\min(p_{1,b}(-\infty),p_{1,b}(\infty))&=\min(1-e^{-(1-b)_+},e^{-1-b}) \\ 
&=\min(p_{2,b}(1-),p_{2,b}(0+))\ge I_2(b). 
\end{aligned}	
\end{equation}
%where $z_+:=\max(0,z)$ for real $z$. 
Finally, for all real $b>0$, 
\begin{equation}\label{eq:I_2}
	I_2(b)=J_2(b),
\end{equation}
with $J_2(b)$ as defined in \eqref{eq:J_2}.  
\end{proposition}

%??? For an illustration of Proposition~\ref{prop:F2}, here is the graph $\{(b,I_2(b))\colon0<b<1\}$, with the red and green parts coming, in view of \eqref{eq:I_2}, from the respective first two cases in \eqref{eq:J_2}: % A=B=infty\A=B=infty(2)-scr1.nb
%\begin{figure}[htbp]
%	\centering
%		\includegraphics[width=.6\textwidth]{pic.pdf}
%	\caption{Graph of $I_2$ over the interval $(0,1)$.}
%	\label{fig:pic}
%\end{figure}

\begin{proof}[Proof of Proposition~\ref{prop:F2}]
First of all, \eqref{eq:p_{2,b}=} follows immediately from the definitions of $p_{2,b}(s)$ in \eqref{eq:p_{2,b}:=}, of $\tf_{c,u,h,k}(x)$ in \eqref{eq:f_{c,u,h,k}}, of $\tc(s)$ in \eqref{eq:tc etc}, and of $t$ in \eqref{eq:t}. % pr2 in A=B=infty\A=B=infty(2).nb 

The attainment of the minimum in the definition of $I_1(b)$ in \eqref{eq:I_2 le I_1} follows from the first equality in \eqref{eq:I_1 ge I_2} -- which in turn follows from Proposition~\ref{prop:F1}. 
%, because, in view of the definition of $r_b$ in \eqref{eq:r_b}, we have $p_{1,b}(r)=0$ if $b\ge1$ and $r\le r_b$. 

The second equality in \eqref{eq:I_1 ge I_2} follows immediately from %\eqref{eq:p_{1,b}=} and 
\eqref{eq:p_{1,b} special}.
%the definitions of $p_{1,b}(r)$ in \eqref{eq:p_{1,b}}, of $f_{A,B,c,h}(x)$ in \eqref{eq:f_{A,B,c,h}}, and the relevant definitions in \eqref{r=0 or pm infty}. 
%In particular, for all real $b>0$, 
%\begin{equation}\label{eq:pm infty}
%	(p_{1,b}(-\infty),p_{1,b}(\infty))=(1-e^{-(1-b)_+},e^{-1-b}). 
%\end{equation}

The third equality in \eqref{eq:I_1 ge I_2} follows from \eqref{eq:p_{2,b}=} and 
the definition of $t$ in \eqref{eq:t}. 

The inequality in \eqref{eq:I_1 ge I_2} follows immediately from 
the definition of $I_2(b)$ in \eqref{eq:I_2 le I_1}. 

It remains to verify 
\begin{enumerate}[(i)]
\item the inequality in \eqref{eq:S_2 le S_1}; 
%	\item the last two lines in the statement of Proposition~\ref{prop:F2}, which would provide for the correctness of the definition of $b_{**}$ as the unique root of equation \eqref{eq:rat} in the interval $[1/\sqrt2,1)$; 
	\item the equality in \eqref{eq:I_2}. 
\end{enumerate}

Let us begin with task (ii). For any $s\in(0,1)$ and $t$ as in \eqref{eq:t}, 
\begin{equation*}
	s-t\le b\iff
	\left\{
	\begin{alignedat}{2}
	& s\in(0,s_1(b)] &&\text{ if }0<b<1, \\
	& s\in(0,1) &&\text{ if }b\ge1,
	\end{alignedat}
	\right.
\end{equation*}
where 
\begin{equation*}
	s_1(b):=\frac{b+\sqrt{2-b^2}}2.
\end{equation*}

Also, 
\begin{equation*}
	p_{21,b}'(s)=(1-b(s+\sqrt{1 - s^2}))
	\frac{s}{t^2(s+t)^2}\,e^{-\frac{b+t-s}{t}},
\end{equation*}
which equals $1-b(s+\sqrt{1 - s^2})$ in sign. It follows that, for any $s\in(0,1)$ such that $s-t\le b$, 
\begin{equation}\label{p_{21,b}'(s)}
	p_{21,b}'(s)\ge0\iff
	\left\{
	\begin{alignedat}{2}
	& s\in(0,s_1(b)] &&\text{ if }0<b\le\frac1{\sqrt2}, \\
	& s\in(0,s_2(b)]\cup[s_3(b),s_1(b)] &&\text{ if }\frac1{\sqrt2}<b<1, \\ 
	& s\in\emptyset &&\text{ if }b\ge1,
	\end{alignedat}
	\right.
\end{equation}
where 
\begin{equation*}
	s_2(b):=\frac{1-\sqrt{2 b^2-1}}{2 b}\quad\text{and}\quad 
	s_3(b):=\frac{1+\sqrt{2 b^2-1}}{2 b}. 
\end{equation*}
So, 
\begin{equation}\label{eq:max p21}
	\sup_{0<s<1,\,s-t\le b}\, p_{21,b}(s)=
	\left\{
	\begin{alignedat}{2}
	& p_{21,b}(s_1(b)) &&\text{ if }0<b\le\frac1{\sqrt2}, \\
	& p_{21,b}(s_2(b))\vee p_{21,b}(s_1(b)) &&\text{ if }\frac1{\sqrt2}<b<1, \\ 
	& p_{21,b}(0+)=e^{-1-b} &&\text{ if }b\ge1,
	\end{alignedat}
	\right.
\end{equation}
\begin{equation}\label{eq:min p21}
	\inf_{0<s<1,\,s-t\le b}\, p_{21,b}(s)=
	\left\{
	\begin{alignedat}{2}
	& p_{21,b}(0+)=e^{-1-b} &&\text{ if }0<b\le\frac1{\sqrt2}, \\
	& p_{21,b}(0+)\wedge p_{21,b}(s_3(b))
 &&\text{ if }\frac1{\sqrt2}<b<1, \\ 
	& p_{21,b}(1-)=0 &&\text{ if }b\ge1; 
	\end{alignedat}
	\right.
\end{equation}
here, as usual, $\al\vee\be:=\max(\al,\be)$ and $\al\wedge\be:=\min(\al,\be)$. So, we have determined the extrema of $p_{21,b}$. 
%Actually, \eqref{eq:min p21} is given here only for completeness (to match \eqref{eq:max p21}) and will not be needed in our proofs.  

Let us now determine the extrema of $p_{22,b}$. For any $s\in(0,1)$ and $t$ as in \eqref{eq:t}, 
\begin{equation*}
	s-t\ge b\iff
	\left\{
	\begin{alignedat}{2}
	& s\in[s_1(b),1) &&\text{ if }0<b<1, \\
	& s\in\emptyset &&\text{ if }b\ge1. 
	\end{alignedat}
	\right.
\end{equation*}

Also, for $b\in(0,1)$, 
\begin{equation*}
	p_{22,b}'(s)=(1+b(s+\sqrt{1 - s^2}))
	\frac1{s(s+t)^2}\,e^{\frac{b+t-s}s}>0
\end{equation*}
for all $s\in[s_1(b),1)$. 
So, for $s\in(0,1)$ such that $s-t\ge b$, we have $b<1$, 
\begin{equation}\label{eq:max p22}
	\sup_{0<s<1,\,s-t\ge b}\, p_{22,b}(s)=p_{22,b}(1-)=1-e^{b-1}, 
%	\left\{
%	\begin{alignedat}{2}
%	& p_{22,b}(1-) &&\text{ if }0<b\le\frac1{\sqrt2}, \\
%	& p_{22,b}(1-) &&\text{ if }\frac1{\sqrt2}<b<1, 
%	\end{alignedat}
%	\right.
\end{equation}
\begin{equation}\label{eq:min p22}
	\inf_{0<s<1,\,s-t\ge b}\, p_{22,b}(s)=p_{22,b}(s_1(b))=p_{21,b}(s_1(b)); 
%	\left\{
%	\begin{alignedat}{2}
%	& p_{22,b}(s_1(b)) &&\text{ if }0<b\le\frac1{\sqrt2}, \\
%	& p_{22,b}(s_1(b)) &&\text{ if }\frac1{\sqrt2}<b<1;
%	\end{alignedat}
%	\right.
\end{equation}
the latter equality holds, in view of the definitions of $p_{21,b}(s)$ and $p_{22,b}(s)$ in \eqref{eq:p_{2,b}=}, because for $s\in(0,1)$ we have $s=s_1(b)\iff s-t=b$.  
So, we have determined the extrema of $p_{22,b}$ as well. 

Now, in view of \eqref{eq:min p21}, \eqref{eq:min p22}, and 
the definition of $I_2(b)$ in \eqref{eq:I_2 le I_1}, 
\begin{equation*}%\label{eq:I_2(b)=}
	I_2(b)=
	\left\{
	\begin{alignedat}{2}
	& p_{21,b}(0+)\wedge p_{21,b}(s_1(b)) &&\text{ if }0<b\le\frac1{\sqrt2}, \\
	& p_{21,b}(0+)\wedge p_{21,b}(s_3(b))\wedge p_{21,b}(s_1(b))
 &&\text{ if }\frac1{\sqrt2}<b<1, \\ 
	& 0 &&\text{ if }b\ge1.  
	\end{alignedat}
	\right.
\end{equation*}
Looking back at \eqref{p_{21,b}'(s)}, we see that $p_{21,b}(0+)\le p_{21,b}(s_1(b))$ if $0<b\le\frac1{\sqrt2}$, and \break 
$p_{21,b}(s_3(b))\le p_{21,b}(s_1(b))$ if $\frac1{\sqrt2}<b<1$. Therefore and because $p_{21,b}(0+)=e^{-1-b}$ and $p_{21,b}(s_3(b))=\frac{1}{2} \left(1-\sqrt{2 b^2-1}\right) e^{\sqrt{2 b^2-1}-1}$, the above expression for $I_2(b)$ simplifies as follows: 
\begin{equation*}%\label{eq:I_2(b)=}
	I_2(b)=
	\left\{
	\begin{alignedat}{2}
	& e^{-1-b} &&\text{\quad if }0<b\le\frac1{\sqrt2}, \\
	& e^{-1-b}\wedge \Big(\frac{1-\sqrt{2 b^2-1}}{2} e^{\sqrt{2 b^2-1}-1} \Big)
 &&\text{\quad if }\frac1{\sqrt2}<b<1, \\ 
	& 0 &&\text{\quad if }b\ge1.  
	\end{alignedat}
	\right.
\end{equation*}
Thus, to complete task (ii), of proving the equality in \eqref{eq:I_2}, it remains to recall Lemma~\ref{lem:b_{**}}. 
%the completed verification, in task (ii), of the last two lines in the statement of Proposition~\ref{prop:F2}. 

Finally, let us carry out task (i), of proving the inequality in \eqref{eq:S_2 le S_1}. Using the definition of $S_2(b)$ in \eqref{eq:S_2 le S_1}, the expression of $p_{2,b}(s)$ in \eqref{eq:p_{2,b}=}, and formulas \eqref{eq:max p21} and \eqref{eq:max p22}, 
we have 
\begin{equation}\label{eq:S_2(b)=}
	S_2(b)=
	\left\{
	\begin{alignedat}{2}
	& p_{21,b}(s_1(b))\vee(1-e^{b-1}) &&\text{ if }0<b\le\frac1{\sqrt2}, \\
	& p_{21,b}(s_2(b))\vee p_{21,b}(s_1(b)) \vee(1-e^{b-1}) 
	&&\text{ if }\frac1{\sqrt2}<b<1, \\ 
	& e^{-1-b} &&\text{ if }b\ge1. 
	\end{alignedat}
	\right.
\end{equation}
Note that 
\begin{equation}\label{eq:p_{21,b}(s_1(b))}
	p_{21,b}(s_1(b))=\frac{1}{2}-\frac{b}{2 \sqrt{2-b^2}}, 
\end{equation}
which clearly decreases in $b\in(0,1)$. 
Next, for $b\in(\frac1{\sqrt2},1)$, 
\begin{equation}\label{p_{21,b}(s_2(b))}
	p_{21,b}(s_2(b))=\frac12\, e^{-1 - \sqrt{2b^2-1}}\, (1 + \sqrt{2b^2-1}), 
\end{equation}
which decreases in $b\in(\frac1{\sqrt2},1)$, being the composition of the decreasing function $[0,\infty)\ni\ga\mapsto\frac12\, e^{-1 - \ga} (1 + \ga)$ and the increasing function $(\frac1{\sqrt2},1)\ni b\mapsto\sqrt{2b^2-1}$. 

Also, the function $S_1$ is nonincreasing, being the pointwise supremum of nonincreasing functions. 

Using the mentioned nonincreasing patterns of $p_{21,b}(s_1(b))$, $p_{21,b}(s_2(b))$, and $S_1(b)$, we have the following. 

For $b\in(0,3/10]$, 
\begin{equation*}
	S_1(b)\ge S_1(3/10)\ge p_{1,3/10}(-64/10)>0.504>1/2
	%=p_{21,0+}(s_1(0+))\ge 
	>p_{21,b}(s_1(b)) 
\end{equation*}
and, by \eqref{eq:p_{1,b} special}, $S_1(b)\ge1-e^{b-1}$. So, by \eqref{eq:S_2(b)=}, the inequality in \eqref{eq:S_2 le S_1} holds for $b\in(0,3/10]$. 

For $b\in(3/10,51/100]$, 
\begin{equation*}
	S_1(b)\ge S_1(51/100)\ge p_{1,51/100}(-33/10)>0.40>p_{21,3/10}(s_1(3/10))\ge p_{21,b}(s_1(b)) 
\end{equation*}
and, again by \eqref{eq:p_{1,b} special}, $S_1(b)\ge1-e^{b-1}$. So, by \eqref{eq:S_2(b)=}, the inequality in \eqref{eq:S_2 le S_1} holds for $b\in(3/10,51/100]$.   

For $b\in(51/100,1/\sqrt2]$, 
\begin{equation*}
	S_1(b)\ge S_1(1/\sqrt2)\ge p_{1,1/\sqrt2}(-16/10)>0.31>p_{21,51/100}(s_1(51/100))\ge p_{21,b}(s_1(b)) 
\end{equation*}
and, again by \eqref{eq:p_{1,b} special}, $S_1(b)\ge1-e^{b-1}$. So, by \eqref{eq:S_2(b)=}, the inequality in \eqref{eq:S_2 le S_1} holds for $b\in(51/100,1/\sqrt2]$. 

Thus, the inequality in \eqref{eq:S_2 le S_1} holds for all $b\in(0,1/\sqrt2]$.   

Recalling \eqref{eq:p_{1,b} special}, \eqref{eq:p_{21,b}(s_1(b))}, and \eqref{p_{21,b}(s_2(b))}, we have 
\begin{equation*}
	p_{1,1}(0)=\frac{1}{2}-\frac{1}{2 \sqrt{3}}=p_{21,1/\sqrt2}(s_1(1/\sqrt2))
	>\frac{1}{2 e}=p_{21,1/\sqrt2}(s_2(1/\sqrt2)).
\end{equation*}
So, for $b\in(1/\sqrt2,1)$, 
\begin{multline*}
	S_1(b)\ge S_1(1)\ge p_{1,1}(0)=p_{21,1/\sqrt2}(s_1(1/\sqrt2)) \\ 
	=p_{21,1/\sqrt2}(s_1(1/\sqrt2))\vee p_{21,1/\sqrt2}(s_2(1/\sqrt2))\ge p_{21,b}(s_1(b))\vee p_{21,b}(s_2(b)) 
\end{multline*}
and, again by \eqref{eq:p_{1,b} special}, $S_1(b)\ge1-e^{b-1}$. So, by \eqref{eq:S_2(b)=}, the inequality in \eqref{eq:S_2 le S_1} still holds for $b\in(1/\sqrt2,1)$. 

For $b\ge1$, 
\begin{equation*}
	S_1(b)\ge p_{1,b}(\infty)=e^{-1-b}, 
\end{equation*}
once again by \eqref{eq:p_{1,b} special}. Therefore, 
by \eqref{eq:S_2(b)=}, the inequality in \eqref{eq:S_2 le S_1} holds for $b\ge1$ as well. 

We conclude that the inequality in \eqref{eq:S_2 le S_1} holds for all $b>0$. 

This completes task (i). Proposition~\ref{prop:F2} is now completely proved. 
\end{proof} 

\begin{proof}[Proof of Theorem~\ref{th:df bounds}] 
The attainment of suprema and infima in \eqref{eq:sups,infs} follows immediately from Corollary~\ref{cor:b}, which actually implies that the latter supremum and infimum in \eqref{eq:sups,infs} are attained at some functions in $\F_1\cup\F_{2,\infty}$. 

The equalities $\max_{r\in\oR}p_{1,b}(r)=\max_{r\ge r_b}p_{1,b}(r)=\max_{r> r_b}p_{1,b}(r)$ in \eqref{eq:S(b)=} follow from Proposition~\ref{prop:F1}. 

For all real $b$,  
\begin{equation*}
\begin{alignedat}{2}
	S(b)&=\sup_{f\in\F}\int_b^\infty f &&\quad \text{by \eqref{eq:sups,infs}} \\ 
	&=1-\inf_{f\in\F}\int_{-\infty}^b f &&\quad \text{because $\int_\R f=1$ } \\ 
	&=1-\inf_{f\in\F}\int_{-b}^\infty f &&\quad \text{by reflection $x\leftrightarrow-x$ } \\ 
	&=1-I(-b) &&\quad \text{by \eqref{eq:sups,infs} },  
\end{alignedat}
\end{equation*}
which proves \eqref{eq:I(b)=}. 

Consider now the case $b>0$ in \eqref{eq:S(b)=}. 
Then, in view of  
\begin{itemize}
	\item the definitions of $S_1(b)$ and $S_2(b)$ in \eqref{eq:S_2 le S_1}; 
	\item the definitions of $p_{2,b}(s)$, $\tf_{c,u,h,k}(x)$, $p_{1,b}(r)$, $f_{A,B,c,h}(x)$ 
respectively in \eqref{eq:p_{2,b}:=}, \eqref{eq:f_{c,u,h,k}}, \eqref{eq:p_{1,b}}, \eqref{eq:f_{A,B,c,h}};
\item ``parametrization'' Theorems~\ref{th:par2} and \ref{th:par1}, 
\end{itemize}
we have 
\begin{equation}\label{eq:S=S_1}
	S(b)=S_1(b)\vee S_2(b)=S_1(b)
\end{equation}
by \eqref{eq:S_2 le S_1}, and 
\begin{equation}\label{eq:I=I_2}
	I(b)=I_1(b)\wedge I_2(b)=I_2(b)
\end{equation}
by \eqref{eq:I_1 ge I_2}.   
Now \eqref{eq:S=S_1} and the definition of $S_1(b)$ in \eqref{eq:S_2 le S_1} imply the case $b>0$ of \eqref{eq:S(b)=}. 

If $b<0$, then, by \eqref{eq:I(b)=} and \eqref{eq:I=I_2},  
\begin{equation*}
	S(b)=1-I(-b)=1-I_2(-b), 
\end{equation*}
which proves the case $b<0$ of \eqref{eq:S(b)=}.

The case $b=0$ of \eqref{eq:S(b)=} can be dealt with similarly, and in a simpler way, than the case $b>0$ of \eqref{eq:S(b)=}. Alternatively, by \cite[Lemma 5.4]{lovasz-vempala}, $I(0)\ge1/e$ and hence $S(0)\le1-1/e$, which implies $S(0)=1-1/e$, in view of the p.d.f.\ $f\in\F_1\subset\F$ defined by \eqref{eq:f_-infty}. So, by \eqref{eq:I(b)=}, we get $S(0)=1-I(0)=1-1/e$, which proves the case $b=0$ of \eqref{eq:S(b)=} as well.

It remains to prove \eqref{eq:S(0),S(1)} and \eqref{eq:S(b) sim}, as the last two sentences in the statement of Theorem~\ref{th:df bounds} can be verified by simple direct calculations. 

The first equality in \eqref{eq:S(0),S(1)} follows immediately from \eqref{eq:S=S_1} and the consideration (in the proof of Proposition~\ref{prop:F1}) of Case~2 on page~\pageref{case2}. 
%
%Note that $p'_{1,1}(0\pm)=0$. So, by the case $b>0$ of \eqref{eq:S(b)=} and Proposition~\ref{prop:F1}, $S(1)=p_{1,1}(0)$. 
%
In view of \eqref{eq:f_{A,B,c,h}} and \eqref{r=0 or pm infty}, we also have the second equality in \eqref{eq:S(0),S(1)}, which completes the proof of \eqref{eq:S(0),S(1)}. 

Finally, to prove \eqref{eq:S(b) sim}, % \\exp-bound\\2.nb
note first that, by %\eqref{eq:p_{1,b}}, \eqref{eq:p_1=}, and \eqref{eq:F,g,G}, for real $b>\sqrt3$ and $r\ge r_b$ we have  
\eqref{eq:p_{1,b}=}, for real $b>\sqrt3$ and $r\ge r_b$ we have 
\begin{equation}\label{eq:p_1<}
	p_{1,b}(r)=\frac{e^{g_b(r)}-1}{e^r-1}<q_b(r):=\frac{e^{g_b(r)}}{e^r-1},
%	<\frac{e^{(s(r)-b\sqrt{u(r)})/(e^r-1)}}{e^r-1}
%	\underset{b,r\to\infty}\sim e^{-r+(s(r)-b\sqrt{u(r)})/(e^r-1)}=:q_b(r). 
\end{equation}
where $g_b(r)$ is as defined in \eqref{eq:F,g,G}. Next, for $r>0$ we have $2+e^r(r -2)+r>0$ and hence
\begin{equation*}
	q'_b(r)=\frac{e^{g_b(r)} r e^r}{2(e^r-1)^3\sqrt{u(r)}}\,
	(-b (2 + e^r (r-2) + r) - 2 \sqrt{u(r)}\,)<0
\end{equation*}
for real $b>\sqrt3$ and $r\ge r_b$, so that $q_b(r)$ decreases in $r\ge r_b$ if $b>\sqrt 3$. Recalling now \eqref{eq:S=S_1}, the definition of $S_1(b)$ in \eqref{eq:S_2 le S_1}, and Proposition~\ref{prop:F1}, 
%the fact that $p_{1,b}(r)=0$ on $[-\infty,r_b]$ for $b>1$, %and $r_b>0$ if $b>\sqrt3$, 
for $b>\sqrt3$ we have 
\begin{equation}\label{eq:S<q}
	S(b)=S_1(b)=\max_{r\ge r_b}p_{1,b}(r)\le\sup_{r\ge r_b}q_b(r)
	=q_b(r_b). 
\end{equation}
Next, $r-\B(r)\to1$ as $r\to\infty$ and hence $r_b-b=r_b-\B(r_b)\to1$ as $b\to\infty$. 
Since $\B(r_b)=b$, one has $g_b(r_b)=0$ and therefore 
\begin{equation}\label{eq:q sim}
	q_b(r_b)=\frac1{e^{r_b}-1}=\frac1{e^{b+1+o(1)}-1}\sim e^{-1-b} 
\end{equation}
%
%
%So, 
%\begin{equation}\label{eq:q sim}
%	q_b(r_b)=q_b(b+1+o(1))\sim e^{-1-b}
%\end{equation}
as $b\to\infty$. On the other hand, 
\begin{equation}\label{eq:S>}
	S(b)\ge p_{1,b}(\infty)=e^{-1-b}
\end{equation}
for $b>0$, by \eqref{eq:p_{1,b} special}. 
Now \eqref{eq:S(b) sim} follows from \eqref{eq:S<q}, \eqref{eq:q sim}, and \eqref{eq:S>}. 

This completes the proof of Theorem~\ref{th:df bounds}. 
\end{proof} 

%!!!!stopped here at 9;34 pm on 9/16/26

\begin{proof}[Proof of Theorem~\ref{th:upper bound}] % \\exp-bound\\3.nb \\checking\\Th 1.15 th(colon)upper bound.nb \\exp-bound\\positive r.nb \\exp-bound\\negative r.nb
By %\eqref{eq:S=S_1} and \eqref{eq:S_2 le S_1}, and because $p_{1,b}(r)=0$ if $b>1$ and $r\le r_b$ (whereas $r_b=-\infty$ for $b\in(0,1]$),  
\eqref{eq:S(b)=} and Proposition~\ref{prop:F1}, 
%\eqref{eq:F,g,G}, and because $p_{1,b}(r)=0$ for $r\le r_b$, 
for real $b>0$ we have 
\begin{equation*}
	S(b)=\max_{r\in\oR}p_{1,b}(r)=\max_{r\ge r_b}p_{1,b}(r). 
\end{equation*}
Hence, by \eqref{eq:p_1=} and \eqref{eq:F,g,G}, 
\begin{equation*}
C_{\sup}=\sup_{b\ge0}\max_{r\ge r_b}\RAT(b,r),	
\end{equation*}
where $r_0:=-\infty$ and 
\begin{equation}\label{eq:RAT}
	\RAT(b,r):=e^b\,\frac{e^{g_b(r)}-1}{e^r-1}
\end{equation}
for $r\ne0$, with 
\begin{equation}\label{eq:RAT(b,0)}
	\RAT(b,0):=e^b\Big(\frac12-\frac b{2\sqrt3}\Big),
\end{equation}
%$\RAT(b,0):=e^b(\frac12-\frac b{2\sqrt3})$, 
by continuity. 
This and other calculations in this proof were done in Mathematica notebook Proof\_Theorem2.7.nb.

Note that for $(b,r)=(\frac{34}{100},-\frac{56}{10})$ we have $r_b=-\infty$ and $\RAT(b,r)=0.681\ldots$, so that 
\begin{equation}\label{eq:>.68}
	C_{\sup}>0.68. 
\end{equation}
By \eqref{eq:S(b)=}, $S(0)=1-1/e<0.64$. 
%Also, as was noted at the end of Remark~\ref{rem:df bounds}, the functions $S$ is $1$-Lipshitz. So, for 
For 
\begin{equation*}
	b_1:=1/100,
\end{equation*}
by the downward monotonicity of $S$,  
\begin{equation*}
	\sup_{0\le b\le b_1}\max_{r\ge r_b}\RAT(b,r)
	=\sup_{0\le b\le b_1}e^b S(b)\le e^{1/100}S(0)<e^{1/100}0.64<0.65<0.68<C_{\sup}. 
\end{equation*}
Next, in view of \eqref{eq:S(b) sim}, for some real $b_2>b_1$ we have 
\begin{equation*}
	\sup_{b\ge b_2}\max_{r\ge r_b}\RAT(b,r)
	=\sup_{b\ge b_2}e^b S(b)<e^{-1}+0.01<0.68<C_{\sup}. 
\end{equation*} 
Further, uniformly in $b$ in the compact interval $[b_1,b_2]$ (or in any compact set) 
\begin{equation}\label{eq:RAT(b,infty)}
	\RAT(b,r)\underset{r\to\infty}\longrightarrow1/e<0.68<C_{\sup},
\end{equation}
%$\RAT(b,r)=0$ for $b>1$ and $\le r_b$, 
by \eqref{eq:>.68}, 
and 
\begin{equation}\label{eq:RAT(b,-infty)}
	\RAT(b,r)\underset{r\to-\infty}\longrightarrow e^b(1-e^{b-1})\le e/4<0.68<C_{\sup}. 
\end{equation}

It follows that
\begin{equation}\label{eq:C_*=}
	C_{\sup}=\RAT(b_\cri,r_\cri)
\end{equation}
for some critical point $(b_\cri,r_\cri)$ of the function $\RAT$ such that $b_\cri>0$ and $r_\cri>r_{b_\cri}$. For such a critical point, we have $r_\cri\ne0$, because, in view of \eqref{eq:RAT(b,0)}, for $b>0$ 
\begin{equation}\label{eq:RAT(b,0)<}
	\RAT(b,0)=e^b\Big(\frac12-\frac b{2\sqrt3}\Big)\le\frac{e^{\sqrt{3}-1}}{2 \sqrt{3}}<0.61<0.68<C_{\sup}. 
\end{equation}

We shall now show that such a critical point is unique and can be computed to any accuracy. Let 
\begin{equation}\label{eq:DrRAT}
	(D_r\RAT)(b,r):=(\dd_r\RAT)(b,r)\frac{(e^r-1)^2}{e^b}
	=e^r+e^{g(r)} (e^r (g'(r)-1)-g'(r)). 
\end{equation}
``Solving'' now the equation $(D_r\RAT)(b,r)=0$ for $e^{g(r)}$, we get 
\begin{equation}\label{eq:eg1}
	e^{g(r)}=\eg_1(b,r):=\frac{1}{1-(1-e^{-r}) g'(r)}.
\end{equation}

Next, letting 
\begin{equation*}
	R_1(r):=\frac{s(r)}{e^r-1}\quad\text{and}\quad 
	R_2(r):=\frac{\sqrt{u(r)}}{e^r-1}, 
\end{equation*}
we have 
\begin{equation}\label{eq:g,R1,R2}
	g(r)=g_b(r)=R_1(r)-b R_2(r)
\end{equation}
%$g(r)=R_1(r)-b R_2(r)$ 
and 
\begin{equation}\label{eq:DbRAT}
	(D_b\RAT)(b,r):=(\dd_b\RAT)(b,r)\frac{e^r-1}{e^b}
	=e^{g(r)}(1-R_2(r))-1. 
\end{equation}
``Solving'' now the equation $(D_b\RAT)(b,r)=0$ for $e^{g(r)}$, we get 
\begin{equation}\label{eq:eg2}
	e^{g(r)}=\eg_2(r):=\frac{1}{1-R_2(r)}.
\end{equation}

So, letting 
\begin{equation*}
	d(b,r):=\Big(\frac1{\eg_1(b,r)}-\frac1{\eg_2(r)}\Big)\,2(e^r-1)\sqrt{u(r)},
\end{equation*}
we have 
\begin{equation}\label{eq:d(b,r)}
	d(b,r)=
b r \tau(r)+2 u(r)-2\si(r) \sqrt{u(r)}, 
\end{equation}
where 
\begin{equation}\label{eq:tau(r)}
	\tau(r):=e^r (r-2)+r+2.
\end{equation}
Note that $r\tau(r)>0$ for real $r\ne0$. 

If $(b,r)=(b_\cri,r_\cri)$, then $d(b,r)=0$. Solving the latter equation for $b$, we get 
\begin{equation}\label{eq:tb}
	b=\tb(r):=2\frac{\si(r)\sqrt{u(r)}-u(r)}{r\tau(r)}>0 
\end{equation}
(for nonzero real $r$, assumed by default, according to our convention). 

Let now 
\begin{equation}\label{eq:Q(r)}
	Q(r):=(D_b\RAT)(\tb(r),r)=\exp(g_{\tb(r)}(r))(1-R_2(r))-1
\end{equation}
and 
\begin{equation}\label{eq:DQ}
	(DQ)(r):=Q'(r)\,2(e^r-1)^3 r^2 \tau(r)^2 \sqrt{u(r)} \exp(-g_{\tb(r)}(r))
	=c_0(r)(e^r-1)-2c_1(r)\sqrt{u(r)},  
\end{equation}
where $c_0(r)$ and $c_1(r)$ are certain E-polynomials, both of E-degree $6$. 

We find that 
\begin{equation}\label{eq:c_0}
	c_0(r)(e^r-1)\le0\iff(r=0\text{ or }r\le\tr_2),
\end{equation}
for a certain real number 
\begin{equation*}
	\tr_2:=-5.42\ldots<0,
\end{equation*}
and 
\begin{equation}\label{eq:c_1}
\text{$c_1(r)>0$ (for $r\ne0$)}. 	
\end{equation}
Repeated-hammering proofs of \eqref{eq:c_0} and \eqref{eq:c_1} will be given in Appendices~\ref{app:c0} and \ref{app:c1}, respectively. 

So, if $r\le\tr_2$, then $c_1(r) > 0$ and $c_0(r)(e^r-1)\le0$, so that $(DQ)(r) < 0$. 
% and hence $Q'(r) > 0$. 

And 
if $r > \tr_2$ and $r \ne 0$, then $c_1(r) > 0$ and $c_0(r)(e^r-1) > 0$, 
so that 
\begin{equation*}
	(DQ)(r) > 0 \iff r V(r)>0, 
\end{equation*}
%$(DQ)(r) > 0$ %and hence $Q'(r) < 0$) 
%iff $r V(r)>0$, 
where 
\begin{equation}\label{eq:V:=}
	V(r):=\frac{c_0(r)^2(e^r-1)^2-(2c_1(r)\sqrt{u(r)})^2}{e^r r^3 \tau(r)^2}, 
\end{equation}
which is an E-polynomial of E-degree $9$.
We find that 
\begin{equation}\label{eq:V}
	V(r)>0 \iff (r<\tr_3 \text{ or }r>0),  
\end{equation}
for a certain real number 
\begin{equation*}
	\tr_3=-3.04\ldots\in(\tr_2,0). 
\end{equation*}
A repeated-hammering proof of \eqref{eq:V} will be given in Appendix~\ref{app:V}.

Note also that $Q'(r)$ equals $r (DQ)(r)$ in sign.

Therefore and because $Q(0\pm)=0$ and $Q(-\infty+)=2/e-1<0$ (and because the functions $Q$ and $V$ are real analytic on $\R\setminus\{0\}$), we see that $Q$ is (strictly) increasing on $(-\infty,\tr_3]$, decreasing on $]\tr_3,0)$, and increasing on $(0,\infty)$. So, $Q$ has a unique nonzero root $\tr_1\in(-\infty,\tr_3)$, and 
\begin{equation}\label{eq:tr1}
	\tr_1=-5.6367982\ldots. 
\end{equation}
So, recalling \eqref{eq:tb} and \eqref{eq:C_*=}, and using interval calculations implemented via the Mathematica command Interval, we get 
\begin{equation}\label{eq:tbtr1}
	\tb(\tr_1)=0.336266195\ldots
\end{equation}
and 
\begin{equation}%\label{eq:C_*=0.68}
	C_{\sup}=\RAT(\tb(\tr_1),\tr_1)=0.681100549\ldots. 
\end{equation} 
\end{proof}

\section{Proof of Proposition~\ref{prop:p'(r)}}\label{proof:p'(r)}
%\begin{proof}[Proof of Proposition~\ref{prop:p'(r)}]
Take any real $r$. If $r<r_b$, then the right and left derivatives of $p_1$ at $r$ are $0$ and hence \eqref{eq:p'(r)} obviously holds. If $r=r_b$, then the left derivative of $p_1$ at $r$ is again $0$ and hence \eqref{eq:p'(r)} holds for the left derivative. 
If $r>r_b$, then $p_1$ is just differentiable at $r$. 

So, in the rest of this proof, we may and will assume that $r\ge r_b$ (or, equivalently, $b\le\B(r)$) and consider only the right derivative of $p_1$ at such $r$. 

By \eqref{eq:p_{1,b}} and \eqref{eq:f_{A,B,c,h}}, and because $\A(r)>0$ for all $r$,
\begin{equation*}
	p'_1(r)=\frac{\c'(r)}{\c(r)}\,p_1(r)
	+\c(r)e^{\h(r)\B(r)}\B'(r)
	+\c(r)\h'(r)\int_b^{\B(r)} x e^{\h(r)x}\,dx; 
\end{equation*}
one may note here that 
%here the functions $\c',\B',\h'$ are extended to $\R$ by continuity, so 
$\c'(0)=0$, $\B'(0)=\frac1{2\sqrt3}$, $\h'(0)=-\frac1{2\sqrt3}$. 

By Lemma~\ref{lem:B}, $\B'>0$. Also, 
\begin{equation*}
	0\le\c(r)\int_b^{\B(r)} x e^{\h(r)x}\,dx
	\le \B(r)\c(r)\int_b^{\B(r)} e^{\h(r)x}\,dx=\B(r) p_1(r). 
\end{equation*}

So, 
\begin{equation}\label{eq:p'_1>}
	p'_1(r)>-\frac{|\c'(r)|}{\c(r)}\,p_1(r)
	-|\h'(r)|\B(r) p_1(r). 
\end{equation}

Next, 
\begin{equation}\label{eq:crat-def}
	\c\rat(r):=\frac{\c'(r)}{\c(r)}\Big/\frac{e^r r}{2(e^r-1)^2u(r)}
\end{equation}
is an E-polynomial of E-degree $2$
and 
\begin{equation}\label{eq:tcrat}
	\wt{\c\rat}(r):=\Big(\frac{\c'(r)}{\c(r)}-\frac8{100}\Big)50(e^r-1)^2 u(r)
\end{equation}
is an E-polynomial of E-degree $4$. It will be proved in Appendix~\ref{app:crat} 
%\\checking\\Prop 2.1 prop(colon)p'(r).nb , \\checking\\crat-pos.nb
that 
\begin{equation}\label{eq:crat}
\text{$c\rat(r)>0>\wt{c\rat}(r)$ for $r>0$. }	
\end{equation}
Also, the function $\c$ is even (say, by the symmetry $A\leftrightarrow B$) and hence the function $\frac{\c'}\c$ is odd. 
Therefore and by continuity, 
\begin{equation*}
	\Big|\frac{\c'}{\c}\Big|\le\frac8{100}. 
\end{equation*}

Also, the function $\h$ is odd (say again, by the symmetry $A\leftrightarrow B$) and hence $\h'$ is even. Also, for real $r$, 
\begin{equation*}
	\h'(r)=-\frac{e^r\,r\,\tau(r)}{2(e^r-1)^2\sqrt{u(r)}}\le0,  
\end{equation*}
with $\tau(r)$ as defined in \eqref{eq:tau(r)}, so that 
$r\tau(r)>0$ for real $r\ne0$. 
%
%$\h'\le0$ 
Next,
\begin{equation}\label{eq:Dh'B}
	\wt{\h'\B}(r):=\Big(-\h'(r)\B(r)-\frac35\Big)10(e^r-1)^2 u(r)<0 
	\text{ for $r>0$},  
\end{equation}
%is an E-polynomial of E-degree $4$, which is $<0$ for $r>0$, 
as will be shown in Appendix~\ref{app:h'B}. % \\checking\\h'Bpos.nb 
Also, in view of $\B'>0$, the positive function $\B$ is increasing.
It follows that 
\begin{equation*}
	|\h'|\B=-\h'\B\le\frac35. 
\end{equation*}
Recalling now \eqref{eq:p'_1>} and Theorem~\ref{th:upper bound}, we get 
\begin{equation*}%\label{eq:p'_1>}
	p'_1(r)>-\Big(\frac8{100}+\frac35\Big)p_1(r)
	>-\Big(\frac8{100}+\frac35\Big)\frac{69}{100}\,e^{-b}
	>-\frac12\,e^{-b}.  
\end{equation*}
This completes the proof of Proposition~\ref{prop:p'(r)}. 
\qed
%\end{proof}

%\cite[Lemma~5.5]{lovasz-vempala}: $f(0)\ge1/8$. 

\newpage 

\section{Discussion}\label{discussion} 

\subsection{Consequences and related bounds}\label{related}
Theorems~\ref{th:df bounds} and \ref{th:upper bound} may be considered the main results of this paper. However, the general results given by  Theorems~\ref{th:2-piece}, \ref{rem:H}, \ref{th:4 zeros}, \ref{th:par1}, and  \ref{th:par2} may also be of interest. 

Theorem~\ref{th:upper bound} admits the following immediate corollary. 

\begin{corollary}\label{cor:two-side}
For any r.v. $X$ with distribution in $\L$ and any real $t$, 
\begin{equation}\label{eq:two-side}
	\P(|X|\ge t)\le 2C_{\sup}e^{-t}. 
\end{equation}
\end{corollary}
The constant factor $2C_{\sup}=1.362\dots$ in \eqref{eq:two-side} is almost twice as good as the constant factor $e=2.718\dots$ in the bound $e^{1-t}=e\,e^{-t}$ in (used in this paper) inequality \eqref{eq:exp-decay} by Lovasz and Vempala. However, the bound in \eqref{eq:exp-decay} was obtained without assuming the zero-mean condition $\E X=0$; as for the condition $\E X^2\le1$ for \eqref{eq:exp-decay}, it is in this setting equivalent to the condition $\E X^2=1$, in view of rescaling. 

%\medskip
%\hrule
%\medskip

There have been other results concerning exponential decrease of the tails of logconcave distributions. 
In particular, Borell's Theorem~8.1 \cite{borell83} and Pr\'ekopa's Theorem~6 \cite{prekopa73}
%(see e.g.\ \cite{wellner-logconcav} 
imply that for any $\mu\in\L$ there is some real $\ep>0$ such that $\int_\R e^{\ep|x|}\mu(dx)<\infty$. 
Lemma~10 of Schoenberg \cite{schoenberg51} implies that for any $f\in\F$ there are some real $\ep>0$ and $a$ such that $f(x)\le e^{a-\ep x}$ for all real $x$.  
Essentially, these facts, with an unspecified $\ep>0$, follow from the observation that, if $\psi\colon\R\to\R$ is a convex function such that $\psi(x)\to\infty$ as $x\to\infty$, then for some $\ep\in(0,\infty]$ we have $\psi(x)/x\to\ep$ as $x\to\infty$. 
A multidimensional generalization of the mentioned result in \cite{schoenberg51} was given by Cule and Samworth \cite[Lemma~1]{cule-samworth}. 

Devroye's Inequality~1 \cite{devroye} states that for all log-concave p.d.f.'s $f$ on $[0,\infty)$ with a mode at $0$ and $f(0)=1$  and for all real $x\ge0$ 
%
%the maximum of $f(x)$ over all log-concave p.d.f.'s $f$ on $[0,\infty)$ with mode at $0$ and $f(0)=1$ is 
\begin{equation}\label{eq:devr}
	%m(x):=
	f(x)\le \begin{cases}
	1 &\text{ if }0\le x\le1, \\ 
	z_x &\text{ if }x>1,
	\end{cases}
%	1(0\le x<1)+1(x\ge1)z_x, 
\end{equation}
where, for $x>1$, $z_x$ is the unique solution $z\in(0,1)$ of the equation 
\begin{equation}\label{eq:z_x}
	xz e^{-xz}=x e^{-x}. 
\end{equation}
The upper bound on $f(x)$ in \eqref{eq:devr} is exact for each real $x\ge0$: for $x\in[0,1]$ the bound is attained by the uniform density over $[0,1]$ and for $x>1$ the bound is attained by
the p.d.f.\ $f_x$ given by the formula $f_x(y):=z_x^{y/x}\,1(0\le y\le x)$. 
Note also that $z_x\sim e^{-x}$ as $x\to\infty$, but $z_x>e^{-x}$ for $x>1$. %(???) 

Actually, the condition that
$f$ be supported on $[0,\infty)$ is not needed for \eqref{eq:devr}.
Also, the first case of the inequality in \eqref{eq:devr} is obvious, given the conditions that $f$ has a mode at $0$ and $f(0)=1$. 

As for the second case of the inequality in \eqref{eq:devr}, there the condition that $0$ is a mode of $f$ is not needed. 
%
%Actually, the conditions that $0$ a mode of $f$ and that $f$ be supported on $[0,\infty)$ are not needed for \eqref{eq:devr}. 
Also, the condition $f(0)=1$ can be generalized so as to allow $f(0)$ to be any positive number. 

Indeed, for any log-concave p.d.f.\ $f$ and any real $x_0$ and $x_1$ such that $x_0<x_1$, with $x_t$ as defined in \eqref{eq:x_t}, 
%\begin{equation}\label{eq:x_t}
%	x_t:=(1-t)x_0+tx_1
%\end{equation}
%for real $t$, 
we have 
\begin{equation}\label{eq:=L}
	\int_{x_0}^{x_1}f=(x_1-x_0)\int_0^1 f(x_t)\,dt
	\ge(x_1-x_0)\int_0^1 f(x_0)^{1-t}f(x_1)^t\,dt
	=(x_1-x_0)L(f(x_0),f(x_1)),
\end{equation}
the logarithmic mean of $f(x_0)$ and $f(x_1)$. Recall that 
\begin{equation*}
	L(\al,\be)=\int_0^1 \al^{1-t}\be^t\,dt
	=\left\{	
\begin{alignedat}{2}
	& \frac{\al-\be}{\ln\al-\ln\be} &&\text{ if }\al>0,\be>0, \al\ne\be, \\ 
	& 0 &&\text{ if }\al\ge0,\be\ge0, \al\be=0, \\ 
	& \al &&\text{ if }\al=\be>0,  
\end{alignedat}
	\right.
\end{equation*}
so that $L(\al,\be)$ is continuous in $(\al,\be)\in[0,\infty)^2$, increasing in $\al\ge0$ for each $\be>0$, and increasing in $\be\ge0$ for each $\al>0$. 

If $\ga:=f(0)>0$ and $x>0$, then \eqref{eq:=L} with $x_0=0$ and $x_1=x$, together with the inequality $1\ge\int_{x_0}^{x_1}f$, yields 
\begin{equation}\label{eq:1/x ge}
	\frac1x\ge L(\ga,f(x))
\end{equation}
Consider now the equation 
\begin{equation}\label{eq:z_{ga,x}}
	\frac1x=L(\ga,z)
\end{equation}
for $z$. Since $L(\ga,z)$ continuously increases from $0$ to $\infty$ as $z$ increases from $0$ to $\infty$, equation \eqref{eq:z_{ga,x}} has a unique root $z=z_{\ga,x}\in(0,\infty)$. Comparing \eqref{eq:1/x ge} and \eqref{eq:z_{ga,x}}, and recalling that $L(\cdot,\cdot)$ is continuously increasing in each of its two arguments, we get the inequality 
\begin{equation}\label{eq:f le z}
	f(x)\le z_{\ga,x}
\end{equation}
for all real $x>0$. 

By default in this section, let us use $z$ for $z_{\ga,x}$, for brevity.  
%If $z=0$, then \eqref{eq:f le z} provides the prefect upper bound, $0$, on $f(x)$. So, wlog $z
Recall that $\ga>0$ and $z>0$. If $z=\ga$, then $L(\ga,z)=\ga$ and hence \eqref{eq:z_{ga,x}} implies $x=1/\ga$. Until further notice in this section, let us exclude this case, thus assuming that $x\in(0,\infty)\setminus\{1/\ga\}$.  

Then equation \eqref{eq:z_{ga,x}} can be rewritten as 
\begin{equation}\label{eq:z}
	\frac1x=\frac{\ga-z}{\ln\ga-\ln z},
\end{equation}
which can be further rewritten as $x(\ga-z)=\ln\ga-\ln z$, and then as 
$e^{\ga x}e^{-xz}=\ga/z$, $ze^{-xz}=\ga e^{-\ga x}$, and finally as 
\begin{equation}\label{eq:eq z}
	-xze^{-xz}=-\ga x e^{-\ga x}. 
\end{equation}
We see that indeed equation \eqref{eq:z_x} is the special case of equation \eqref{eq:eq z}, corresponding to $\ga=1=f(0)$. 

Next, $\{-\ga x e^{-\ga x}\colon x>0\}=[-1/e,0)$. Recall that the case $x=1/\ga$ has been currently excluded, which excludes the possibility $z=\ga$, so that we have so far had $-xz\ne-\ga x$. It follows that 
\begin{equation}\label{eq:z=}
\begin{aligned}
	z=z_{\ga,x}&=
	\begin{cases}
	-\frac1x\,W_{-1}(-\ga x e^{-\ga x}) &\text{ if }0<x<1/\ga, \\ 
	-\frac1x\,W(-\ga x e^{-\ga x}) &\text{ if }x>1/\ga, \\ 
	\ga &\text{ if }x=1/\ga, 
	\end{cases} \\ 
	&=\begin{cases}
	-\frac1x\,W_{-1}(-\ga x e^{-\ga x}) &\text{ if }0<x\le1/\ga, \\ 
	-\frac1x\,W(-\ga x e^{-\ga x}) &\text{ if }x\ge1/\ga,  
	\end{cases}  	
\end{aligned}	
\end{equation}
where 
\begin{itemize}
	\item $W=W_0$ is the $0$ branch of the Lambert function \cite{knuth96}, whose restriction to the interval $[-1/e,0)$ is the function inverse to the bijective function \\ $[-1,0)\ni w\mapsto w e^w\in[-1/e,0)$; 
	\item $W_{-1}$ is the $-1$ branch of the Lambert function, whose restriction to the interval $[-1/e,0)$ is the function inverse to the bijective function \\  $(-\infty,-1]\ni w\mapsto w e^w\in[-1/e,0)$. 
\end{itemize}

We have thus proved the following proposition, using a very minimal amount of information on an arbitrary log-concave p.d.f.: its value at $0$. 
%Actually, of course, by re-scaling, wlog we may assume $\ga=1$ -- which we will not do, though, for the sake of convenience. Moreover, we can use horizontal shifting so as, in Proposition~\ref{prop:devr} below, to replace the condition $\ga:=f(0)>0$ by $\ga:=f(x_0)>0$, where $x_0$ is any real number, and then replace there $x$ by $x-x_0$. 
 
\begin{proposition}\label{prop:devr}
Let $f$ be any log-concave p.d.f., and let $\ga:=f(0)>0$. Then for all real $x>0$ 
\begin{equation*}
	f(x)\le%1\wedge 
	z_{\ga,x} 
\end{equation*}
and hence, applying the same result to the p.d.f. $x\mapsto f(-x)$, we get  
\begin{equation*}
	f(x)\le%1\wedge 
	z_{\ga,|x|} 
\end{equation*}
for all real $x<0$. 
\end{proposition} 

Note that 
\begin{equation*}
	z_{\ga,x}\sim\ga e^{-\ga x}\text{ as }x\to\infty. 
\end{equation*}
This follows because, clearly, $W(z)\sim z$ as $z\to0$. 
On the other hand, 
\begin{equation*}
	z_{\ga,x}>\ga e^{-\ga x} 
\end{equation*}
for all real $x>0$: 
\begin{enumerate}[(i)]
	\item if $0<x\le1/\ga$ then $z_{\ga,x}\ge1/x>1/(e x)\ge\ga e^{-\ga x}$ because $W_{-1}(z)\le-1$ for $z\in[-1/e,0)$;
	\item if $x\ge1/\ga$ then $z_{\ga,x}>\ga e^{-\ga x}$ because $we^w>w$ for all real $w\ne0$ and hence $W(z)<z$ for all real $z\ne0$. 
\end{enumerate}

Of course, by re-scaling, in Proposition~\ref{prop:devr} we may wlog assume $\ga=1$ -- which we did not do, though, for the sake of convenience. 

Moreover, we can use horizontal shifting so as, also in Proposition~\ref{prop:devr}, to replace the condition $\ga:=f(0)>0$ by $\ga:=f(x_0)>0$, where $x_0$ is any real number, and then accordingly replace there $x$ by $x-x_0$. We did not do that either, for simplicity of writing. 

We can say more if $f\in\F$, that is, if $f$ is the l.s.c.\ p.d.f.\ of an isotropic log-concave distribution. Indeed, let $m$ denote a mode of $f$ (defined, say, by the condition $\limsup_{x\to m}f(x)\ge f(y)$ for all real $y$), so that $f$ is nonincreasing on $(m,\infty)$. Then, by Theorem~\ref{th:upper bound} and \eqref{eq:=L}, for all real $x_0>m_+$ and all real $x>x_0$,  
\begin{equation*}
	C_{\sup}\, e^{-x_0}\ge\int_{x_0}^\infty f \ge\int_{x_0}^x f \ge(x-x_0)L(f(x_0),f(x))
	\ge (x-x_0)f(x).
\end{equation*}
Now take any real $\thh>0$. Choosing $x_0=x-\thh$ for any $x>\thh+m_+$, we now have 
\begin{equation*}
	f(x)\le C_{\sup}\,\frac1\thh\,e^\thh e^{-x}; 
\end{equation*}
this holds for $x=\thh+m_+$ as well, since $f$ is l.s.c. 
We may want to choose $\thh=1$, so as to minimize $\frac1\thh\,e^\thh$. 

Note also that, according to \cite[Theorem~2]{johnson-rogers} (see also  \cite[Theorem~2.2]{sharma-etal} for a more elementary proof), we have  
\begin{equation}\label{eq:m le}
	m\le\sqrt3. 
\end{equation}
The upper bound $\sqrt3$ on $m$ in \eqref{eq:m le} is the best possible one, as it is attained in the limit when $f$ is the standardized version of the p.d.f.\ $\R\ni x\mapsto\frac{1+tx}{2\sqrt3}\,1(|x|\le\sqrt3)$ with $t\downarrow0$. 

We have thus obtained the following proposition. 

\begin{proposition}\label{prop:iso}
Let $f\in\F$ and let $m$ be a mode of $f$. Then, for 
any $x\ge1+m_+$, 
\begin{equation}\label{eq:f le}
		f(x)\le eC_{\sup}\, e^{-x}. 
\end{equation}
In view of \eqref{eq:m le}, inequality \eqref{eq:f le} holds for all $f\in\F$ and  all real $x\ge1+\sqrt3.$ 
\end{proposition} 

Proposition~\ref{prop:iso} may be compared with (used in this paper) inequality  \eqref{eq:f le1} by Lovasz and Vempala, and also with the following result by 
Karlin, Proschan, and Barlow \cite[Theorem~5]{karlin_etal61}: 
% Barlow, Karlin, and Proschan \cite[Theorem~5]{BKP}: 
if $X$ is a \emph{nonnegative} r.v.\ with a log-concave p.d.f.\ $f_X$ and $\E X=1$, then 
$f_X(x)\le e^{-x}$ for some (unspecified) real $x_0$ and all real $x>x_0$. 

In conclusion of this subsection, let us provide exact upper and lower bounds on the exponential moments of measures $\mu\in\L$. For any real $\ka>0$, any real $y$, and all real $x$, let 
\begin{equation}
	E_{\ka,y}(x):=e^{\ka x}\,1(x\le y)+e^{\ka y}\sum_{j=0}^3\frac{\ka^j(x-y)^j}{j!}
	%(1+\ka(x-y)+\ka^2(x-y)^2/2+\ka^3(x-y)^3/6)
	\,1(x>y). 
\end{equation}
Then condition \eqref{eq:H le} and hence the conclusion of Theorem~\ref{th:2-piece} hold for $H=E_{\ka,y}$. So, in view of formula \eqref{eq:f_{A,B,u,c,h,k}} for $f\in\F_2$, there is a minimizer $f_{A,B,u,c,h,k}$ of $\int_\R E_{\ka,y}f$ over $f\in\F$. 

%\begin{equation}\label{eq:f_{A,B,u,c,h,k}}
%	f(x)=f_{A,B,u,c,h,k}(x):=c e^{h(x-u)}1(-A<x\le u)+c e^{-k(x-u)}1(u<x<B)
%\end{equation}

Also, the function $E_{\ka,y}$ is $3$-good in the sense that, for any real  $\la_1,\la_2,\la_3$, the function 
\begin{equation*}
\R\ni x\mapsto R(x):=R_{\la_1,\la_2,\la_3}(x):=E_{\ka,y}(x)-\la_1-\la_2 x-\la_3 x^2
\end{equation*}
has at most three zeros (cf.\ the definitions of $4$-good functions in the statement of Theorem~\ref{th:4 zeros}); this follows because the function $E'_{\ka,y}$ is strictly convex. 

Reasoning now as in the proof of Theorem~\ref{th:4 zeros}, we see the following. 
If $f_{A,B,u,c,h,k}$ is not a one-piece function, then $-A<u<B$ and $h\ne-k$, so that, by \eqref{eq:ints}, $R$ has at least $4$ zeros, a contradiction. 

So, $f_{A,B,u,c,h,k}$ is a one-piece function, say with $h=-k$. However, for $f=f_{A,B,u,c,h,-h}$, the differentiation in $c$ and the differentiation in $h$ still yield $\int_{-A}^B Rf=0=\int_{-A}^B xR(x)f(x)$, so that, by Lemma~\ref{lem:2 zeros}, $R$ has at least two zeros in the interval $(-A,B)$. If $A$ and $B$ are both finite, then, in view of \eqref{eq:R(-A)} and \eqref{eq:R(B)}, $R(-A)=0=R(B)=0$, which yields at least four real zeros for $R$, again a contradiction. 

So, $f_{A,B,u,c,h,k}$ is a one-piece function and at least one of the endpoints, $-A$ or $B$, is infinite. Therefore, in view of Theorem~\ref{th:par1} and (say) \eqref{r=0 or pm infty}, the minimizer $f_{A,B,u,c,h,k}$ of $\int_\R E_{\ka,y}f$ over $f\in\F_2$ is one of the two functions, $f^+$ or $f^-$, where 
\begin{equation}
	f^+(x):=e^{-1-x}\,1(x>-1)\quad\text{and}\quad 
	f^-(x):=f^+(-x)=e^{-1+x}\,1(x<1)
\end{equation}
for real $x$. 
Note also that, for each real $x$, $E_{\ka,y}(x)$ is nondecreasing in $y$ and converges to $E_\ka(x):=e^{\ka x}$ as $y\to\infty$.  
It follows that %, for any $f\in\F_2$, 
\begin{equation}
	\inf_{f\in\F_2}\int_\R E_\ka f \ge \inf_{f\in\F_2}\int_\R E_{\ka,y}f
	=\min_\pm \int_\R E_{\ka,y}f^\pm
	\underset{y\to\infty}\longrightarrow\min_\pm \int_\R E_\ka f^\pm 
	\ge\inf_{f\in\F_2}\int_\R E_\ka f,
\end{equation}
by the monotone convergence theorem and because $f^\pm\in\F_2$. 
We conclude that 
\begin{equation}
	\inf_{f\in\F_2}\int_\R E_\ka f=\min_{f\in\F_2}\int_\R E_\ka f
	=\min_\pm \int_\R E_\ka f^\pm, 
\end{equation}
for any real $\ka>0$. The case $\ka=0$ is trivial. 

The case $\ka<0$ is dealt with similarly to the case $\ka>0$ or by using  ``horizontal'' reflection ($E_{\ka,y}(x)\leftrightarrow E_{\ka,y}(-x)$ and $E_\ka(x)\leftrightarrow E_\ka(-x)$). Similarly, or by ``vertical'' reflection   ($E_{\ka,y}\leftrightarrow-E_{\ka,y}$ and $E_\ka\leftrightarrow-E_\ka$), 
%(followed by passage to the limit), 
dealt with is $\sup_{f\in\F_2}\int_\R E_\ka f$. 

Thus, we obtain the following. 

\begin{proposition}\label{prop:exp}
For any real $\ka$, 
\begin{equation}
	\min_{f\in\F}\int_\R e^{\ka x}f(x)\,dx=
	\min_{f\in\F_2}\int_\R e^{\ka x}f(x)\,dx=\frac{e^{|\ka|}}{1+|\ka|}
\end{equation}
and 
\begin{equation}
	\max_{f\in\F}\int_\R e^{\ka x}f(x)\,dx=
	\max_{f\in\F_2}\int_\R e^{\ka x}f(x)\,dx=
	\left\{
	\begin{alignedat}{2}
	&\frac{e^{-|\ka|}}{1-|\ka|} &&\text{ if }|\ka|<1, \\
	&\infty &&\text{ if }|\ka|\ge1. 
	\end{alignedat}
	\right.
\end{equation}
\end{proposition}

%By monotone convergence, $\int_\R E_{\ka,y}f^\pm$ is nondecreasing to $\int_\R E_\ka f^\pm$ as $y$ increases to $\infty$. 

\bigskip

A broad survey on log-concavity was given by Saumard and Wellner \cite{wellner-logconcav}. 

%\subsection{On this paper}\label{this paper}

\subsection{The %localization 
finite-dimensional reduction method of Fradelizi and Gu\'edon}

Condition \eqref{eq:log-concave} that a function $f$ be log concave is a system of uncountably many inequalities, indexed by real $x_0,x_1$ and $t\in(0,1)$, and these inequalities are 
nonlinear in $f$. So, finding the optimal bound on a functional of log-concave distributions is a problem of nonlinear infinite-dimensional optimization, which may be very nontrivial -- cf.\ e.g.\ \cite{Fattorini_1999}.  

The crucial tools for dealing with such problems were provided by Theorem~1 of Fradelizi and Gu\'edon \cite{frad-guel} and by the extreme-point argument in the proof of their Theorem~2. These tools, used in the proof of Theorem~\ref{rem:H}, reduce the problem to a finite-dimensional one.

%The crucial tools in dealing with such problems were provided by Theorems~1 and 2  by Fradelizi and Gu\'edon \cite{frad-guel}, which reduce these problems to finite-dimensional ones. These tools were used in the present paper 
%%in the %Subsection~\ref{compact} 
%in the proof of Theorem~\ref{rem:H}. 

I learned about the paper \cite{frad-guel} from an
\href{https://mathoverflow.net/questions/512106/maximize-expectation-of-a-function-among-log-concave-random-variables/513788\#513788}
{answer by Bogdan Grechuk on MathOverflow}; see Supplement~A, described on page~\pageref{suppl}. 
%
%I learned about the paper \cite{frad-guel} from an answer by Bogdan Grechuk on \break 
%MathOverflow; see Supplement~A. %~\ref{suppl}. 
In that answer, the exact upper bound on %\break 
$\E(1+|X|)e^{-|X|}$ was obtained for all r.v.'s $X$ with isotropic log-concave distributions; the reduction there to the two-piece p.d.f.'s was based on \cite{frad-guel}, and the resulting finite-dimensional optimization was verified in Lean \url{https://github.com/Mirkes/OneShotSeparabilityTheorem/blob/main/LogConcaveFamily6Theorem2_WithExistence_Live.lean}. Based on that verification, ChatGPT provided an exposition of the corresponding proof.
%
%The answer relied on a Lean verification \url{https://github.com/Mirkes/OneShotSeparabilityTheorem/blob/main/LogConcaveFamily6Theorem2_WithExistence_Live.lean}, based on which ChatGPT provided an exposition of the corresponding proof. 
The part of the proof of Theorem~\ref{rem:H} presented in Subsection~\ref{compact} and the proof of Theorem~\ref{th:4 zeros} were adapted from the ChatGPT exposition.  

%Ideas from the proofs of Proposition~2.1 and Lemmas~5.1 and 5.3 in the ChatGPT exposition were used in the proofs of Theorems~\ref{rem:H} and \ref{th:4 zeros} 
%%Subsection~\ref{compact} and in the proofs of Lemmas \ref{lem:k>1}
%in the present paper. 

% at  \url{https://mathoverflow.net/questions/512106/maximize-expectation-of-a-function-among-log-concave-random-variables/513788#513788}. The answer relied on a Lean verification, based on which ChatGPT provided an exposition of the corresponding proof. 

%\section{Proof of Theorem~\ref{rem:H}}\label{2-piece proof}
%
%\subsection{``Compact'' version of Theorem~\ref{rem:H}}\label{compact} 

%
%
%!!!! $z_{\ga,x}\sim\ga e^{-\ga x}$ as $x\to\infty$ but (???) $z_{\ga,x}>\ga e^{-\ga x}$ for $x>0$ : if $0<x<1/\ga$ then $z_{\ga,x}>1/x>1/(e x)\ge\ga e^{-\ga x}$

% \\EJP\\paper\\devroye\\Untitled-1.nb
%\begin{figure}[htbp]
%	\centering
%		\includegraphics[width=.9\textwidth]{pic-devr.pdf}
%	\caption{1/2, 1, 2}
%	\label{fig:pic}
%\end{figure} 

%!!!! do $f(x)\le C_{\sup}\frac{e^\thh}\thh\,e^{-x}$ for $x\ge\tilde m_+ +\thh$ (p. 6 of the notes)

\vspace{1cm}

\appendix

\section{General formal description of repeated hammering} \label{app:hammer}

%The definitions of the polynomials $p_j(r)$ and other definitions given in any one of the subsequent appendices may conflict with definitions in other appendices or elsewhere in the paper. In particular, the definitions of the polynomials $p_1(r)$ in the appendices conflict with the definition of $p_1(r)$ previously given in \eqref{eq:p_{1,b}}. This is done intentionally, to keep simpler symbols. However, the definitions given in a particular appendix should be considered local, in the sense that, within any particular appendix, the definitions given there take precedence over the definitions of the same symbols given elsewhere in this paper.  

Let $\FF\colon I\to\R$ be an E-polynomial of E-degree $k$ as in Remark~\ref{rem:e-poly}, so that 
\begin{equation}\label{eq:E-poly}
	\FF=\sum_{j=0}^k \p_j \EE^j 
\end{equation}
and the polynomial $p_k$ is nonzero. 
The E-degree $k$ of $\FF$, the polynomials $\p_0,\dots,\p_k$, and their respective degrees, say $d_0,\dots,d_k$, are of course attributes/functions of $\FF$; as usual, the degree of the zero polynomial is assumed to be $-1$. In our considerations, the interval $I$ will usually be %$\R$ or 
$(0,\infty)$ or $(-\infty,0)$. %??? 

The $k$ rounds of repeated hammering of the E-polynomial $\FF$ starting with 
\begin{equation}\label{eq:FF_0}
	\wt\FF_0:=\FF 
\end{equation}
successively produce 
\begin{equation}\label{eq:wtFFqc}
	\wt\FF_{\qc(j)}:=
\left\{	
\begin{aligned}
	\frac{\wt\FF_{\qc(j-1)}^{(\q(j))}}\EE &\text{ if }p_{j-1}\ne0, \\ 
	\frac{\wt\FF_{\qc(j-1)}}\EE &\text{ if }p_{j-1}=0  
\end{aligned}
\right.	 
\end{equation}
for $j=1,\dots,k$, where 
\begin{equation}\label{eq:q,qc}
	\q(j):=(d_{j-1})_+ +1,\quad \qc(j):=\sum_{i=1}^j\q(i)\ \text{(with $\qc(0):=0$)},
\end{equation}
so that $\qc$ means ``$\q$ cumulative''. 

As should be clear from Remark~\ref{rem:e-poly}, the result of these $k$ rounds of hammering is the usual polynomial $\wt\FF_{\qc(k)}$. 

Let us provide more detail on the repeated hammering procedure. For all $\be=1,\dots,\qc(k)$ we have 
\begin{equation}\label{eq:wtFF_be}
	\wt\FF_\be:=
	\left\{
	\begin{aligned}
	\frac{\wt\FF'_{\be-1}}{\EE} &\text{ if $\be=\qc(j)$ and $p_{j-1}\ne0$ for some $j=1,\dots,k$}, \\ 
	\frac{\wt\FF_{\be-1}}{\EE} &\text{ if $\be=\qc(j)$ and $p_{j-1}=0$ for some $j=1,\dots,k$}, \\ 
	 \wt\FF'_{\be-1}&\text{ otherwise},
	\end{aligned}
	\right.
\end{equation}
with $\wt\FF_0$ still defined by \eqref{eq:FF_0}; also, the latter definition \eqref{eq:wtFF_be} is in agreement with the previous definition \eqref{eq:wtFFqc}. 

Let $\De$ denote the linear operator on the space of all differentiable functions $f$ that is the operator $f\mapsto f/\EE$ or the operator $f\mapsto f'/\EE$ or the operator $f\mapsto f'$. Then for any non-constant E-polynomial $\FF$ and any of the three choices of $\De$ we have  
\begin{equation}\label{eq:De impl}
\begin{aligned}
	\FF(0)\ge0\ \&\ \De\FF>0\text{ on }(0,\infty) 
	&\implies \FF>0\text{ on }(0,\infty); \\ 
	\FF(0)\le0\ \&\ \De\FF<0\text{ on }(0,\infty) 
	&\implies \FF<0\text{ on }(0,\infty); \\ 
	\FF(0)<0\ \&\ \De\FF>0\text{ on }(0,\infty) 
	&\implies \FF\text{ is $-+$ on }(0,\infty); \\ 
	\FF(0)>0\ \&\ \De\FF<0\text{ on }(0,\infty) 
	&\implies \FF\text{ is $+-$ on }(0,\infty); \\ 
		\FF(0)\ge0\ \&\ \De\FF\text{ is $+-$ on }(0,\infty) 
	&\implies \FF \text{ is $+-$ on }(0,\infty); \\ 
		\FF(0)\le0\ \&\ \De\FF\text{ is $-+$ on }(0,\infty) 
	&\implies \FF \text{ is $-+$ on }(0,\infty). 
\end{aligned}
\end{equation}
The last four implications hold because 
any non-constant E-polynomial function 
either increases to $\infty$ or decreases to $-\infty$ in a neighborhood of $\infty$. 

Note also that for any polynomial $\p$ and any $j=0,1,\dots$ we have $(p\EE^j)'=(j\p+p')E^j$. So, the E-polynomials $\wt\FF_\be$ are non-constant for all $\be=0,\dots,\qc(k)-1$ and therefore all the six implications in \eqref{eq:De impl} hold with $\wt\FF_\be$ in place of $\FF$, for any $\be=0,\dots,\qc(k)-1$. 

Thus, we obtain the following. 

\begin{proposition}\label{prop:hammer}
Let $\FF$ be any E-polynomial of E-degree $k$, as in \eqref{eq:E-poly}. 
\begin{enumerate}[(i)]
	\item If $\wt\FF_{\qc(k)}>0$ on $(0,\infty)$ and $\wt\FF_\be(0)\ge0$ for all $\be=0,\dots,\qc(k)-1$, then $\FF>0$ on $(0,\infty)$. 
	\item If $\wt\FF_{\qc(k)}<0$ on $(0,\infty)$ and $\wt\FF_\be(0)\le0$ for all $\be=0,\dots,\qc(k)-1$, then $\FF<0$ on $(0,\infty)$. 
	\item Suppose that $\wt\FF_{\qc(k)}>0$ on $(0,\infty)$ and there is some  $\be_*\in\{0,\dots,\qc(k)-1\}$ such that for all $\be\in\{0,\dots,\qc(k)-1\}$
	\begin{itemize}
	\item $\wt\FF_\be(0)\le0$ if $\be<\be_*$;
	\item $\wt\FF_\be(0)<0$ if $\be=\be_*$; 
	\item $\wt\FF_\be(0)\ge0$ if $\be>\be_*$.
\end{itemize} 
Then $\FF$ is $-+$ on $(0,\infty)$. 
	\item Suppose that $\wt\FF_{\qc(k)}<0$ on $(0,\infty)$ and there is some  $\be_*\in\{0,\dots,\qc(k)-1\}$ such that for all $\be\in\{0,\dots,\qc(k)-1\}$
	\begin{itemize}
	\item $\wt\FF_\be(0)\ge0$ if $\be<\be_*$;
	\item $\wt\FF_\be(0)>0$ if $\be=\be_*$; 
	\item $\wt\FF_\be(0)\le0$ if $\be>\be_*$.
\end{itemize} 
Then $\FF$ is $+-$ on $(0,\infty)$. 
\end{enumerate}
\end{proposition} 

The signs of $\wt\FF_{\qc(k)}$ and of $\wt\FF_\be(0)$ for all $\be=0,\dots,\qc(k)-1$ can be quickly computed by the command \texttt{hammer} in the Mathematica package hammer.wl. This package was produced by saving the Mathematica notebook hammer.nb as a (plain-text) .wl \break 
(Wolfram Language) file. 
The files hammer.nb and hammer.wl are included with the other Mathematica notebooks in Supplement~B, described on page~\pageref{mathematica}. Put all these files into the same folder.

%The files hammer.nb and hammer.wl are in the same folder, %\break 
%MathematicaFiles, as the other Mathematica notebooks with details of calculations pertaining to the present paper. 

\section{Repeated-hammering proof of the fact that \texorpdfstring{$\B'(r)>0$}{B'} for real \texorpdfstring{$r\ne0$}{r} } \label{app:B}
Refer to Appendix~\ref{app:hammer} for a general description of repeated hammering and corresponding notations. 

Let here $\FF=D\B$. Then, by \eqref{eq:DB}, $\FF$ 
%\begin{equation}
%	(D\B)(r)=\sum_{j=0}^2 p_j(r)e^{jr}
%\end{equation}
is an E-polynomial of E-degree $k=2$, with 
\begin{equation*}
	\p_0(r):=r^2+4 r+2,\quad \p_1(r):=-(r^3-r^2+4 r+4),\quad \p_2(r):=2.
\end{equation*}

The $k=2$ rounds of hammering of this E-polynomial $\FF$ result in the polynomial 
\begin{equation*}
	\wt\FF_7(r)=\wt\FF_{\qc(k)}=16>0.
\end{equation*}

%The two corresponding rounds of repeated hammering produce 
%\begin{equation}
%	(D\B)_3(r):=\frac{(D\B)^{(3)}(r)}{e^r}; \quad 
%	(D\B)_7(r):=\frac{(D\B)_3^{(4)}(r)}{e^r}=16>0.  
%\end{equation}
Also, 
\begin{equation}\label{eq:ders B}
	\big(\sign\wt\FF_0^{(j)}(0)\colon j=0,1,2\big)=(0,0,0), \quad 
	\big(\sign\wt\FF_3^{(j)}(0)\colon j=0,1,2,3\big)=(0,0,0,1).  
\end{equation}
These calculations were done in Mathematica notebook DetailsAppendixB,positive\_r.nb. 

By Proposition~\ref{prop:hammer}(i), $(D\B)(r)=\FF(r)>0$ for all real $r>0$.

The last $1$ in \eqref{eq:ders B} makes the consideration of the sign of $(D\B)(r)$ for real $r<0$ somewhat less straightforward. This is overcome by the following reflection trick. Now redefine $\FF$ by the formula 
\begin{equation*}
	\FF(r):=(D\B)(-r)/e^{-2r}=\sum_{j=0}^2 \p_{2-j}(-r)e^{jr},
\end{equation*} 
which also is an E-polynomial of E-degree $2$.  

The corresponding $k=2$ rounds of repeated hammering of the latter $\FF$ result in 
$\wt\FF_5(r)=\wt\FF_{\qc(k)}(r)=2 r (5 + r)>0$ for $r>0$. 
Also, here 
\begin{equation*}%\label{eq:ders B}
	\wt\FF_0(0)=0=\wt\FF_1^{(i)}(0)  
\end{equation*}
for $i=0,1,2,3$. 
These latter calculations were done in Mathematica notebook 
DetailsAppendixB,negative\_r.nb.

Again by Proposition~\ref{prop:hammer}(i), $\FF(r)>0$ for $r>0$ and hence 
$(D\B)(r)=\FF(-r)e^{2r} >0$ for $r<0$.

Thus, $(D\B)(r)>0$ and hence $\B'(r)>0$ for all real $r\ne0$. \qed 
 
\section{Repeated-hammering proof of \texorpdfstring{\eqref{eq:C1,C2<0}}{{eq:C1,C2<0}}
} \label{app:C1,C2<0}
Refer again to Appendix~\ref{app:hammer} for a general description of repeated hammering and corresponding notations. 

\subsection{Proving \texorpdfstring{\eqref{eq:C1<0}}{{eq:C1<0}}
} \label{app:C1<0}

In this subsection, let $\FF=C_1$, %so that  
%\begin{equation}
%	C_1(r)=\sum_{j=0}^7 p_j(r)e^{jr}, 
%\end{equation}
which is an E-polynomial of E-degree $k=7$, with 
\begin{equation*}
\begin{aligned}
	\p_0(r)&:=-8,\\ 
	\p_1(r) &:= -2 r^4-20 r^3-20 r^2-40 r+40, \\
	\p_2(r) &:= -3 r^6-6 r^5-38 r^4+44 r^3+60 r^2+200 r-72, \\
	\p_3(r) &:= -4 r^7-5 r^6-46 r^5+40 r^4-32 r^3-40 r^2-400 r+40, \\
	\p_4(r) &:= 4 r^7-5 r^6+46 r^5+40 r^4+32 r^3-40 r^2+400 r+40, \\
	\p_5(r) &:= -3 r^6+6 r^5-38 r^4-44 r^3+60 r^2-200 r-72, \\
	\p_6(r) &:= -2 r^4+20 r^3-20 r^2+40 r+40, \\
	\p_7(r) &:= -8,
\end{aligned}
\end{equation*}
of degrees 
\begin{equation*}
	(d_0,\dots,d_7)=(0, 4, 6, 7, 7, 6, 4,0). 
\end{equation*}
The $k=7$ corresponding rounds of repeated hammering of this $\FF$ result in 
\begin{equation*}
	\wt\FF_{\qc(k)}(r)=-1872381094133760000000<0.
\end{equation*}
Also, here $\wt\FF_{\qc(j-1)}^{(i)}(0)\le0$ for all $j=1,\dots,7$ and $i=0,\dots,\q(j)-1$. 
These calculations were done in Mathematica notebook DetailsAppendixC1.nb. 

By Proposition~\ref{prop:hammer}(ii), 
\begin{equation}\label{eq:C1<0}
	C_1(r)=\FF(r)<0\text{ if }r>0. 
\end{equation}

\subsection{Proving \texorpdfstring{\eqref{eq:C2<0}}{{eq:C2<0}}
} \label{app:C2<0}

In this subsection, let $\FF=C_2$, 
which is an E-polynomial of E-degree $k=4$, with 
\begin{equation*}
\begin{aligned}
	\p_0(r)&:=r^3+8 r^2+16 r+8,\\ 
	\p_1(r) &:= r^5+4 r^4+14 r^3+16 r^2-32 r-32, \\
	\p_2(r) &:=48-48 r^2, \\
	\p_3(r) &:= -r^5+4 r^4-14 r^3+16 r^2+32 r-32, \\
	\p_4(r) &:= -r^3+8 r^2-16 r+8, 
\end{aligned}
\end{equation*}
of degrees 
\begin{equation*}
	(d_0,\dots,d_4)=(3,5,2,5,3). 
\end{equation*}
The $k=4$ rounds of repeated hammering of the latter $\FF$ result in 
\begin{equation*}
	\wt\FF_{\qc(k)}(r)=-41472 (36 r^3+846 r^2+5607 r+9676)<0
\end{equation*}
for $r>0$. 
Also, here $\wt\FF_{\qc(j-1)}^{(i)}(0)\le0$ for all $j=1,\dots,4$ and $i=0,\dots,\q(j)-1$.  
These calculations were done in Mathematica notebook %\break 
DetailsAppendixC2.nb.

Again by Proposition~\ref{prop:hammer}(ii), 
\begin{equation}\label{eq:C2<0}
	C_2(r)=\FF(r)<0\text{ if }r>0. 
\end{equation}

\medskip
\hrule
\medskip

Thus, $C_1(r)<0$ and $C_2(r)<0$ if $r>0$.

To complete the proof of \eqref{eq:C1,C2<0}, it remains to note the symmetries 
\begin{equation*}
	C_1(-r)=e^{-7r}C_1(r)\quad\text{and}\quad C_2(-r)=e^{-4r}C_2(r)
\end{equation*}
for all real $r$. 
\qed 

\section{Repeated-hammering proof of \texorpdfstring{\eqref{eq:DDrhoB(r)<0}}{app:DDrhoB}
}\label{app:DDrhoB}

Let in this appendix 
\begin{equation*}
	\FF(r):=\frac{(DD\rho\B)(-r)}{e^{-6r}}
\end{equation*}
for real $r$,  
%Then  
%\begin{equation}
%	\FF(r)=\sum_{j=0}^6 p_j(r)e^{jr}
%\end{equation}
which is an E-polynomial of E-degree $6$, with 
\begin{equation*}
\begin{aligned}
	\p_0(r)&:=8,\\ 
	\p_1(r) &:= -2 r^5-16 r^4-28 r^3-28 r^2+24 r-32, \\
	\p_2(r) &:=-2 r^7-7 r^6-32 r^5-36 r^4-24 r^3+72 r^2-96 r+40, \\
	\p_3(r) &:= 2 r^7+22 r^5+48 r^4+136 r^3-48 r^2+144 r, \\
	\p_4(r) &:= -r^6-4 r^5+4 r^4-88 r^3-8 r^2-96 r-40, \\ 
	\p_5(r) &:= 4 r^3+12 r^2+24 r+32, \\ 
	\p_6(r) &:=  -8
\end{aligned}
\end{equation*}
of degrees 
\begin{equation*}
	(d_0,\dots,d_6)=(0,5,7,7,6,3,0). 
\end{equation*}
The $k=6$ corresponding rounds of repeated hammering of the current $\FF$ result in 
\begin{equation*}
	\wt{\FF}_{\qc(6)}(r)=-41278242816000000<0.
\end{equation*}
Also, $\wt{\FF}_{\qc(j-1)}^{(i)}(0)\le0$ for all $j=1,\dots,6$ and $i=0,\dots,\q(j)-1$. 
These calculations were done in Mathematica notebook DetailsAppendixD.nb. 

Again by Proposition~\ref{prop:hammer}(ii), $\FF(r)<0$ if $r>0$ and hence $(DD\rho\B)(r)<0$ if $r<0$. \qed

\section{Repeated-hammering proof of \texorpdfstring{\eqref{eq:DrhoB}}{app:DDrhoB}
}\label{app:DrhoB}

\subsection{Repeated-hammering proof of \texorpdfstring{\eqref{eq:DrhoB}}{app:DDrhoB}, 
case \texorpdfstring{$r<0$}{r<0}
}\label{app:DrhoB,r<0}

To consider the case $r<0$, let  
\begin{equation*}
	\FF(r):=\frac{(D\rho\B)(-r)}{e^{-4r}}, 
\end{equation*}
%Then  
%\begin{equation}
%	\FF(r)=\sum_{j=0}^4 p_j(r)e^{jr}
%\end{equation}
which is an E-polynomial of E-degree $k=4$, with 
\begin{equation*}
\begin{aligned}
	\p_0(r)&:=-4 (r^2+2 r+2),\\ 
	\p_1(r) &:= r^6+5 r^5+16 r^4+20 r^3+24 r^2+20 r+32, \\
	\p_2(r) &:= 3 r^5-4 r^4-12 r^3-36 r^2-12 r-48, \\
	\p_3(r) &:= -8 r^3+16 r^2-4 r+32, \\
	\p_4(r) &:= 4 r-8 
\end{aligned}
\end{equation*}
of degrees 
\begin{equation*}
	(d_0,\dots,d_4)=(2,6,5,3,1). 
\end{equation*}
The $k=4$ corresponding rounds of repeated hammering of the current $\FF$ result in 
%starting with 
%\begin{equation}
%	\wt{f}_0(r):=f(r) 
%\end{equation}
%successively produce 
%\begin{equation}
%	\wt{f}_{\qc(j)}(r):=\frac{\wt{f}_{\qc(j-1)}^{(\q(j))}(r)}{e^r} 
%\end{equation}
%for $j=1,\dots,4$, where 
%\begin{equation}
%	\q(j):=d_{j-1}+1,\quad \qc(j):=\sum_{i=1}^j\q(i)\ \text{(with $\qc(0):=0$)}.
%\end{equation}
%
%We have 
\begin{equation*}
	\wt{\FF}_{\qc(4)}(r)=2985984 (97 + 12 r)>0 
\end{equation*}
for $r>0$. 
Also, $\wt{\FF}_{\qc(j-1)}^{(i)}(0)\ge0$ for all $j=1,\dots,4$ and $i=0,\dots,\q(j)-1$. 
These calculations were done in Mathematica notebook DetailsAppendixE1.nb. 

By Proposition~\ref{prop:hammer}(i), $\FF(r)>0$ if $r>0$ and hence 
\begin{equation}\label{eq:DrhoB>0,r<0}
\text{$(D\rho\B)(r)>0$ if $r<0$.}	
\end{equation}
 
\subsection{Repeated-hammering proof of \texorpdfstring{\eqref{eq:DrhoB}}{app:DDrhoB}, 
case \texorpdfstring{$r>0$}{r>0}
}\label{app:DrhoB,r>0}

To consider the case $r>0$, let   
\begin{equation*}
	\FF(r):=(D\rho\B)(r),
\end{equation*}
which is an E-polynomial of E-degree $4$, with 
\begin{equation*}
\begin{aligned}
	\p_0(r)&:=-4 (r+2),\\ 
	\p_1(r) &:= 8 r^3+16 r^2+4 r+32, \\
	\p_2(r) &:= -3 r^5-4 r^4+12 r^3-36 r^2+12 r-48, \\
	\p_3(r) &:= r^6-5 r^5+16 r^4-20 r^3+24 r^2-20 r+32, \\
	\p_4(r) &:= -4 r^2+8 r-8 
\end{aligned}
\end{equation*}
of degrees 
\begin{equation*}
	(d_0,\dots,d_4)=(1,3,5,6,2). 
\end{equation*}
The $k=4$ corresponding rounds of repeated hammering of the latter $\FF$ result in  
\begin{equation*}
	\wt{\FF}_{\qc(4)}(r)=-13824 (2623 + 520 r + 24 r^2)<0
\end{equation*}
for $r>0$. 

Next, $\wt{\FF}_{\qc(j-1)}^{(i)}(0)\ge0$ for $j=1,2,3$ and all $i=0,\dots,\q(j)-1$, 
$\wt{\FF}_{\qc(j-1)}^{(i)}(0)>0$ for $j=4$ and $i=0,1$, and  $\wt{\FF}_{\qc(j-1)}^{(i)}(0)<0$ for $j=4$ and $i=2,\dots,\q(4)-1$. 
These calculations were done in Mathematica notebook DetailsAppendixE2.nb. 

By Proposition~\ref{prop:hammer}(iv),    
$\FF$ is $+-$ on $(0,\infty)$. 
We also have 
\begin{equation}\label{eq:0.329}
	\FF(3290/1000)>0>\FF(3291/1000).
\end{equation}
Therefore and because here $\FF=D\rho\B$, for some real $r_*\ap3.29$ and all real $r>0$ we have 
$(D\rho\B)(r)>0\iff r<r_*$. It is also clear (cf.~\eqref{eq:0.329}) that $r_*$ can be effectively computed to any precision. 

\medskip
\hrule
\medskip

To complete the proof of \eqref{eq:DrhoB}, it remains to recall \eqref{eq:DrhoB>0,r<0}. \qed 

\section{Repeated-hammering proof of \texorpdfstring{\eqref{eq:c_0}}{app:c0}
}\label{app:c0}

\subsection{Repeated-hammering proof of \texorpdfstring{\eqref{eq:c_0}}{app:c0,r>0}, case \texorpdfstring{$r>0$}{c0,r>0}
}\label{app:c0,r>0}

To consider the case $r>0$, let  
\begin{equation*}
	\FF(r):=c_0(r), 
\end{equation*}
which is an E-polynomial of E-degree $k=6$, with 
\begin{equation*}
\begin{aligned}
	\p_0(r)&:=4 (2 + r)^2, \\ 
	\p_1(r) &:= -r^6-4 r^5-4 r^4+8 r^3+12 r^2-48 r-96, \\
	\p_2(r) &:= -2 r^7-15 r^6-44 r^5-52 r^4-56 r^3-80 r^2+240, \\
	\p_3(r) &:= 2 r^7+r^6+44 r^5+108 r^4+120 r^3+120 r^2+160 r-320, \\
	\p_4(r) &:= -r^6+4 r^5-44 r^4-104 r^3-60 r^2-240 r+240, \\
	\p_5(r) &:= -8 r^4+32 r^3-4 r^2+144 r-96, \\
	\p_6(r) &:= 8 r^2-32 r+16  
\end{aligned}
\end{equation*}
of degrees 
\begin{equation*}
	(d_0,\dots,d_6)=(2,6,7,7,6,4,2). 
\end{equation*}
The $k=6$ corresponding rounds of repeated hammering of the current $\FF$ result in 
\begin{equation*}
	\wt{\FF}_{\qc(6)}(r)=33022594252800000 (36053 + 5880 r + 225 r^2)>0 
\end{equation*}
for $r>0$. 
Also, $\wt{\FF}_{\qc(j-1)}^{(i)}(0)\ge0$ for all $j=1,\dots,6$ and $i=0,\dots,\q(j)-1$. 
These calculations were done in Mathematica notebook DetailsAppendixF1.nb. 

By Proposition~\ref{prop:hammer}(i), $\FF(r)>0$ if $r>0$ and hence 
\begin{equation}\label{eq:c0>0,r>0}
\text{$c_0(r)(e^r-1)=\FF(r)(e^r-1)>0$ if $r>0$.}	
\end{equation} 

\subsection{Repeated-hammering proof of \texorpdfstring{\eqref{eq:c_0}}{app:c0,r<0}, case \texorpdfstring{$r<0$}{c0,r<0}
}\label{app:c0,r<0}

To consider the case $r<0$, let   
\begin{equation*}
	\FF(r):=\frac{c_0(-r)}{e^{-6r}},
\end{equation*}
which is an E-polynomial of E-degree $6$, with 
\begin{equation*}
\begin{aligned}
	\p_0(r)&:=8 (r^2+4 r + 2),\\ 
	\p_1(r) &:= -8 r^4-32 r^3-4 r^2-144 r-96, \\
	\p_2(r) &:= -r^6-4 r^5-44 r^4+104 r^3-60 r^2+240 r+240, \\
	\p_3(r) &:= -2 r^7+r^6-44 r^5+108 r^4-120 r^3+120 r^2-160 r-320, \\
	\p_4(r) &:= 2 r^7-15 r^6+44 r^5-52 r^4+56 r^3-80 r^2+240, \\
	\p_5(r) &:= -r^6+4 r^5-4 r^4-8 r^3+12 r^2+48 r-96, \\
	\p_6(r) &:= 4 r^2-16 r+16 
\end{aligned}
\end{equation*}
of degrees 
\begin{equation*}
	(d_0,\dots,d_6)=(2,4,6,7,7,6,2). 
\end{equation*}
The $k=6$ corresponding rounds of repeated hammering of the latter $\FF$ result in  
\begin{equation*}
	\wt{\FF}_{\qc(6)}(r)=206391214080000 (76283 + 10740 r + 360 r^2)>0
\end{equation*}
for $r>0$. 

Next, $\wt{\FF}_{\qc(j-1)}^{(i)}(0)\le0$ for $j=1,2,3,4$ and all $i=0,\dots,\q(j)-1$, $\wt{\FF}_{\qc(j-1)}^{(i)}(0)<0$ for $j=5$ and $i=0$ -- whereas $\wt{\FF}_{\qc(j-1)}^{(i)}(0)>0$ for $j=5$ and $i=1,\dots,\q(5)-1$ and for $j=6$ and all $i=0,\dots,\q(6)-1$. 
These calculations were done in Mathematica notebook DetailsAppendixF2.nb.

By Proposition~\ref{prop:hammer}(iii),  
$\FF$ is $-+$ on $(0,\infty)$. We also have 
\begin{equation}\label{eq:0.542}
	\FF(5424/1000)<0<\FF(5425/1000).
\end{equation}
Therefore and because here $\FF(r):=\frac{c_0(-r)}{e^{-6r}}$, for some real $\tr_2\ap-5.42$ and all real $r<0$ we have 
$c_0(r)(e^r-1)=\FF(-r)e^{6r}(e^r-1)\le0\iff r\le\tr_2$. It is also clear that $\tr_2$ can be effectively computed to any precision.  

\medskip
\hrule
\medskip

To complete the proof of \eqref{eq:c_0}, it remains to recall \eqref{eq:c0>0,r>0}. \qed  

\section{Repeated-hammering proof of \texorpdfstring{\eqref{eq:c_1}}{app:c1}
}\label{app:c1}

\subsection{Repeated-hammering proof of \texorpdfstring{\eqref{eq:c_1}}{app:c1,r>0}, case \texorpdfstring{$r>0$}{c1,r>0}
}\label{app:c1,r>0}

To consider the case $r>0$, let  
\begin{equation*}
	\FF(r):=c_1(r), 
\end{equation*}
which is an E-polynomial of E-degree $k=6$, with 
\begin{equation*}
\begin{aligned}
	\p_0(r)&:=2 (2 + r)^2, \\ 
	\p_1(r) &:= r^5+6 r^4+12 r^3+10 r^2-24 r-48, \\
	\p_2(r) &:= -3 r^6-11 r^5-28 r^4-44 r^3-56 r^2+120, \\
	\p_3(r) &:= 7 r^5+36 r^4+60 r^3+84 r^2+80 r-160, \\
	\p_4(r) &:= -r^6+3 r^5-12 r^4-36 r^3-46 r^2-120 r+120, \\
	\p_5(r) &:= -2 r^4+8 r^3+2 r^2+72 r-48, \\
	\p_6(r) &:= 4 r^2-16 r+8  
\end{aligned}
\end{equation*}
of degrees 
\begin{equation*}
	(d_0,\dots,d_6)=(2,5,6,5,6,4,2). 
\end{equation*}
The $k=6$ corresponding rounds of repeated hammering of the current $\FF$ result in 
\begin{equation*}
	\wt{\FF}_{\qc(6)}(r)=515978035200000 (5305 + 956 r + 40 r^2)>0 
\end{equation*}
for $r>0$. 
Also, $\wt{\FF}_{\qc(j-1)}^{(i)}(0)\ge0$ for all $j=1,\dots,6$ and $i=0,\dots,\q(j)-1$. 
These calculations were done in Mathematica notebook DetailsAppendixG1.nb.

By Proposition~\ref{prop:hammer}(i), $\FF(r)>0$ if $r>0$ and hence 
\begin{equation}\label{eq:c1>0,r>0}
\text{$c_1(r)=\FF(r)>0$ if $r>0$.}	
\end{equation}

\subsection{Repeated-hammering proof of \texorpdfstring{\eqref{eq:c_1}}{app:c1,r<0}, case \texorpdfstring{$r<0$}{c1,r<0}
}\label{app:c1,r<0}

To consider the case $r<0$, let  
\begin{equation*}
	\FF(r):=\frac{c_1(-r)}{e^{-6r}}, 
\end{equation*}
which is an E-polynomial of E-degree $k=6$, with 
\begin{equation*}
\begin{aligned}
	\p_0(r)&:=4 (r^2 + 4 r + 2), \\ 
	\p_1(r) &:= -2 r^4-8 r^3+2 r^2-72 r-48, \\
	\p_2(r) &:= -r^6-3 r^5-12 r^4+36 r^3-46 r^2+120 r+120, \\
	\p_3(r) &:= -7 r^5+36 r^4-60 r^3+84 r^2-80 r-160, \\
	\p_4(r) &:= -3 r^6+11 r^5-28 r^4+44 r^3-56 r^2+120, \\
	\p_5(r) &:= -r^5+6 r^4-12 r^3+10 r^2+24 r-48, \\
	\p_6(r) &:= 2 r^2-8 r+8  
\end{aligned}
\end{equation*}
of degrees 
\begin{equation*}
	(d_0,\dots,d_6)=(2,4,6,5,6,5,2). 
\end{equation*}
The $k=6$ corresponding rounds of repeated hammering of the current $\FF$ result in 
\begin{equation*}
	\wt{\FF}_{\qc(6)}(r)=51597803520000 (6137 + 1020 r + 40 r^2)>0 
\end{equation*}
for $r>0$. 
Also, $\wt{\FF}_{\qc(j-1)}^{(i)}(0)\ge0$ for all $j=1,\dots,6$ and $i=0,\dots,\q(j)-1$. 
These calculations were done in Mathematica notebook DetailsAppendixG2.nb.

By Proposition~\ref{prop:hammer}(i), $\FF(r)>0$ if $r>0$ and hence 
\begin{equation*}%\label{eq:c1>0,r>0}
\text{$c_1(r)=\FF(-r)e^{6r}>0$ if $r<0$.}	
\end{equation*}

\medskip
\hrule
\medskip

To complete the proof of \eqref{eq:c_1}, it remains to recall \eqref{eq:c1>0,r>0}. \qed  

\section{Repeated-hammering proof of \texorpdfstring{\eqref{eq:V}}{app:V}
}\label{app:V}

\subsection{Repeated-hammering proof of \texorpdfstring{\eqref{eq:V}}{app:V,r>0}, case \texorpdfstring{$r>0$}{V,r>0}
}\label{app:V,r>0}

To consider the case $r>0$, let  
\begin{equation*}
	\FF(r):=V(r), 
\end{equation*}
which is an E-polynomial of E-degree $k=9$, with 
\begin{equation*}
\begin{aligned}
	\p_0(r)&:=-8 (r^3+6r^2+14r+8), \\ 
	\p_1(r) &:= r^7+4 r^6-48 r^4-64 r^3+96 r^2+768 r+576, \\
	\p_2(r) &:= 4 r^8+30 r^7+136 r^6+272 r^5+528 r^4+560 r^3+640 r^2-2112 r-2304, \\
	\p_3(r) &:= 4 r^9-53 r^7-404 r^6-1024 r^5-1968 r^4-1568 r^3-3136 r^2+2688
   r+5376, \\
	\p_4(r) &:= 12 r^9+40 r^8+100 r^7+432 r^6+1360 r^5+3600 r^4+2240 r^3+6048
   r^2-672 r-8064, \\
	\p_5(r) &:= -4 r^9-32 r^8-153 r^7-292 r^6-640 r^5-3600 r^4-1792 r^3-6272
   r^2-2688 r+8064, \\
	\p_6(r) &:= 4 r^9-12 r^8+62 r^7+200 r^6-80 r^5+1968 r^4+784 r^3+3584 r^2+4032
   r-5376, \\
	\p_7(r) &:= 13 r^7-76 r^6+128 r^5-528 r^4-160 r^3-960 r^2-2688 r+2304, \\
	\p_8(r) &:= -16 r^5+48 r^4+8 r^3+16 r^2+912 r-576, \\
	\p_9(r) &:= 32 r^2-128 r+64  
\end{aligned}
\end{equation*}
of degrees 
\begin{equation*}
	(d_0,\dots,d_9)=(3,7,8,9,9,9,9,7,5,2). 
\end{equation*}
The $k=9$ corresponding rounds of repeated hammering of the current $\FF$ result in 
\begin{multline*}
	\wt{\FF}_{\qc(9)}(r)=201120989295888709478973981333651456000000000 \\ 
	\times(2622764 + 267645 r + 6615 r^2)>0 
\end{multline*}
for $r>0$. 
Also, $\wt{\FF}_{\qc(j-1)}^{(i)}(0)\ge0$ for all $j=1,\dots,9$ and $i=0,\dots,\q(j)-1$. 
These calculations were done in Mathematica notebook DetailsAppendixH1.nb.

By Proposition~\ref{prop:hammer}(i), $\FF(r)>0$ if $r>0$ and hence 
\begin{equation}\label{eq:V>0,r>0}
\text{$V(r)=\FF(r)>0$ if $r>0$.}	
\end{equation}

\subsection{Repeated-hammering proof of \texorpdfstring{\eqref{eq:V}}{app:V,r<0}, case \texorpdfstring{$r<0$}{V,r<0}
}\label{app:V,r<0}

To consider the case $r<0$, let  
\begin{equation*}
	\FF(r):=\frac{V(-r)}{e^{-9r}}, 
\end{equation*}
which is an E-polynomial of E-degree $k=9$, with 
\begin{equation*}
\begin{aligned}
	\p_0(r)&:=32 (r^2+4r+2), \\ 
	\p_1(r) &:= 16 r^5+48 r^4-8 r^3+16 r^2-912 r-576, \\
	\p_2(r) &:= -13 r^7-76 r^6-128 r^5-528 r^4+160 r^3-960 r^2+2688 r+2304, \\
	\p_3(r) &:= -4 r^9-12 r^8-62 r^7+200 r^6+80 r^5+1968 r^4-784 r^3+3584
   r^2-4032 r-5376, \\
	\p_4(r) &:= 4 r^9-32 r^8+153 r^7-292 r^6+640 r^5-3600 r^4+1792 r^3-6272
   r^2+2688 r+8064, \\
	\p_5(r) &:= -12 r^9+40 r^8-100 r^7+432 r^6-1360 r^5+3600 r^4-2240 r^3+6048
   r^2+672 r-8064, \\
	\p_6(r) &:= -4 r^9+53 r^7-404 r^6+1024 r^5-1968 r^4+1568 r^3-3136 r^2-2688
   r+5376, \\
	\p_7(r) &:= 4 r^8-30 r^7+136 r^6-272 r^5+528 r^4-560 r^3+640 r^2+2112 r-2304, \\
	\p_8(r) &:= -r^7+4 r^6-48 r^4+64 r^3+96 r^2-768 r+576, \\
	\p_9(r) &:= 8 r^3-48 r^2+112 r-64  
\end{aligned}
\end{equation*}
of degrees 
\begin{equation*}
	(d_0,\dots,d_9)=(2,5,7,9,9,9,9,8,7,3). 
\end{equation*}
The $k=9$ corresponding rounds of repeated hammering of the current $\FF$ result in 
\begin{multline*}
	\wt{\FF}_{\qc(9)}(r)=8659930163547897612458577100800000000 \\ 
	\times (194015255897 + 27531590310 r + 1270306800 r^2 + 19051200 r^3)>0 
\end{multline*}
for $r>0$. 

Next, $\wt{\FF}_{\qc(j-1)}^{(i)}(0)\le0$ for $j=1,2,3,4$ and all $i=0,\dots,\q(j)-1$, and also for $j=5$ and $i=0,1$; 
$\wt{\FF}_{\qc(j-1)}^{(i)}(0)<0$ for $j=5$ and $i=2$; 
and $\wt{\FF}_{\qc(j-1)}^{(i)}(0)>0$ for $j=5$ and $i=3,\dots,\q(5)-1$ and for $j=6,7,8,9$ and all $i=0,\dots,\q(j)-1$. 
These calculations were done in Mathematica notebook DetailsAppendixH2.nb.

By Proposition~\ref{prop:hammer}(iii), 
$\FF$ is $-+$ on $(0,\infty)$. We also have 
\begin{equation}\label{eq:0.304}
	\FF(3041/1000)<0<\FF(3042/1000).
\end{equation}
Therefore and because here $\FF(r):=\frac{V(-r)}{e^{-9r}}$, for some real $\tr_3\ap-3.04$ and all real $r<0$ we have 
$V(r)=\FF(-r)e^{9r}>0\iff r<\tr_3$. It is also clear that $\tr_3$ can be effectively computed to any precision.

\medskip
\hrule
\medskip

To complete the proof of \eqref{eq:V}, it remains to recall \eqref{eq:V>0,r>0}. \qed

\section{Repeated-hammering proof of the inequalities  \texorpdfstring{\eqref{eq:crat}}{app:V}
}\label{app:crat} 

First here, let 
\begin{equation*}
	\FF(r):=\wt{c\rat}(r), 
\end{equation*}
which is an E-polynomial of E-degree $k=4$, with 
\begin{equation*}
\begin{aligned}
	\p_0(r)&:=-4, \\ 
	\p_1(r) &:= -21 r^2-100 r+16, \\
	\p_2(r) &:= 50 r^3-8 r^2+200 r-24, \\
	\p_3(r) &:= 29 r^2-100 r+16, \\
	\p_4(r) &:= -4  
\end{aligned}
\end{equation*}
of degrees 
\begin{equation*}
	(d_0,\dots,d_4)=(0,2,3,2,0). 
\end{equation*}
The $k=4$ corresponding rounds of repeated hammering of the current $\FF$ result in 
\begin{equation*}
	\wt{\FF}_{\qc(4)}(r)=-6912<0 
\end{equation*}
for $r>0$. 
Also, $\wt{\FF}_{\qc(j-1)}^{(i)}(0)\le0$ for all $j=1,\dots,4$ and $i=0,\dots,\q(j)-1$. 
These calculations were done in Mathematica notebook DetailsAppendixIpart1.nb.

By Proposition~\ref{prop:hammer}(ii), $\FF(r)<0$ if $r>0$ and hence 
\begin{equation}\label{eq:wtcrat<0,r>0}
\text{$\wt{c\rat}(r)=\FF(r)<0$ if $r>0$.}	
\end{equation}

\medskip
\hrule
\medskip

Now let 
\begin{equation*}
	\FF(r):=c\rat(r), 
\end{equation*}
which is an E-polynomial of E-degree $k=2$, with 
\begin{equation*}
\begin{aligned}
	\p_0(r)&:=-r-4, \\ 
	\p_1(r) &:= 2 r^2+8, \\
	\p_2(r) &:= r-4 
\end{aligned}
\end{equation*}
of degrees 
\begin{equation*}
	(d_0,\dots,d_2)=(1,2,1). 
\end{equation*}
The $k=2$ corresponding rounds of repeated hammering of the current $\FF$ result in 
\begin{equation*}
	\wt{\FF}_{\qc(2)}(r)=4r>0 
\end{equation*}
for $r>0$. 
Also, $\wt{\FF}_{\qc(j-1)}^{(i)}(0)=0$ for $j=1,2$ and $i=0,\dots,\q(j)-1$. 
These calculations were done in Mathematica notebook DetailsAppendixIpart2.nb.

By Proposition~\ref{prop:hammer}(i), $\FF(r)>0$ if $r>0$ and hence 
\begin{equation}\label{eq:crat>0,r>0}
\text{$c\rat(r)=\FF(r)>0$ if $r>0$.}	
\end{equation}

\medskip
\hrule
\medskip

Now inequalities \eqref{eq:crat} follow immediately from \eqref{eq:crat>0,r>0} and \eqref{eq:wtcrat<0,r>0}. \qed 

\section{Repeated-hammering proof of \texorpdfstring{\eqref{eq:Dh'B}}{app:h'B}
}\label{app:h'B} 

Let here   
\begin{equation*}
	\FF(r):=\wt{\h'\B}(r), 
\end{equation*}
which is an E-polynomial of E-degree $k=4$, with 
\begin{equation*}
\begin{aligned}
	\p_0(r)&:=-6, \\ 
	\p_1(r) &:= 11 r^2+10 r+24, \\
	\p_2(r) &:= 5 r^3-2 r^2-20 r-36, \\
	\p_3(r) &:= 5 r^3-9 r^2+10 r+24, \\
	\p_4(r) &:= -6  
\end{aligned}
\end{equation*}
of degrees 
\begin{equation*}
	(d_0,\dots,d_4)=(0,2,3,3,0). 
\end{equation*}
The $k=4$ corresponding rounds of repeated hammering of the current $\FF$ result in 
\begin{equation*}
	\wt{\FF}_{\qc(4)}(r)=-10368<0 
\end{equation*}
for $r>0$. 

Next, $\wt{\FF}_{\qc(j-1)}^{(i)}(0)=0$ for $j=1,2$ and all $i=0,\dots,\q(j)-1$, 
\begin{equation*}
\begin{aligned}
	\Big(\sign\wt{\FF}_{\qc(3-1)}^{(i)}(0)\colon  i=0,1,2,3\Big)
	&=(0, 0, -1, 1), \\ 
	\Big(\sign\wt{\FF}_{\qc(4-1)}^{(i)}(0)\colon  i=0,1,2,3\Big)
	&=(-1,-1, -1, -1).	
\end{aligned}
\end{equation*}
These calculations were done in Mathematica notebook DetailsAppendixJ.nb, as well as the calculations for statements \eqref{eq:wtFF_qc(3-1)^{(3)} signs} and \eqref{eq:maps into} below. 

By Proposition~\ref{prop:hammer}(iv), $\wt{\FF}_{\qc(3-1)}^{(3)}$ is $+-$ on $(0,\infty)$, whence 
\begin{equation}\label{eq:wtFFqc(3-1)^{(2)}}
	\text{$\wt{\FF}_{\qc(3-1)}^{(2)}$ is $\u\d$ on $(0,\infty)$, with  $\wt{\FF}_{\qc(3-1)}^{(2)}(0)<0$ and $\wt{\FF}_{\qc(3-1)}^{(2)}(\infty-)=-\infty<0$.}
\end{equation}
We also have 
\begin{equation}\label{eq:wtFF_qc(3-1)^{(3)} signs}
	\wt{\FF}_{\qc(3-1)}^{(3)}\Big(\frac{177}{1000}\Big)
	>0>\wt{\FF}_{\qc(3-1)}^{(3)}\Big(\frac{178}{1000}\Big), 
\end{equation}
so that the only switch in sign of the function $\wt{\FF}_{\qc(3-1)}^{(3)}=\big(\wt{\FF}_{\qc(3-1)}^{(2)}\big)'$ on $(0,\infty)$, from $+$ to $-$, occurs at some point in the interval $[\frac{177}{1000},\frac{178}{1000}]$. Using the Mathematica command Interval, we see that %the function 
\begin{equation}\label{eq:maps into}
	\text{the function $\wt{\FF}_{\qc(3-1)}^{(2)}$ maps the interval $[\tfrac{177}{1000},\tfrac{178}{1000}]$ into the interval $[-62,-46]\subset(-\infty,0)$.}
\end{equation}
%
%$\wt{\FF}_{\qc(3-1)}^{(2)}$ maps the interval $[\frac{177}{1000},\frac{178}{1000}]$ into the interval $[-62,-46]\subset(-\infty,0)$. 
Recalling here \eqref{eq:wtFFqc(3-1)^{(2)}}, we see that $\wt{\FF}_{\qc(3-1)}^{(2)}<0$ on $(0,\infty)$. 
So, by the second line of \eqref{eq:De impl} -- applied repeatedly, backwards to $\wt{\FF}_5=\wt{\FF}_4^{(1)}=\wt{\FF}_{\qc(3-1)}^{(1)},\wt{\FF}_4,\wt{\FF}_3,\wt{\FF}_2,\wt{\FF}_1,\wt{\FF}_0=\FF$ 
%instead of $\FF$ in that line and then to $\wt{\FF}_{\qc(3-1)}^{(0)}$ instead of $\FF$ in that line 
-- we have 
$\wt{\h'\B}=\FF<0$ on $(0,\infty)$. 
\qed

\newpage
 
\bibliographystyle{abbrv}
%\bibliographystyle{ims}
%\bibliography{are.citations}
%\bibliography{citat}

%\bibliography{citations}

\bibliography%{P:/pCloudSync/mtu_pCloud_02-02-17/bib_files/citations01-09-20}
{C:/Users/ipinelis/Documents/pCloudSync/mtu_pCloud_02-02-17/bib_files/citations10-06-26}

\def\cprime{$'$} \def\polhk#1{\setbox0=\hbox{#1}{\ooalign{\hidewidth
  \lower1.5ex\hbox{`}\hidewidth\crcr\unhbox0}}}
  \def\polhk#1{\setbox0=\hbox{#1}{\ooalign{\hidewidth
  \lower1.5ex\hbox{`}\hidewidth\crcr\unhbox0}}}
  \def\polhk#1{\setbox0=\hbox{#1}{\ooalign{\hidewidth
  \lower1.5ex\hbox{`}\hidewidth\crcr\unhbox0}}} \def\cprime{$'$}
  \def\polhk#1{\setbox0=\hbox{#1}{\ooalign{\hidewidth
  \lower1.5ex\hbox{`}\hidewidth\crcr\unhbox0}}} \def\cprime{$'$}
  \def\polhk#1{\setbox0=\hbox{#1}{\ooalign{\hidewidth
  \lower1.5ex\hbox{`}\hidewidth\crcr\unhbox0}}} \def\cprime{$'$}
  \def\cprime{$'$}
\begin{thebibliography}{10}

\bibitem{barbagallo_etal}
M.~L. Barbagallo, G.~Jeronimo, and J.~Sabia.
\newblock Decision problem for a class of univariate {P}faffian functions.
\newblock {\em Appl. Algebra Engrg. Comm. Comput.}, 35(2):207--232, 2024.

\bibitem{billingsley99}
P.~Billingsley.
\newblock {\em Convergence of probability measures}.
\newblock Wiley Series in Probability and Statistics: Probability and
  Statistics. John Wiley \& Sons, Inc., New York, second edition, 1999.
\newblock A Wiley-Interscience Publication.

\bibitem{borell74}
C.~Borell.
\newblock Convex measures on locally convex spaces.
\newblock {\em Ark. Mat.}, 12:239--252, 1974.

\bibitem{borell83}
C.~Borell.
\newblock Convexity of measures in certain convex cones in vector space
  {$\sigma $}-algebras.
\newblock {\em Math. Scand.}, 53(1):125--144, 1983.

\bibitem{knuth96}
R.~M. Corless, G.~H. Gonnet, D.~E.~G. Hare, D.~J. Jeffrey, and D.~E. Knuth.
\newblock On the {L}ambert {$W$} function.
\newblock {\em Adv. Comput. Math.}, 5(4):329--359, 1996.

\bibitem{cule-samworth}
M.~Cule and R.~Samworth.
\newblock Theoretical properties of the log-concave maximum likelihood
  estimator of a multidimensional density.
\newblock {\em Electron. J. Stat.}, 4:254--270, 2010.

\bibitem{devroye}
L.~Devroye.
\newblock A simple algorithm for generating random variates with a log-concave
  density.
\newblock {\em Computing}, 33(3-4):247--257, 1984.

\bibitem{Fattorini_1999}
H.~O. Fattorini.
\newblock {\em Infinite Dimensional Optimization and Control Theory}.
\newblock Encyclopedia of Mathematics and its Applications. Cambridge
  University Press, 1999.

\bibitem{feinberg_etal}
E.~A. Feinberg, P.~O. Kasyanov, and N.~V. Zadoianchuk.
\newblock Fatou's lemma for weakly converging probabilities.
\newblock {\em Theory Probab. Appl.}, 58(4):683--689, 2014.

\bibitem{folland}
G.~B. Folland.
\newblock {\em Real analysis}.
\newblock Pure and Applied Mathematics (New York). John Wiley \& Sons Inc., New
  York, 1984.
\newblock Modern techniques and their applications, A Wiley-Interscience
  Publication.

\bibitem{frad-guel}
M.~Fradelizi and O.~Gu\'edon.
\newblock A generalized localization theorem and geometric inequalities for
  convex bodies.
\newblock {\em Adv. Math.}, 204(2):509--529, 2006.

\bibitem{johnson-rogers}
N.~L. Johnson and C.~A. Rogers.
\newblock The moment problem for unimodal distributions.
\newblock {\em Ann. Math. Statistics}, 22:433--439, 1951.

\bibitem{karlin_etal61}
S.~Karlin, F.~Proschan, and R.~E. Barlow.
\newblock Moment inequalities of {P}\'olya frequency functions.
\newblock {\em Pacific J. Math.}, 11:1023--1033, 1961.

\bibitem{keilson72}
J.~Keilson.
\newblock {A Threshold for Log-Concavity for Probability Generating Functions
  and Associated Moment Inequalities}.
\newblock {\em The Annals of Mathematical Statistics}, 43(5):1702 -- 1708,
  1972.

\bibitem{lovasz-vempala}
L.~Lov\'asz and S.~Vempala.
\newblock The geometry of logconcave functions and sampling algorithms.
\newblock {\em Random Structures Algorithms}, 30(3):307--358, 2007.

\bibitem{phelps}
R.~R. Phelps.
\newblock {\em Lectures on {C}hoquet's theorem}.
\newblock D. Van Nostrand Co., Inc., Princeton, N.J.-Toronto, Ont.-London,
  1966.

\bibitem{pin06}
I.~Pinelis.
\newblock On l'{H}ospital-type rules for monotonicity.
\newblock {\em {JIPAM}. J. Inequal. Pure Appl. Math.}, 7(2):Article 40, 19 pp.
  (electronic), 2006.
\newblock \url{https://emis.de/ft/14435}.

\bibitem{pourciau}
B.~H. Pourciau.
\newblock Modern multiplier rules.
\newblock {\em Amer. Math. Monthly}, 87(6):433--452, 1980.

\bibitem{prekopa71}
A.~Pr{\'e}kopa.
\newblock Logarithmic concave measures with application to stochastic
  programming.
\newblock {\em Acta Sci. Math. (Szeged)}, 32:301--316, 1971.

\bibitem{prekopa73}
A.~Pr\'ekopa.
\newblock On logarithmic concave measures and functions.
\newblock {\em Acta Sci. Math. (Szeged)}, 34:335--343, 1973.

\bibitem{rudin}
W.~Rudin.
\newblock {\em Functional analysis}.
\newblock International Series in Pure and Applied Mathematics. McGraw-Hill,
  Inc., New York, second edition, 1991.

\bibitem{wellner-logconcav}
A.~Saumard and J.~A. Wellner.
\newblock Log-concavity and strong log-concavity: a review.
\newblock {\em Stat. Surv.}, 8:45--114, 2014.

\bibitem{schoenberg51}
I.~J. Schoenberg.
\newblock On {P}\'olya frequency functions. {I}. {T}he totally positive
  functions and their {L}aplace transforms.
\newblock {\em J. Analyse Math.}, 1:331--374, 1951.

\bibitem{sharma-etal}
R.~Sharma, R.~Bhandari, and R.~Saini.
\newblock On lower bounds for variance and moments of unimodal distributions.
\newblock {\em Calcutta Statist. Assoc. Bull.}, 67(265-266):1--14, 2015.

\end{thebibliography}

\begin{supplement}
\stitle{%Title of Supplement A.
Supplement A: A Sharp Expectation Bound for
Isotropic Log-Concave Random Variables}\label{suppl}
\sdescription{%Short description of Supplement A.
Exposition by ChatGPT of an answer by Bogdan Grechuk on MathOverflow at %\\  
\url{https://mathoverflow.net/a/513788/36721}. 
% \url{https://mathoverflow.net/questions/512106/maximize-expectation-of-a-function-among-log-concave-random-variables/513788\#513788}. 
 The answer relied on a Lean verification \url{https://github.com/Mirkes/OneShotSeparabilityTheorem/blob/main/LogConcaveFamily6Theorem2_WithExistence_Live.lean}, based on which ChatGPT provided an exposition of the corresponding proof.}
\end{supplement} 

\begin{supplement}
\stitle{%Title of Supplement A.
Supplement B: Mathematica Files}\label{mathematica}
\sdescription{%Short description of Supplement B.
%Zip file MathematicaFiles.zip containing  \break 
Mathematica notebooks and Mathematica package \texttt{hammer.wl} providing details of symbolic and exact numerical calculations in proofs in the paper. 
These Mathematica files are available as ancillary files at \url{https://arxiv.org/src/2610.10289/anc}.
Before using these files, %Mathematica notebooks, %extract 
put all of them %in this zip file 
into a folder/directory of your own. 
}
\end{supplement}
\bigskip

\begin{acks}

The work on this paper has been aided by ChatGPT in a number of ways. Specifically, ChatGPT provided the following: 

\begin{itemize}
	\item an exposition of the proof corresponding to the Lean verification referred to in the MathOverflow answer by Bogdan Grechuk, as was noted at the end of the Section~\ref{discussion};    
	\item help reading some places in \cite{frad-guel}; 
	\item pointer to Milman's theorem and reference \cite[Theorem~3.25]{rudin} to it; 
	\item parametrization of the solutions $(a,b)\in(0,\infty)$ of the equation $a e^{-a}=b e^{-b}$;  
	\item pointers to Lemmas~5.4, 5.5(a), and 5.7 of \cite{lovasz-vempala}; 
	\item a more streamlined proof of Proposition~\ref{prop:L_2}; 
	\item Lemma~\ref{lem:perturb}, used for interchanging a derivative with an integral; 
		\item Lemma~\ref{lem:k>1}; 	
		\item references to \cite{barbagallo_etal} and \cite[Theorem~1.1]{feinberg_etal}; 
		\item pointers to \cite[Theorem~2]{johnson-rogers} and  \cite[Theorem~2.2]{sharma-etal}; 
		\item verification and referee-style editorial suggestions.
\end{itemize}
%
%%Proposition~2.2 \verb!{prop:L_2}! : more streamlined proof  
%
%%Lemma {lem:perturb}, used for interchanging a derivative with an integral 
%
%%Refs. to \cite{barbagallo_etal}, \cite[Theorem~1.1]{feinberg_etal}
%
%verification and editing suggestions, referee-report style
%
%%\cite[Theorem~2]{johnson-rogers} and  \cite[Theorem~2.2]{sharma-etal} 
%
\end{acks}

%%%%%%%%%%%%%%%%%%%%%%%%%%%%%%%%%%%%%%%%%%%%%%%%%%%%%%%%%%%%%%%%%%%
%%                                                               %%
%% You have reached the end of your document.                    %%
%%                                                               %%
%%%%%%%%%%%%%%%%%%%%%%%%%%%%%%%%%%%%%%%%%%%%%%%%%%%%%%%%%%%%%%%%%%%

\end{document}

\maketitle

\begin{abstract}
Let $X$ be a real log-concave random variable with mean zero and variance
one.  We prove the sharp inequality
\[
  \E\bigl[(1+|X|)e^{-|X|}\bigr]\le 2\sqrt2-2.
\]
Equality is attained by the centered symmetric Laplace law with density
$2^{-1/2}e^{-\sqrt2|x|}$.  The proof first uses the generalized localization
theorem of Fradelizi and Gu\'edon to reduce the problem to densities whose
logarithm is affine on at most two intervals.  We then give a
finite-dimensional variational argument.  A compact quantile-density
parametrization supplies attainment; two zero-counting arguments exclude all
nonterminal maximizers; and an elementary one-variable calculation treats
the remaining asymmetric-Laplace family.
\end{abstract}

\begin{center}
\small
This text expands the MathOverflow answer \cite{GrechukMO}.  That answer is
licensed under CC BY-SA 4.0; this derivative draft is distributed under the
same license.  It is an expository draft, not an author-approved manuscript.
\end{center}

\section{The result}

A Borel probability measure on $\R$ is called log-concave if it has a density
of the form $e^{-V}$ on an interval, where $V$ is convex; point masses are
included in the customary limiting sense.  It is isotropic when its first
two moments are $0$ and $1$.

\begin{theorem}\label{thm:main}
If $X$ is an isotropic log-concave random variable on $\R$, then
\begin{equation}\label{eq:main}
  \E\bigl[(1+|X|)e^{-|X|}\bigr]\le 2\sqrt2-2.
\end{equation}
The constant is sharp.  It is attained by the probability density
\begin{equation}\label{eq:laplace-density}
  f_*(x)=\frac1{\sqrt2}e^{-\sqrt2|x|},\qquad x\in\R.
\end{equation}
\end{theorem}

The expectation at \eqref{eq:laplace-density} is
\begin{align*}
  2\int_0^\infty (1+x)e^{-x}\frac1{\sqrt2}e^{-\sqrt2x}\dd x
  &=\sqrt2\left(\frac1{1+\sqrt2}
      +\frac1{(1+\sqrt2)^2}\right)\\
  &=2\sqrt2-2.
\end{align*}
The rest of the paper proves that no other isotropic log-concave law gives a
larger value.

The problem arose in connection with stochastic separation; see
\cite{GGT23} for that background.  The proof below is independent of the
high-dimensional application.

\section{Localization to two log-affine pieces}

Let $\Pp(\R)$ be the space of Borel probability measures on $\R$, with the
topology of weak convergence, and let $\Lc$ denote its isotropic log-concave
members.  Define $\Lc_2$ to consist of all $\mu\in\Lc$ whose density can be
written
\begin{equation}\label{eq:two-affine-potential}
 f(x)=\exp\bigl(-\max\{\ell_1(x),\ell_2(x)\}\bigr)
       \one_{[A,B]}(x),
 \qquad -\infty\le A<B\le\infty,
\end{equation}
where $\ell_1,\ell_2$ are affine and may coincide.  Write $\Lc_2^*$ for the
subclass with both endpoints finite.

The next proposition is the precise localization statement needed here.  It
is a short consequence of Fradelizi and Gu\'edon's generalized localization
theorem \cite{FG06}; their earlier one-constraint result is \cite{FG04}.

\begin{proposition}[Two-constraint localization]\label{prop:localization}
Let $\Phi:\Pp(\R)\to\R$ be convex and weakly continuous.  Then
\begin{equation}\label{eq:localization}
 \sup_{\mu\in\Lc}\Phi(\mu)
 =\sup_{\mu\in\Lc_2}\Phi(\mu)
 =\sup_{\mu\in\Lc_2^*}\Phi(\mu).
\end{equation}
\end{proposition}

\begin{proof}
First suppose that $\mu\in\Lc$ is supported by a compact interval $K$.
For a probability measure $\rho$ on $K$, put
\[
 A_1(\rho)=\int x\dd\rho(x),
 \qquad
 A_2(\rho)=\int(x^2-1)\dd\rho(x).
\]
Let $P_+$ be the set of log-concave probabilities on $K$ satisfying
$A_1,A_2\ge0$, and set
\[
 C=\overline{\conv}P_+,
 \qquad
 H=C\cap\{A_1=A_2=0\}.
\]
The two functionals are nonnegative on $C$, so $H$ is a compact face of
$C$.  By Bauer's maximum principle \cite{Bauer58}, the restriction of
$\Phi$ to $H$ has an extreme maximizer $\nu$.  Since $H$ is a face,
$\nu\in\Ext C$.

The extreme-point inclusion established in the proof of
\cite[Theorem 2]{FG06}, followed by \cite[Theorem 1 and Corollary 1(A)]{FG06},
shows that $\nu$ is either a point mass or has a potential which is the
maximum of at most two affine functions.  A point mass cannot belong to
$H$: $A_1(\delta_x)=0$ gives $x=0$, whereas $A_2(\delta_0)=-1$.  Hence
$\nu\in\Lc_2^*$, and
\[
  \Phi(\mu)\le \Phi(\nu)\le
  \sup_{\rho\in\Lc_2^*}\Phi(\rho).
\]

For a general $X\sim\mu\in\Lc$, condition on $|X|\le R$, recenter, and
rescale to variance one.  For all sufficiently large $R$ this gives a
compactly supported isotropic log-concave law $T_R\mu$.  Finite second moment
implies that the truncated means and standard deviations tend to $0$ and
$1$, respectively; consequently $T_R\mu\Rightarrow\mu$.  Weak continuity of
$\Phi$ completes the proof.  The reverse inequalities in
\eqref{eq:localization} follow from the inclusions
$\Lc_2^*\subset\Lc_2\subset\Lc$.
\end{proof}

For the function
\begin{equation}\label{eq:g-alpha}
 g_\alpha(x)=(1+\alpha|x|)e^{-\alpha|x|},\qquad \alpha>0,
\end{equation}
the functional $\Phi_\alpha(\mu)=\int g_\alpha\dd\mu$ is linear and weakly
continuous, since $g_\alpha$ is bounded and continuous.  Thus the main
theorem reduces to a finite-dimensional assertion about $\Lc_2$.

\section{The two-piece family}

After changing the affine parameters, every density in $\Lc_2$ has the form
\begin{equation}\label{eq:hinge-density}
 f(x)=\exp\{c+\beta x-\delta(x-m)_+\}\one_{[A,B]}(x),
 \qquad \delta\ge0,
\end{equation}
with extended endpoints $A<m<B$.  The parameter $m$ is redundant when
$\delta=0$.

We shall use the following finite-dimensional result.  Its slightly more
general $\alpha$-form makes the endpoint calculation regular; only
$\alpha=1$ is needed for Theorem \ref{thm:main}.

\begin{proposition}[Sharp bound in the two-piece family]
\label{prop:family}
For $0<\alpha\le1$, define
\begin{align}
 S_\alpha
 &=e^{-1}\left\{
   \int_0^1(1+\alpha t)e^{(1-\alpha)t}\dd t
   +\frac{1+2\alpha}{(1+\alpha)^2}\right\},
   \label{eq:Salpha}\\
 L_\alpha
 &=\frac{2(1+\sqrt2\,\alpha)}{(\sqrt2+\alpha)^2}.
 \label{eq:Lalpha}
\end{align}
Then
\begin{equation}\label{eq:family-sharp}
 \sup_{\mu\in\Lc_2}\int g_\alpha\dd\mu
 =\max\{S_\alpha,L_\alpha\}.
\end{equation}
Both displayed endpoint values are realized: $S_\alpha$ by a centered
one-sided exponential law and $L_\alpha$ by the isotropic symmetric Laplace
law.
\end{proposition}

The proof occupies the next three sections.  Its organization is useful in
its own right:
\begin{enumerate}
\item the objective attains a maximum in the extended family
\eqref{eq:hinge-density};
\item a maximizer with a genuine kink cannot have a finite endpoint;
\item a compact one-piece exponential cannot be a maximizer;
\item the remaining one-sided and asymmetric-Laplace laws satisfy
\eqref{eq:family-sharp}.
\end{enumerate}

\section{Attainment by quantile-density compactification}

Let $F$ be the distribution function of a density $f$, let $Q=F^{-1}$ be
its quantile function, and set
\begin{equation}\label{eq:quantile-density}
 I(t)=f(Q(t)),\qquad 0<t<1.
\end{equation}
On an interval on which $\log f$ has slope $s$, differentiation gives
\[
 I'(t)=\frac{f'(Q(t))}{f(Q(t))}=s.
\]
Thus the quantile-density profile of \eqref{eq:hinge-density} is a
nonnegative concave function, affine on each side of at most one point
$p\in[0,1]$.

Up to a positive scalar, such a profile is described by
\begin{equation}\label{eq:profile-coordinates}
 (p,\ell,h,r)=\bigl(p,I(0),I(p),I(1)\bigr)/\max I.
\end{equation}
The parameter set is
\begin{equation}\label{eq:profile-compact}
 \mathcal C=\left\{(p,\ell,h,r)\in[0,1]^4:
 (1-p)\ell+pr\le h,
 \ \max(\ell,h,r)=1\right\}.
\end{equation}
It is compact.

\begin{lemma}[Profile reconstruction and attainment]
\label{lem:attainment}
Every density \eqref{eq:hinge-density}, modulo translation and positive
scaling, determines a point of $\mathcal C$.  Conversely, every point of
$\mathcal C$ determines, after centering and variance normalization, a law
of the form \eqref{eq:hinge-density}.  For every $\alpha\ge0$, the resulting
standardized objective is continuous on $\mathcal C$.  In particular, the
supremum over the isotropic family \eqref{eq:hinge-density} is attained.
\end{lemma}

\begin{proof}
Join $(0,\ell)$, $(p,h)$, and $(1,r)$ by affine segments, omitting a
collapsed segment if $p=0$ or $p=1$, and call the resulting profile $I_u$.
The inequality in \eqref{eq:profile-compact} is exactly concavity.  Moreover,
$I_u(t)>0$ for $0<t<1$.  Define
\begin{equation}\label{eq:raw-quantile}
 Q_u(t)=\int_{1/2}^t\frac{\dd s}{I_u(s)}.
\end{equation}
The probability law having $Q_u$ as quantile has density
$f_u(Q_u(t))=I_u(t)$.  Since $I_u$ is piecewise affine,
$\log f_u$ is piecewise affine as a function of $x=Q_u(t)$; concavity of
$I_u$ makes its two slopes nonincreasing.  A zero value at $t=0$ or $t=1$
corresponds to an infinite support endpoint.  Centering $Q_u$ and dividing
by its standard deviation yields an isotropic member of
\eqref{eq:hinge-density}.

For completeness, the required continuity has a simple common domination.
Any nonnegative concave profile with maximum one satisfies
\[
 I_u(t)\ge \min(t,1-t),\qquad 0<t<1.
\]
Hence $Q_u$ is bounded in absolute value by a constant multiple of
$1+|\log t|+|\log(1-t)|$, whose square is integrable on $(0,1)$.  Dominated
convergence therefore gives continuity of the mean and variance of $Q_u$.
The variance is positive at every $u\in\mathcal C$ and hence, by compactness,
is bounded away from zero.  Finally, $0\le g_\alpha\le1$, so another
application of dominated convergence gives continuity of the standardized
objective.  The extreme value theorem finishes the proof.
\end{proof}

\section{Variational exclusion of nonterminal maximizers}

We record two elementary zero-counting facts.  The first is the
one-dimensional sign-change lemma behind many moment arguments; compare
Karlin and Novikoff \cite{KN63} or Karlin and Studden \cite{KS66}.

\begin{lemma}[Two zeroes from two moments]\label{lem:two-zeroes}
Let $a<b$, let $w>0$ be continuous on $[a,b]$, and let $R$ be continuous and
not identically zero.  If
\[
 \int_a^b R(x)w(x)\dd x=0,
 \qquad
 \int_a^b xR(x)w(x)\dd x=0,
\]
then $R$ has at least two distinct zeroes in $(a,b)$.
\end{lemma}

\begin{proof}
The first identity forces at least one zero, say $c$.  If this were the only
zero, continuity would make $(x-c)R(x)$ have a constant weak sign, and a
strict sign on a nonempty interval.  But
\[
 \int_a^b(x-c)R(x)w(x)\dd x=0,
\]
a contradiction.
\end{proof}

\begin{lemma}[Zero count for the residual]\label{lem:residual-zero-count}
Let $0<\alpha$ and
\[
 R(x)=g_\alpha(x)-\lambda_0-\lambda_1x-\lambda_2x^2.
\]
Then $R$ has at most four distinct zeroes.
\end{lemma}

\begin{proof}
By Rolle's theorem it is enough to show that $R'$ has at most three
distinct zeroes.  Write
\[
 H(x)=-x\bigl(\alpha^2e^{-\alpha|x|}+2\lambda_2\bigr),
 \qquad R'(x)=H(x)-\lambda_1.
\]
The function $H$ is odd.  On the positive half-line, its shape is governed
by
\[
 P(t)=t\bigl(\alpha^2e^{-\alpha t}+2\lambda_2\bigr),
 \qquad
 P'(t)=2\lambda_2-alpha^2e^{-\alpha t}(\alpha t-1).
\]
The last exponential factor increases up to $t=2/\alpha$ and decreases
afterward.  If $2\lambda_2\ge0$, a horizontal line can meet the positive
profile in at most three points, and the sign of the level excludes roots
on the opposite half-line whenever all three occur.  If
$2\lambda_2<0$, a positive level gives at most two positive roots, while a
nonpositive level gives at most one; oddness gives the complementary count
on the negative half-line.  In every case a horizontal line meets $H$ at
most three times.  This proves the assertion.
\end{proof}

The next lemma contains the two variational exclusions.  The details are
included because they are the main structural part of the calculation.

\begin{lemma}[Classification of an attained maximizer]
\label{lem:classification}
Fix $\alpha>0$, and suppose an isotropic density of the form
\eqref{eq:hinge-density} maximizes $\int g_\alpha f$ over that family.
Then it is either
\begin{enumerate}
\item a one-piece exponential density on a half-line, or
\item a two-piece exponential density on the whole line.
\end{enumerate}
\end{lemma}

\begin{proof}
We divide the proof into two cases.

\smallskip
\noindent\emph{A genuine kink with a finite endpoint.}
Suppose $\delta>0$ and at least one of $A,B$ is finite.  Apply the ordinary
finite-dimensional Lagrange multiplier theorem to the mass, mean, and
second-moment constraints.  The moment Jacobian has rank three: a linear
dependence would produce a quadratic polynomial vanishing on a nondegenerate
interval.  Therefore there are multipliers $\lambda_0,\lambda_1,\lambda_2$
such that the residual
\[
 R(x)=g_\alpha(x)-\lambda_0-\lambda_1x-\lambda_2x^2
\]
satisfies the stationarity equations.  Variations in $c,\beta,\delta,m$
give
\begin{align}
 \int_A^m Rf\dd x&=0,&
 \int_A^m xRf\dd x&=0,\label{eq:left-residual-moments}\\
 \int_m^B Rf\dd x&=0,&
 \int_m^B xRf\dd x&=0.\label{eq:right-residual-moments}
\end{align}
Indeed, the $c$- and $\beta$-variations first give the two full-support
identities; the $m$-variation gives the right zeroth identity, and the
$\delta$-variation gives its first-moment counterpart.  Variation of a
finite support endpoint gives $R(A)=0$ or $R(B)=0$, as appropriate.

Lemma \ref{lem:two-zeroes} supplies two zeroes on each side of $m$.
Together with the finite endpoint this gives five zeroes of $R$, contrary
to Lemma \ref{lem:residual-zero-count}.  The same calculation on a half-line
is justified by the exponential tail.

\smallskip
\noindent\emph{A compact one-piece exponential.}
Now suppose $\delta=0$ and $A,B$ are finite.  Write
$w(x)=e^{c+\beta x}$.  The endpoint and $c,\beta$ variations yield
\begin{equation}\label{eq:one-piece-stationarity}
 R(A)=R(B)=0,\qquad
 \int_A^B Rw\dd x=\int_A^B xRw\dd x=0.
\end{equation}
Turning on a hinge of nonnegative size at any $s\in[A,B]$ gives the
one-sided stationarity inequality
\begin{equation}\label{eq:hinge-Q}
 Q(s):=\int_s^B(x-s)R(x)w(x)\dd x\ge0.
\end{equation}

Set
\begin{equation}\label{eq:gamma}
 \gamma=\int_A^B x^2e^{-\alpha|x|}w(x)\dd x.
\end{equation}
Since $w$ is an isotropic nondegenerate density, $0<\gamma<1$.  Integration
by parts in \eqref{eq:one-piece-stationarity} gives
\begin{equation}\label{eq:lambda2}
 2\lambda_2=-\alpha^2\gamma.
\end{equation}
Consequently
\begin{align}
 R'(x)&=\alpha^2x\bigl(\gamma-e^{-\alpha|x|}\bigr)-\lambda_1=:S(x),
 \label{eq:Rprime-S}\\
 S'(x)&=\alpha^2\left\{
 \gamma+(\alpha|x|-1)e^{-\alpha|x|}\right\}.
 \label{eq:Sprime}
\end{align}
The function $u\mapsto(1-u)e^{-u}$ decreases from $1$ to $0$ on $[0,1]$.
Thus there is $t>0$ such that $S$ is increasing on
$(-\infty,-t]$, decreasing on $[-t,t]$, and increasing on $[t,\infty)$.

On the other hand,
\[
 Q'(s)=-\int_s^B Rw\dd x,\qquad
 Q''(s)=R(s)w(s),\qquad
 Q'''(s)=(Rw)'(s).
\]
At $A$, equations \eqref{eq:one-piece-stationarity} give
$Q(A)=Q'(A)=Q''(A)=0$; at $B$ the same three equalities hold directly.
Since $Q\ge0$ on $[A,B]$, it follows that $Q'''(A)\ge0$ and
$Q'''(B)\le0$.  Using $R(A)=R(B)=0$ gives
\begin{equation}\label{eq:endpoint-slopes}
 R'(A)\ge0,\qquad R'(B)\le0.
\end{equation}

The increase-decrease-increase shape of $R'=S$, together with
\eqref{eq:endpoint-slopes} and $R(A)=R(B)=0$, implies
$R\ge0$ on $[A,B]$.  One quick proof is to assume $R(x)<0$ at an interior
point and apply the mean value theorem on $[A,x]$ and $[x,B]$; it would force
a negative value of $S$ to the left of a positive value in a way incompatible
with its three monotonicity intervals and the endpoint signs.  The first
identity in \eqref{eq:one-piece-stationarity} then forces $R\equiv0$.  This
is impossible because $g_\alpha$ is not a quadratic polynomial on any
nondegenerate interval.

The excluded cases leave exactly the two alternatives stated in the lemma.
\end{proof}

\section{The terminal calculation}

The one-piece half-line law is immediate.  Up to reflection, isotropy forces
\[
 X=Y-1,\qquad Y\sim\operatorname{Exp}(1).
\]
Splitting the integral at $Y=1$ gives \eqref{eq:Salpha}.  In particular,
\begin{equation}\label{eq:S1}
 S_1=\int_0^\infty(1+|y-1|)e^{-|y-1|}e^{-y}\dd y
 =\frac9{4e}.
\end{equation}

It remains to handle a whole-line two-piece exponential.  Every such law can
be written as
\begin{equation}\label{eq:asymmetric-laplace-mixture}
 X=\begin{cases}
 a-b-aT,&\text{with probability }a/(a+b),\\
 a-b+bT,&\text{with probability }b/(a+b),
 \end{cases}
 \qquad T\sim\operatorname{Exp}(1),
\end{equation}
where $a,b>0$.  It is centered as written, and
\begin{equation}\label{eq:asym-variance}
 \Var X=a^2+b^2.
\end{equation}
Thus isotropy is equivalent to $a^2+b^2=1$.

\begin{lemma}[Asymmetric-Laplace endpoint inequality]
\label{lem:asym-laplace}
Let $0<\alpha\le1$, let $a,b\ge0$, and suppose $a^2+b^2=1$.  For the law
\eqref{eq:asymmetric-laplace-mixture}, with the evident limiting
interpretation if $a=0$ or $b=0$,
\begin{equation}\label{eq:asym-bound}
 \E g_\alpha(X)\le\max\{S_\alpha,L_\alpha\}.
\end{equation}
\end{lemma}

\begin{proof}
Symmetry in $a,b$ allows us to assume $a\ge b$.  Put
\[
 p=\frac ba\in[0,1],\qquad
 a=\frac1{\sqrt{1+p^2}},\qquad
 z=\alpha a,\qquad q=1-p.
\]
For $0<\alpha<1$, direct integration in
\eqref{eq:asymmetric-laplace-mixture} gives
\begin{equation}\label{eq:K0}
 K_0(p):=\E g_\alpha(X)
 =e^{-zq}A(p,z)+e^{-q}B(p,z),
\end{equation}
where
\begin{align}
 A(p,z)&=\frac{1-z-z^2+z(1-z-z^2)p-z^2(1-z)p^2}
 {(1-z)^2(1+zp)^2},\label{eq:A}\\
 B(p,z)&=\frac{4z^3}{(1+p)(1-z^2)^2}.
 \label{eq:B}
\end{align}
The apparent singularities are removable at the limiting parameters.  The
endpoint values are
\begin{equation}\label{eq:K-endpoints}
 K_0(0)=S_\alpha,\qquad K_0(1)=L_\alpha.
\end{equation}

Here is a compact calculus certificate that no interior parameter exceeds
both endpoints.  Let
\[
 c=1+p^2,\qquad u=\frac{(1-p)^2}{c},\qquad
 M=-\frac{c^2}{2(1-p^2)},\qquad
 \mathcal D=\partial_p-\frac{pz}{c}\partial_z.
\]
Starting with $C_0=A$ and $D_0=B$, define recursively
\begin{align}
 C_{j+1}&=M\left(\mathcal DC_j+\frac{z(1+p)}cC_j\right),
 &D_{j+1}&=M(\mathcal DD_j+D_j),\label{eq:recurrence}\\
 K_j(p)&=e^{-zq}C_j(p,z)+e^{-q}D_j(p,z),
 &&z=\frac\alpha{\sqrt{1+p^2}}.\notag
\end{align}
Since $\dd u/\dd p=-2(1-p^2)/c^2$, differentiation gives
\begin{equation}\label{eq:derivative-recurrence}
 \frac{\dd K_j}{\dd u}=K_{j+1},\qquad j=0,1,2.
\end{equation}

We next record the sign certificate for the third derivative.  Put
$v=(1-p)(1-z)$ and
\[
 T_3=1-v+\frac{v^2}2-\frac{v^3}6,
 \qquad T_4=T_3+\frac{v^4}{24}.
\]
For $0<p<1$ and $0<z<1$, elementary expansion of the recurrence gives
\begin{align}
 C_3+T_3D_3
 &=-\frac{z^3c^3Q_3(p,z)}
 {24(1-p)(1+p)^7(1-z)(1+z)^5(1+pz)^5},
 \label{eq:Q3}\\
 C_3+T_4D_3
 &=-\frac{z^3c^3Q_4(p,z)}
 {48(1-p)(1+p)^7(1+z)^5(1+pz)^5}.
 \label{eq:Q4}
\end{align}
The numerator polynomials are positive on the unit square.  A convenient
fully algebraic check is the M\"obius substitution
\[
 p=\frac{x}{1+x},\qquad z=\frac{y}{1+y}.
\]
After clearing denominators,
\begin{align*}
 (1+x)^{15}(1+y)^{10}Q_3\left(\frac{x}{1+x},\frac{y}{1+y}\right)
 &=P_3(x,y),\\
 (1+x)^{16}(1+y)^{10}Q_4\left(\frac{x}{1+x},\frac{y}{1+y}\right)
 &=P_4(x,y),
\end{align*}
where direct collection from \eqref{eq:recurrence} shows that every
coefficient of $P_3$ and $P_4$ is nonnegative, while their constant
coefficients are $4$ and $6$, respectively.  Hence $Q_3,Q_4>0$.

The alternating Taylor bounds
\[
 T_3<e^{-v}<T_4\qquad(v>0)
\]
and \eqref{eq:Q3}--\eqref{eq:Q4}, using $T_3$ if $D_3\le0$ and $T_4$ if
$D_3\ge0$, now give
\[
 K_3=e^{-zq}\bigl(C_3+e^{-v}D_3\bigr)<0.
\]
Thus $K_1=\dd K_0/\dd u$ is strictly concave.  Direct substitution in the
recurrence gives $K_1(0)=0$, where $p=0$ corresponds to $u=1$.  If $K_0$
were larger at an interior point than at both endpoints, the mean value
theorem would give a positive value of $K_1$ to one side and a negative
value to the other.  Strict concavity of $K_1$, together with its zero value
at $u=1$, makes this impossible.  Therefore
\[
 K_0(p)\le\max\{K_0(0),K_0(1)\},
\]
and \eqref{eq:K-endpoints} proves \eqref{eq:asym-bound} for
$0<\alpha<1$.

Finally let $\alpha\uparrow1$.  Since $0\le g_\alpha\le1$ and
$g_\alpha\to g_1$ pointwise, dominated convergence applies to the genuine
expectation in \eqref{eq:asymmetric-laplace-mixture}.  The endpoint values
in \eqref{eq:Salpha}--\eqref{eq:Lalpha} are continuous at $1$, so the same
inequality holds for $\alpha=1$.
\end{proof}

\begin{proof}[Proof of Proposition \ref{prop:family}]
Lemma \ref{lem:attainment} supplies a maximizer.  Lemma
\ref{lem:classification} says that it is one-sided exponential or
whole-line asymmetric Laplace.  The former has value $S_\alpha$; the latter
is bounded by Lemma \ref{lem:asym-laplace}.  Conversely, the centered
one-sided exponential realizes $S_\alpha$, and the symmetric choice
$a=b=1/\sqrt2$ in \eqref{eq:asymmetric-laplace-mixture} realizes
$L_\alpha$.  This proves \eqref{eq:family-sharp}.
\end{proof}

\section{Completion of the proof}

At $\alpha=1$, equations \eqref{eq:Salpha} and \eqref{eq:Lalpha} give
\[
 S_1=\frac9{4e},
 \qquad
 L_1=2\sqrt2-2.
\]
Moreover $S_1<L_1$.  For an entirely elementary exact check, use
$e>\sum_{k=0}^6 1/k!=1957/720$ and square the resulting positive rational
inequality.

Proposition \ref{prop:localization}, applied to
$\Phi_1(\mu)=\int g_1\dd\mu$, and Proposition \ref{prop:family} therefore
give
\[
 \sup_{\mu\in\Lc}\int(1+|x|)e^{-|x|}\dd\mu(x)
 =\sup_{\mu\in\Lc_2}\int g_1\dd\mu
 =2\sqrt2-2.
\]
The density \eqref{eq:laplace-density} supplies equality.  This proves
Theorem \ref{thm:main}.

\begin{remark}
The proof deliberately separates the conceptual localization step from the
finite-dimensional calculus.  The former applies to every convex weakly
continuous functional on isotropic log-concave measures.  The latter uses
the special derivative structure of $g_\alpha$ and is where sharpness enters.
\end{remark}